\documentclass[12 pt]{article}
\usepackage[english]{babel}
\usepackage{graphicx}
\usepackage[active]{srcltx}
\usepackage{amsmath}
\usepackage{amsthm}
\usepackage{amsfonts,amssymb}
\usepackage[mathcal]{eucal}
\theoremstyle{plain}
\newtheorem{theorem}{Theorem}
\newtheorem{corollary}{Corollary}
\newtheorem{lemma}{Lemma}
\newtheorem{proposition}{Proposition}
\theoremstyle{definition}
\newtheorem{definition}{Definition}
\theoremstyle{remark}
\newtheorem{example}{Example}
\newtheorem{remark}{Remark}
\usepackage{hyperref}
\newcommand{\rav}{\stackrel{\triangle}{=}}
\newcommand{\ravref}[1]{\stackrel{(\ref{#1})}{=}}

\newcommand{\epsi}{\varepsilon}
\newcommand{\rtw}{\mathrm{w}}
\newcommand{\rty}{\mathrm{y}}

\newcommand{\hjk}[1]{{\hat{#1}}}
\newcommand{\newcorr}[1]{{#1}}
\newcommand{\corr}[1]{{#1}}
\usepackage{amsopn}
\newcommand{\doc}{{\em{Proof\ }}}
\newcommand{\bo}{\hfill {$\Box$}}

\newcommand{\lscinf}[1]{\mathop{{#1}\textrm{-inf}}}
\newcommand{\lscsup}[1]{\mathop{{#1}\textrm{-sup}}}
\newcommand{\lscextr}[1]{\mathop{{#1}\textrm{-extr}}}
\DeclareMathOperator*{\meas}{meas}
\DeclareMathOperator*{\dist}{dist}
\DeclareMathOperator*{\as}{\textrm{home}}
\DeclareMathOperator*{\Limsup}{Limsup}
\DeclareMathOperator*{\Liminf}{Liminf}
\DeclareMathOperator*{\elim}{e-lim}
\DeclareMathOperator*{\elimsup}{e-limsup}
\DeclareMathOperator*{\eliminf}{e-liminf}
\DeclareMathOperator*{\epi}{epi}
\DeclareMathOperator*{\co}{co}
\DeclareMathOperator*{\cl}{cl}

\DeclareMathOperator*{\lsc}{lsc}
\DeclareMathOperator*{\brd}{brd}
\DeclareMathOperator*{\internary}{int}

\DeclareMathOperator*{\sign}{sign}
\DeclareMathOperator*{\Gr}{gph}

\begin{document}
	
	\title{Necessary Conditions\\ In Infinite-Horizon Control Problems\\ That Need No Asymptotic Assumptions\thanks{Krasovskii Institute of Mathematics and Mechanics, 16 S.Kovalevskaja St., Yekaterinburg, Russia, {\it khlopin@imm.uran.ru}}
	}

	
	\author{Dmitry Khlopin
	}\maketitle
\begin{abstract}
		We consider an infinite-horizon optimal control problem with an asymptotic terminal constraint. 
		For the the weakly overtaking criterion and  the overtaking criterion,  necessary boundary conditions on co-state arcs are deduced, these  conditions need no assumptions about the asymptotic behavior of the motion, co-state arc, cost functional, and its derivatives. In the absence of an asymptotic terminal constraint,  these boundary  conditions with the Pontryagin Maximum Principle allow raising the  co-state arcs, corresponding to some asymptotic  subdifferentials of the cost functional (fixing the optimal control) at infinity. If this set is a singleton, 
		these conditions coincide with the co-state arc representation proposed by  Aseev and Kryazhimskii.
		These results are illustrated by several examples.



	\end{abstract}

	\section{Introduction}

	We will consider   an infinite-horizon optimal control problem,
	\begin{align*}
	\textrm{minimize\ }	&l(y(0))+\int_{0}^\infty f_0(\tau,y(\tau),u(\tau))\, d\tau\\ 
	\textrm{subject to\ }	&\frac{d{y}(t)}{dt}=f(t,y(t),u(t))\quad \textrm{a.e.},  \\
	&y(t)\in\mathbb{R}^m,\ u(t)\in U,\ y(0)\in \mathcal{C}_0,\ \Limsup_{\theta\uparrow\infty} \{\Lambda(\theta,y(\theta))\}\subset \mathcal{C}_\infty
	\end{align*}
	with      the relations of the Pontryagin Maximum Principle   corresponding to this problem:
	\begin{subequations}
		\begin{eqnarray}
		\label{sys_x}
		\frac{d{y}(t)}{dt}&=& f\big(t,y(t),{u}(t)\big),\\
		\label{sys_psi}
		-\frac{d{\psi}(t)}{dt}&=&\frac{\partial
			{H}}{\partial x}\big(y(t),\psi(t),{u}(t),\lambda,t\big),\\
		\label{maxH}
		\sup_{\upsilon\in
	U}H\big(y(t),\psi(t),\upsilon,\lambda,t\big)
		&=&
		H\big(y(t),\psi(t),{u}(t),\lambda,t\big).
		\end{eqnarray}
	\end{subequations}
	Here,
	${\mathbb{R}}_{+}\rav [0;\infty)$
	is the time interval of the control system, the sets $\mathcal{C}_0$ and $\mathcal{C}_\infty$ are nonempty subsets of $\mathbb{X}\rav\mathbb{R}^m$, $l$ and $\Lambda$ is a scalar function on $\mathbb{X}$, the set $U$ of  control parameters  is a nonempty subset of a certain finite-dimensional real vector space, and
	the Hamilton--Pontryagin function
	$H:{\mathbb{X}}\times{\mathbb{X}}^*\times {U}\times{\mathbb{R}_+}\times{\mathbb{R}_+}\to{\mathbb{R}}$
	is given by
	\[
	H(x,\psi,u,\lambda,t)\rav\psi f\big(t,x,u\big)-\lambda
	f_0\big(t,x,u\big)\quad
	\forall (x,\psi,u,\lambda,t)\in{\mathbb{X}}\times{\mathbb{X}}^*\times {U}\times{\mathbb{R}_+}\times{\mathbb{R}_+}.
	\]
	
	It is well known that the relations
	\eqref{sys_x}--\eqref{maxH} are  necessary conditions for finite-horizon control problems \cite{ppp}. 
	For infinite-horizon control problems, the Pontryagin Maximum Principle was proved in 
	the pioneering paper~\cite{Halkin}   in the case of the 
	finite optimality criterion  (the optimality on each finite interval of the corresponding problem with fixed ends). 
	In this sense, each optimal  process  as an extension of optimal  processes for the finite horizon problems admits the extension of \eqref{sys_x}-\eqref{maxH}  from each finite interval to $\mathbb{R}_+$. 
	By passing to the limit, one obtains  the necessary conditions on the whole  $\mathbb{R}_+$, system \eqref{sys_x}--\eqref{maxH} and the usual transversality condition at zero.

	Naturally, on the one hand, this proof does not require any supplementary information on system~\eqref{sys_x}--\eqref{maxH} at infinity; on the other hand, nor is any such knowledge gained.  In
	particular,
	this system of necessary relations  lacks one more boundary condition on the co-state arc, which corresponds to the transversality condition at the right end.
	In fact, without such a condition, relations \eqref{sys_x}--\eqref{maxH} only serve to point towards the variety of their solutions without offering a tool to choose one from among them.    
	
	To limit the search, various supplementary conditions are used; otherwise, in each specific problem, 
	one could exhaustively search through all solutions~\cite{sorger}.
	The limit value at infinity for the motion itself
	can be specified~\cite[Subsect. 4.24]{ppp}.
	The solution and/or control lie in a certain class of functions 
	(see~\cite{aucl,brodskii,Tauchnitz,Tauchnitz2020}).  
	It is also possible to connect some condition with  the value function \cite{sagara,khlopin2015lipschitz,cannarsa2018value}.
	One can try to apply some relations due to economic  reasons \cite{shell} and ones of convexity \cite{ssbook} and stability \cite{tan_rugh,smirn}.    
	Nevertheless, some supplementary information (in the form of boundary conditions) 
	can be also tried to reclaim 
	from a certain optimality criterion.
	In this paper, we obtain such conditions for rather mild optimality criteria such as the weakly overtaking criterion and the overtaking criterion (the optimality for  the upper and lower pointwise limits of the cost functional, respectively).
	
	In \cite{norv}, the necessary conditions, including a complete set of transversality conditions
	for the co-state arc at infinity, were considered.  In the case of a final state dependent cost
	functional, as well as of an asymptotic terminal constraint of 
	a linear structure, these conditions were proved under very strong asymptotic assumptions on motions and costs. In \cite{Pereira,Tauchnitz2020},  a similar result was shown under some asymptotic assumptions, guaranteeing the existence of the limit of motions. 
	In~\cite{KhlopinIMM2018}, 
	the necessary transversality condition was obtained as a consequence of the stability of the limiting subdifferentials with respect to the uniform convergence. 
	In this paper, this approach is applied to the derivation of  necessary conditions based on the stability of subdifferentials. It makes it possible, while continuing to follow Halkin's method, to reformulate  the overtaking criterion in these terms, and pass to the limit within the necessary conditions for specially selected problems for the increasing sequence of time intervals.
	
	Adhering to this approach, we apply the classical results \cite{slope} with respect to the Fr\'echet subdifferential for the epigraphical convergence and refine  the upper-estimates \cite[Theorem~6.2]{trieman} of the  limiting subdifferential of the pointwise lower and upper limit (see Lemmata~\ref{cap} and \ref{co}). Further, we apply these estimates	directly to the transversality condition of
	the classical Pontryagin Maximum Principle  \cite{ClarkeNew,vinter} for some Bolza 
	problems  for the increasing sequence of time intervals. Thus, we
	establish the necessary conditions for the weakly overtaking and overtaking criteria (see Theorems~\ref{9} and~\ref{8}), which notably require no asymptotic assumptions on the motion,  cost functional, adjoint variable,  and its gradients.

	A special focus of this paper is the question of the accuracy of the considered transversality conditions at infinity for the problem without asymptotic constraint. 
	If 
	the cost functional  gradients  have a limit for large time,  the proposed transversality  condition \eqref{WAKKa}  is explicit, i.e. it points to a unique solution to the adjoint  system of \eqref{sys_psi}; moreover, this solution coincides with the  co-state arc representation \cite{kr_as2004} proposed by Kryazhimskii and Aseev.
	On the other hand,  in the case  of the periodic functional,
	even for an infinite-horizon control problem linear in $x$ (see Example~\ref{1}), there is no hope to construct any explicit 
	boundary condition on  $\psi$   necessary for the weakly overtaking optimal criterion;  each
	such condition is going to contain a ball.
	At the same time, in this example, the proposed condition 
	\eqref{WAKKa} 
	is 
	the strongest  of all consistent with \eqref{maxH} boundary conditions on the co-state arc.
	
	The  rest  of  the  paper  is  organised  as  follows. 
	First, we introduce the statement of the infinite-horizon control problem, the dynamics and cost functional, and formulate the basic hypotheses on them; we also define all needed optimality criteria.
	In Section~\ref{s2}, we recall the concepts and notions of set-valued and variational analyses. Section~\ref{s4} exhibits  formulations of all theorems and corollaries and   the  discussion of their assumptions and conditions.
	The subsequent section is devoted to examples. 
	Subdifferentials of lower and upper limits of functions defined on finite-dimensional  sets are considered in Section~\ref{e}.
	The last part of the paper (Section~\ref{ff}) contains the proofs of all theorems.

	\section{The statement of the infinite-horizon control problem}
	
	Let
	$
	{\mathbb{R}}_{+}= [0;\infty)$
	be the time interval of  an investigated  control system, and let ${\mathbb{X}}={\mathbb{R}}^m$ be its state space. 
	Let    functions $l:\mathbb{X}\to\mathbb{R}$, $\Lambda:\mathbb{R}_+\times\mathbb{X}\to\mathbb{X}$ and nonempty sets  $\mathcal{C}_0\subset {\mathbb{X}}$, $\mathcal{C}_\infty\subset {\mathbb{X}}$ be given.
	Let a non-empty subset~$U$ of a finite-dimensional  real vector space also be given;
	denote by $\mathcal{U}$ the set of all admissible controls, all Lebesgue measurable functions $u:\mathbb{R}_+\to U$.
	
	Consider the following infinite-horizon optimal control problem:
	\begin{subequations}
		\begin{align}
		\textrm{minimize\ }	&l(y(0))+\int_{0}^\infty f_0(\tau,y(\tau),u(\tau))\, d\tau  \label{sys0_}\\
		\textrm{subject to\ }	&\frac{d{y}(t)}{dt}=f(t,y(t),u(t))\ \quad\textrm{a.e.}, \label{sys_}\\
		&y(t)\in\mathbb{X}=\mathbb{R}^m,\  y(0)\in\mathcal{C}_0,\ \Limsup_{\theta\uparrow\infty}\{\Lambda(\theta,y(\theta))\}\subset\mathcal{C}_\infty,\  u\in\mathcal{U},   \label{sysK_}
		\end{align}
	\end{subequations}   
	 here the symbol $\Limsup$ is defined in Section~\ref{s2}, see \eqref{limsup}.

	Hereinafter, we assume the following conditions to hold:
	\begin{list}{}{}
		\item[(H0)]    $l:\mathbb{X}\to\mathbb{R}$ is a locally Lipschitz continuous function; 
		\item[(H1)]   $f:\mathbb{R}_+\times\mathbb{X}\times U\to\mathbb{X}$ and
		$f_0:\mathbb{R}_+\times\mathbb{X}\times U\to\mathbb{R}$	
		are LB measurable in~$(t,u)$.
	\end{list}

	For every compact interval $I\subset \mathbb{R}_+$ 
	a pair
	$(y,u)\in (AC)(I,\mathbb{X})\times {\mathcal{U}}$ is called a control process
	if the map $I\ni\tau\mapsto  f_0\big(\tau,y(\tau),u(\tau)\big)$
	 is summable and $y$ is a solution to  equation
	 \eqref{sys_} on $I$, i.e., \eqref{sys_} holds for almost all $t\in I$. A pair $({y},{u})\in(AC)(\mathbb{R}_+,\mathbb{X})\times {\mathcal{U}}$ is called admissible control process 
	  if  $y$ satisfies \eqref{sysK_} and  the pair $({y}|_{[0;T]},{u})$ is a control process for every $T>0$.
	
	We will use the following optimality criteria:      
	\begin{subequations}
		\begin{definition}
			Call an admissible control process $(\hjk{y},\hjk{u})$ overtaking optimal \cite{tan_rugh,car1} for problem \eqref{sys0_}--\eqref{sysK_} if for every admissible control process $(y,u)$ it holds that
			\begin{align}
		\label{tauover}
			l(y(0))-l(\hjk{y}(0))+\liminf_{\theta\uparrow\infty}
			\int_{0}^\theta\big[f_0(\tau,y(\tau),u(\tau))-f_0(\tau,\hjk{y}(\tau),\hjk{u}(\tau))\big]\,d\tau
			\geq 0.
			\end{align}
		\end{definition}
		\begin{definition}
			Call an admissible control process $(\hjk{y},\hjk{u})$ weakly overtaking optimal \cite{tan_rugh,car1} for  problem \eqref{sys0_}--\eqref{sysK_} if for every admissible control  process $(y,u)$ it holds that
			\begin{align}
			\label{tauover_}
			l(y(0))-l(\hjk{y}(0))+\limsup_{\theta\uparrow\infty}
			\int_{0}^\theta\big[f_0(\tau,y(\tau),u(\tau))-f_0(\tau,\hjk{y}(\tau),\hjk{u}(\tau))\big]\,d\tau
			\geq 0.	
			\end{align}
		\end{definition}
		\end{subequations}      
		Clearly, an overtaking optimal process is weakly overtaking optimal, but we will relax the both criteria, considering merely its local variants \cite{av_new,Shvartsman}:
		\begin{definition}
			Call an admissible control process $(\hjk{y},\hjk{u})$ locally overtaking (locally weakly overtaking) optimal 
			for  problem \eqref{sys0_}--\eqref{sysK_} if  
			for every natural $n$  there exists a positive $\beta(n)$  such that, for each admissible control process $(y,u)$, 
			from 
			\begin{align}\label{local}
			\max_{t\in[0;n]}||y(t)-\hjk{y}(t)||+\meas \{t\in[0;n]\mid u(t)\neq\hjk{u}(t)\}<\beta(n)
			\end{align}
			and ${u|_{[n;\infty)}=\hjk{u}|_{[n;\infty)}}$
			there follows  inequality \eqref{tauover} (inequality \eqref{tauover_}).
		\end{definition}
	
	Some conditions for  the existence of weakly overtaking optimal
	and overtaking optimal processes are given in \cite{bald,dm,besov}. In this paper, the existence theory is not directly concerned.
	We also do not concern ourselves with sufficient optimality conditions (see \cite{ssbook,tan_rugh,kr_as,sagara,belyakov2020}).
	
	We assume that a certain admissible control process $(\hjk{y},\hjk{u})$ is locally
	weakly overtaking optimal for problem   \eqref{sys0_}--\eqref{sysK_}.  
	We will also assume several local assumptions on  $(\hjk{y},\hjk{u})$, hypotheses $(H2)$--$(H4)$.  
	Note that 
	from \cite[Hypothesis 22.25]{ClarkeNew} for each interval $[0;T]$ it follows  hypotheses $(H2)$--$(H3)$. Besides, the hypothesis $(H4)$ requires  merely the well-posedness of the right side of \eqref{sys_psi} on the graph of optimal process $(\hjk{y},\hjk{u})$.
	
	Hereinafter, we assume that
	\begin{list}{}{}
	\item[(H2)]   
		for each control parameter $\upsilon\in U$, there exists  a  neighbourhood $\mathbb{G}_\upsilon\subset \mathbb{R}_+\times\mathbb{X}$ of the graph of  $\hjk{y}$ and a measurable function $L_\upsilon:\mathbb{R}_+\to\mathbb{R}_+$ such that for all $(t,x')$, $(t,x'')\in \mathbb{G}_\upsilon$ one has
		\begin{equation}\label{301}
		\big\|f(t,x',\upsilon)-{f}(t,x'',\upsilon)\big\|+
		\big\|{f_0}(t,x',\upsilon)-{f_0}(t,x'',\upsilon)\big\|\leq L_\upsilon(t)\|x'-x''\|;
		\end{equation}
	\item[(H3)]   
		there exists a neighbourhood $\hjk{\mathbb{G}}\subset \mathbb{R}_+\times\mathbb{X}$ of the graph of  $\hjk{y}$ such that, for all $(t,x')$, $(t,x'')\in\hjk{\mathbb{G}}$, inequality \eqref{301} holds with $L_{\hjk{u}(t)}(t)$ and $\hjk{u}(t)$ instead of $L_\upsilon(t)$ and $\upsilon$; furthermore, 
		the map $\mathbb{R}_+\ni t\mapsto L_{\hjk{u}(t)}(t)$ is locally summable    (summable on each compact interval);
	\item[(H4)]   
for almost all nonnegative $t$
 the maps  $x\mapsto f(t,x,\hjk{u}(t))$ and $x\mapsto f_0(t,x,\hjk{u}(t))$   are  strictly differentiable (in $x$) at $\hjk{y}(t)$. 
 	 	\end{list}

     For each admissible control~$u\in {\mathcal{U}}$ and $(t,x)\in\mathbb{R}_+\times\mathbb{X}$ denote by $\rty(x,t,u;\cdot)$ 
     a solution $y$ to \eqref{sys_} with the initial condition $y(t)=x$ on an interval. We put that this interval is 
     the maximum existing interval of $y$ (in particular, this interval may be $\{t\}$).
       Regardless additional assumptions (see $(H5)$ and $(H6)$ below), this process may be non-unique; in this case  fix such a process for every $(x,t,u)$.
       Now,
      let us introduce  the cost functional~$J$ defined as follows:
     \[      	J(x,t,u;\theta)\rav\left\{ \begin{array}{ll}
	\int_{t}^{\theta} f_0\big(\tau,\rty(x,t,u;\tau),u(\tau)\big)\,d\tau, & \textrm{\ if\ }\theta\geq t\textrm{\ and\ }  (\rty(x,t,u;\cdot)|_{[t;\theta]},u)  \\
	& \textrm{\ is a control process on $[t;\theta]$};\\
	+\infty, & \textrm{\ otherwise} 
	\end{array}
	\right.\]
     For brevity, let us also introduce
     \[\hjk{J}(x;\theta)\rav \lsc J(x,0,\hjk{u};\theta)
     \qquad\forall \theta\geq 0,x\in\mathbb{X},\]
	 here the symbol $\lsc$ is defined in Section~\ref{s2}, see \eqref{lsc}. 
	 By $(H3)$, for every nonnegative $\theta$, $J(x,0,\hjk{u};\theta)=\hjk{J}(x;\theta)$
	 for all $x$ close enough to $\hjk{y}(0)$.

Later, for greater convenience, we will also consider the much stronger hypotheses:
	\begin{list}{}{}
\item[(H5)]   there exists a neighbourhood $\hjk{\mathbb{G}}_\exists\subset \mathbb{X}$ of the point  $\hjk{y}(0)$ such that for every initial condition $x\in\hjk{\mathbb{G}}_\exists$  one finds  a solution $y=\rty(x,0,\hjk{u};\cdot)$ to \eqref{sys_x}  on $\mathbb{R}$  such that the graph of $y$  belongs to $\hjk{\mathbb{G}}$;
	\item[(H6)] 
	there exists a neighbourhood $\hjk{\mathbb{G}}_\partial\subset \hjk{\mathbb{G}}_\exists$  of the point  $\hjk{y}(0)$ such that
for almost all positive $t$
 the maps  $\hjk{\mathbb{G}}_\partial\ni  x\mapsto f(t,x,\hjk{u}(t))$ and $\hjk{\mathbb{G}}_\partial\ni x\mapsto f_0(t,x,\hjk{u}(t))$   are  Fr\'{e}chet differentiable in $x$ at $x=\rty(x',0,\hjk{u};t)$ for all $x'\in\hjk{\mathbb{G}}_\partial$.	 
	\end{list}
	These hypotheses guarantee that  for every point $(t,x)$ close enough to the graph of $\hjk{y}$  the motion $\rty(x,t,\hjk{u};\cdot)\in (AC)(\mathbb{R}_+;\mathbb{X})$ is unique. Furthermore, the maps 
	${\mathbb{G}}_\exists\ni x\mapsto \rty(x,0,\hjk{u};\theta)$ and ${\mathbb{G}}_\exists\ni x\mapsto J(x,0,\hjk{u};\theta)=\hjk{J}(x;\theta)$ are finite, continuous, and  Fr\'{e}chet differentiable for every positive $\theta$.

	\section{Some definitions from set-valued and variational analyses}
	\label{s2}
	
	We will use elementary notions from the set-valued and variational analyses \cite{RW,Mordukh1,mord5}.
	
	Consider a nonempty set ${\mathcal{X}}$ of  real Euclidean space~$\mathbb{X}$.
	Let $\co {{\mathcal{X}}}$, $\cl {{\mathcal{X}}}$, $\brd \mathcal{X}$, and $\internary {{\mathcal{X}}}$ denote  the convex hull, closure, boundary, and interior of ${{\mathcal{X}}}$.
	The symbol ${\imath}_{\mathcal{X}}$ denotes the indicator function of the set ${\mathcal{X}}$; this function has value $0$ on ${\mathcal{X}}$, but $+\infty$ elsewhere. 
	Recall also that the sequential Painlev\'e--Kuratowski upper and lower limits of a set-valued map  $F:\mathbb{X}\rightrightarrows \mathbb{X}$ at a point $x\in\mathbb{X}$
	is
	\begin{multline}\label{limsup}
	\Limsup_{z\to x}F(z)\rav\bigcap_{\varkappa>0}\cl\bigcup_{||z-x||<\varkappa} F(z)\\=\{\zeta\in \mathbb{X}\mid  \exists\ \textrm{sequences of}\ z_k\to x,\ \zeta_k\to \zeta\
	\textrm{with}\ \zeta_k\in F(z_k)\  
	\textrm{for all} \ 
	k\in\mathbb{N}
	\},
	\end{multline}
	\begin{multline}\label{liminf}
	\Liminf_{z\to x}F(z)\rav\cl\bigcup_{\varkappa>0}\bigcap_{||z-x||<\varkappa} F(z)\\=\{\zeta\in \mathbb{X}\mid  \forall\ \textrm{sequence of}\ z_k\to x,\  \exists\zeta_k\in F(z_k)\  
	\textrm{for all\ } 
	k\in\mathbb{N}\
	\textrm{with}\ \zeta_k\to \zeta 
	\}.
	\end{multline}
	
	For  a point $x\in \mathbb{X}$,
	we say that  $\zeta\in \mathbb{X}^*$ is  a   Fr\'{e}chet (regular) normal  to ${{\mathcal{X}}}$ at $x$ if one has $x\in{{\mathcal{X}}}$ and
	\[\limsup_{z_n\to x} \frac{\zeta(z_n-x)}{||z_n-x||}\leq 0\]
	for all sequences of $z_n\in{{\mathcal{X}}}$  converging to $x$.
	Denote by $\hjk{N}(x;{{\mathcal{X}}})$ the set of all Fr\'{e}chet normals to ${{\mathcal{X}}}$ at $x$; 
	put
	$\hjk{N}(x;{{\mathcal{X}}})\rav\varnothing$ if $x\in \mathbb{X}\setminus{{\mathcal{X}}}$. 
	The sequential Painlev\'e--Kuratowski upper limit of $\hjk{N}(z;{{\mathcal{X}}})$ as $z\to x$ is the set
	\[{N}(x;{{\mathcal{X}}})\rav\Limsup_{z\to x} \hjk{N}(z;{{\mathcal{X}}}),\]
	which is called the limiting (basic, Mordukhovich) normal cone to ${{\mathcal{X}}}$ at $x$.

	Consider an extended-real-valued  function $g:\mathbb{X}\to\mathbb{R}\cup\{-\infty,+\infty\}.$
	Define its epigraph $\epi g \rav\{(x,a)\in \mathbb{X}\times\mathbb{R}\mid a\geq g(x)\}$
	and its graph $\Gr g \rav\{(x,g(x))\in \mathbb{X}\times\mathbb{R}\mid x\in\mathbb{X}\}$.
	Also, recall that
	the lower semicontinuous envelope of the function $g$ is defined as follows:
	\begin{equation}
	    \label{lsc}\lsc g(x)\rav\lim_{\varkappa\downarrow 0}\inf_{||x-z||<\varkappa} g(z)\qquad\forall x\in\mathbb{X}.
	\end{equation}
	Note that this function is  lower semicontinuous.

		For a point $x\in \mathbb{X}$ with $|g(x)|<+\infty$,
		define the limiting (basic, Mordukhovich) subdifferential \cite[Definition 1.77(i)]{Mordukh1} of $g$ at $x$ as
		\begin{align*}
		\partial g(x)\rav\{\zeta\in \mathbb{X}^*\mid (\zeta,-1)\in N(x,g(x); \epi g) \},
		\end{align*}
		the singular limiting subdifferential \cite[Definition 1.77(ii)]{Mordukh1} of $g$ at $x$ as
		\begin{align*}
		\partial^\infty g(x)\rav\{\zeta\in \mathbb{X}^*\mid (\zeta,0)\in N(x,g(x); \epi g) \},
		\end{align*}
		and the Fr\'{e}chet (regular) subdifferential  \cite[(1.36)]{mord5} of $g$ at $x$ as
		\begin{align*}
		\hjk{\partial} g(x)\rav\{\zeta\in \mathbb{X}^*\mid (\zeta,-1)\in \hjk{N}(x,g(x) ; \epi g) \}.
		\end{align*}   
		Put $\partial g(x)=\partial^\infty g(x)=\hjk{\partial} g(x)=\varnothing$  if $|g(x)|=\infty$.

		Note that, since $\mathbb{X}$ is finite-dimensional, for a lower semicontinuous around $x$ function $g$, according to \cite[(1.37) and (1.38)]{mord5},  
		a point $\zeta$ in $\mathbb{X}^*$ lies in $\partial g(x)$   iff
		\begin{multline*}
		\textrm{one finds sequences of }  x_n\in\mathbb{X},\zeta_n\in \hjk{\partial} g(x_n)\ 
		\textrm{ satisfying }x_n\to x,
		\zeta_n\to\zeta, g(x_n)\to g(x);
		\end{multline*}
		furthermore,
		a point $\zeta$ in $\mathbb{X}^*$ lies  in
		$\partial^\infty g(x)$  iff
		\begin{multline*}
		\textrm{one finds sequences of }  \lambda_n>0,x_n\in\mathbb{X},\zeta_n\in \hjk{\partial} g(x_n)\\ \textrm{ satisfying } \lambda_n\downarrow 0,x_n\to x,
		\lambda_n\zeta_n\to\zeta, g(x_n)\to g(x).
		\end{multline*}
	Also, in the case of 
	Lipschitz continuous function $g$,  one has $\partial^\infty g(x)=\{0\}$, although $\partial g(x)$ is not empty and is bounded \cite[Corollary~1.81]{Mordukh1}; in addition,
	$\hjk{N}(x,g(x);\epi g)=\cup_{r\geq 0}r(\hjk{\partial}g(x) \times\{-1\})$ and  $\co\partial g(x)=\co(-\partial (-g))(x)$ 
	hold by \cite[Theorem~8.9]{RW} and
	by \cite[(1.75) and (1.83)]{mord5}), respectively.  Furthermore,  
	$\co\partial g(x)=\{\frac{d g}{dx}(x)\}$ iff $g$ is 
	strictly differentiable at $x$ (see \cite[Ex. 5.2.4]{BorZhu}). 
	At last, notice that,
	 for every  set $A$ and point $x\in A$, due to \cite[(1.43)]{mord5}, one has 
\begin{equation}\label{133}
\begin{array}{rr}
\hjk\partial \dist(x;A)=\hjk{N}(x;A)\cap\{\zeta\in\mathbb{X}^*\,|\,\|\zeta\|\leq 1\},\\  
\partial \dist(x;A)=N(x;A)\cap\{\zeta\in\mathbb{X}^*\,|\,\|\zeta\|\leq 1\}.
	\end{array}
\end{equation}

	\section{The main results}
	\label{s4}
	
	In this section, we will formulate  and  discuss the main results of the article. 
	
	First,
	consider the homeward set for all generated by $\hjk{u}$ motions that passed the asymptotic constraint
	$\mathcal{C}_{\infty}$, which is the set
	\[\mathcal{C}_{\as}\rav\big\{x\in\mathbb{X}\,\big|\,
	\Limsup_{\theta\uparrow \infty} \{\Lambda(\theta,\rty(x,0,\hjk{u};\theta))\}\subset \mathcal{C}_{\infty}
	\textrm{\;and\;}{J}(x,0,\hjk{u};t)<\infty\textrm{\;}\forall t\geq 0
	\big\}.\]
	This set  will be used below instead of $\mathcal{C}_{\infty}$, because the  transversality conditions at infinity will also transfer at zero.
	
\newcorr{
We also consider the following assumptions:}
	\begin{list}{}{}
	\item[\newcorr{$(E_{\sup})$}] \newcorr{For this function}
$$\newcorr{W_{\sup}(x)\rav\displaystyle \limsup_{\substack{\theta\uparrow\infty\\ \ }} \big[\hjk{J}(x;\theta)-\hjk{J}(\hjk{y}(0);\theta)\big]}
$$ 
 \newcorr{at least one of the following conditions is satisfied:}
  \begin{list}{}{}
  \item[\newcorr{$(E'_{\sup})$}] \newcorr{$W_{\sup}$ is lower semicontinuous at $\hjk{y}(0)$;}
  \item[\newcorr{$(E''_{\sup})$}]  \newcorr{$\hjk{y}(0)$ lies in the interior of $\mathcal{C}_0$;}
  \item[\newcorr{$(E'''_{\sup})$}]
     \newcorr{$W_{\sup}+i_{\mathcal{C}_{\as}}$ is lower semicontinuous at $\hjk{y}(0)$.}
 \end{list}
 \item[\newcorr{$(E_{\inf})$}] \newcorr{For this function}
$$\newcorr{W_{\inf}(x)\rav\displaystyle \liminf_{\substack{\theta\uparrow\infty\\ \ }} \big[\hjk{J}(x;\theta)-\hjk{J}(\hjk{y}(0);\theta)\big]}
$$ 
 \newcorr{at least one of the following conditions is satisfied:}
  \begin{list}{}{}
  \item[\newcorr{$(E'_{\inf})$}] \newcorr{$W_{\inf}$ is lower semicontinuous at $\hjk{y}(0)$;}
  \item[\newcorr{$(E''_{\inf})$}]  \newcorr{$\hjk{y}(0)$ lies in the interior of $\mathcal{C}_0$;}
  \item[\newcorr{$(E'''_{\inf})$}]
     \newcorr{$W_{\inf}+i_{\mathcal{C}_{\as}}$ is lower semicontinuous at $\hjk{y}(0)$.}
 \end{list}
\end{list}
\newcorr{The asymptotics in these assumptions can be quite difficult to verify; for instance, see \cite[Theorem~3.2]{bald}. However, the verification of condition $(E''_{\inf})=(E''_{\sup})$ is trivial and really devoid of any asymptotic assumptions.}

	\begin{theorem}\label{9}
		Under conditions $(H0)$--$(H4)$ and $(E_{\sup})$
		let  an admissible control process $(\hjk{y},\hjk{u})$ be locally  weakly overtaking  optimal for problem \eqref{sys0_}--\eqref{sysK_}.
		
		Then,
		there exists a   nonzero solution $(\hjk{\psi},\hjk{\lambda})\in C(\mathbb{R}_+,\mathbb{X}^*)\times\{0,1\}$ of the corresponding to $(\hjk{y},\hjk{u})$ system \eqref{sys_psi}--\eqref{maxH} 
		with transversality conditions \eqref{400} and \eqref{WAKKa}:
		\begin{align}
		\label{400}
		\hjk{\psi}(0)&\in \hjk{\lambda}\partial l(\hjk{y}(0))+N(\hjk{y}(0);\cl{\mathcal{C}_0}),
		\end{align}
	\begin{subequations}
		\begin{multline}
		-(\hjk{\psi}(0),\hjk{\lambda})\in \co N(\hjk{y}(0);\cl \mathcal{C}_{\as})\times\{0\}+\co\Limsup_{\substack{\theta\uparrow\infty,\ x\to \hjk{y}(0),\ \\\hjk{J}(x;\theta)-\hjk{J}(\hjk{y}(0);\theta)\to 0}}
		N(\hjk{y}(0),\hjk{J}(\hjk{y}(0);\theta);\epi \hjk{J}(\cdot;\theta)).
		\label{WAKKa}    
		\end{multline}
        Furthermore in that case, if $(H5)$ is fulfilled, one has
		\begin{align}
-\hjk{\psi}(0)&\in \co N(\hjk{y}(0);\cl \mathcal{C}_{\as})+\co\Limsup_{\substack{\theta\uparrow\infty,\ x\to \hjk{y}(0),\ 0<\lambda\to \hjk{\lambda},\\\hjk{J}(x;\theta)-\hjk{J}(\hjk{y}(0);\theta)\to 0}}
\lambda\hjk{\partial}_x \hjk{J}(x;\theta_n).
\label{WAKKc}
\end{align}
     If in addition   $(H5)$ and $(H6)$ are fulfilled and for a given constant $R$  all the maps $\hjk{J}(\cdot;\theta)$, $\theta>0$, are $R$-Lipschitz continuous on a given  neighborhood of $\hjk{y}(0)$,
     one has 
		\begin{align}
-\hjk{\psi}(0)\in N(\hjk{y}(0);\cl \mathcal{C}_{\as})+\hjk{\lambda}\co\Limsup_{\substack{\theta\uparrow\infty,\ x\to \hjk{y}(0),\\\hjk{J}(x;\theta)-\hjk{J}(\hjk{y}(0);\theta)\to 0}}
\Big\{\frac{\partial \hjk{J}}{\partial x}(x;\theta)\Big\}.
\label{WAKKd}
\end{align}
	\end{subequations}	
	\end{theorem}
	\begin{theorem}\label{8}
		Under conditions $(H0)$--$(H4)$ and $(E_{\inf})$
		let  an admissible control process $(\hjk{y},\hjk{u})$ be locally overtaking  optimal for problem \eqref{sys0_}--\eqref{sysK_}.
		
		Then,   there exists a  nonzero solution $(\hjk{\psi},\hjk{\lambda})\in C(\mathbb{R}_+,\mathbb{X}^*)\times\{0,1\}$ of the corresponding to $(\hjk{y},\hjk{u})$ system
\eqref{sys_psi}--\eqref{maxH} 
with transversality conditions \eqref{400} and \eqref{AKKa}:
\begin{subequations}
		\begin{multline}
-(\hjk{\psi}(0),\hjk{\lambda})\in N(\hjk{y}(0);\cl \mathcal{C}_{\as})\times\{0\}\\+\bigcap_{\substack{(\theta_n)_{n\in\mathbb{N}}\in\mathbb{R}_+^\mathbb{N},\\ \theta_n\uparrow\infty}}
\Limsup_{\substack{n\uparrow\infty,\;x\to\hjk{y}(0),\\ \hjk{J}(x;\theta_n)-\hjk{J}(\hjk{y}(0);\theta_n)\to 0}
}
\check{N}(x,\hjk{J}(x;\theta_n);\epi \hjk{J}(x;\theta_n)).
\label{AKKa}
\end{multline}
        Furthermore, if in addition $(H5)$ is fulfilled, one has
\begin{align}
-\hjk{\psi}(0)&\in N(\hjk{y}(0);\cl \mathcal{C}_{\as})+
\bigcap_{\substack{(\theta_n)_{n\in\mathbb{N}}\in\mathbb{R}_+^\mathbb{N},\\ \theta_n\uparrow\infty}}
\Limsup_{\substack{n\uparrow\infty,\;x\to\hjk{y}(0),\;0<\lambda\to \hjk{\lambda}\\ \hjk{J}(x;\theta_n)-\hjk{J}(\hjk{y}(0);\theta_n)\to 0}
}
\lambda\hjk{\partial}_x \hjk{J}(x;\theta_n).
\label{AKKc}
\end{align}
\end{subequations}

	\end{theorem}
	The proofs of these theorems 
	are  presented in Section~\ref{ff}.
	\begin{remark}
		Conditions \eqref{WAKKa}--\eqref{WAKKd} can possess the continuum of solutions to \eqref{sys_psi}, but it is inescapable.
		In Example~\ref{1}, 
		the cost functional oscillates  at infinity. 
		It was enough that  
		in the corresponding
		infinite-horizon control problem (with a fixed initial state) the dimension of the family of optimal processes reaches $\dim \mathbb{X}$. Since the dynamics and integrand in this example  are also linear in $x$,
		the same dimension is inevitable for the inclusion of 
		{a} boundary necessary
		condition on co-state arcs. Furthermore,   in this example inclusion \eqref{WAKKa}  is the tightest of all boundary conditions on co-state arcs for the weakly overtaking criterion.	
		This example demonstrates that,  for the weakly overtaking  criterion without additional asymptotic assumptions, there is no hope to construct an explicit
		boundary condition on co-state arcs consistent with \eqref{maxH}.	
	\end{remark}	      
	  \begin{remark}
	  The conditions of  Theorems~\ref{9} and~\ref{8} are  assumed that a function $l$ is Lipschitz continuous (see $(H0)$). In the case of a lower semicontinuous at $\hjk{y}(0)$ function $\check{l}:\mathbb{X}\to\mathbb{R}$ one can introduce new dynamics $\check{f}(t,x,a,u)\equiv{f}(t,x,u)$ and integrand $\check{f}_0(t,x,a,u)\equiv{f}_0(t,x,u)$ on $\mathbb{R}_+\times\mathbb{X}\times\mathbb{R}\times\mathbb{U}$ with a new initial condition $\check{\mathcal{C}}_0\rav \cl\epi(\check{l}+\imath_{\mathcal{C}_0})$ and new endpoint cost $(x,a)\mapsto a$.  Then,    together with the corresponding transversality condition at infinity (either \eqref{WAKKa}, or \eqref{AKKa}), each of Theorems~\ref{9} and~\ref{8} yields the transversality condition at zero:
	  $$ (\hjk{\psi}(0),q_a)\in \hjk{\lambda}(0,1)+N(\hjk{y}(0),\check{l}(\hjk{y}(0));\cl\epi(\check{l}+\imath_{\check{\mathcal{C}}_0})),$$
	  here $q_a$ is constant, since the Hamiltonian is independent of $a$; furthermore, this constant is zero, since $\hjk{J}$ of this problem is independent of $a$.
	  So, we obtain the classic transversality condition
	  $$(\hjk{\psi}(0),-\hjk{\lambda})\in N(\hjk{y}(0),\check{l}(\hjk{y}(0));\cl\epi(\check{l}+\imath_{\check{\mathcal{C}}_0})) $$
	  for a control problem with a lower semicontinuous endpoint cost $\check{l}$.
\end{remark}
	  
	\begin{remark}
		In the definitions of the overtaking criterion and the weakly overtaking criterion,
		the time parameter $\theta$    tends to infinity  arbitrarily.
		We could fix an unbounded set $\mathbb{T}$ and  consider these definitions with the additional restriction $\theta_n\in\mathbb{T}$.
		We could apply Theorem~\ref{9} and Theorem~\ref{8} to such definitions, but
		this restriction should have been added in
		transversality conditions. In particular, this idea could be very useful in the case of boundedness of the family of
		$\frac{\partial\hjk{J}}{\partial x}(\cdot;\theta_n)$ for a given sequence of $\theta_n$.
	\end{remark}

	Consider now in more detail the 
	infinite-horizon control problems with free right endpoint, i.e.,  the case of the absence of  asymptotic constraints
	($\mathcal{C}_{\infty}=\mathbb{X}$).
In this case, under the uniform bounded gradients
	$\frac{\partial \hjk{J}}{\partial x}(x;\theta)$, assuming for $f$ and $f_0$ the smoothness in $x$ and the continuity in $u$, necessary condition \eqref{WAKKc} is deduced for the overtaking criterion in \cite{KhlopinIMM2018,KhlopinIFAC}. Now, we may be show more.
	\begin{corollary}\label{72}
		Under conditions $(H0)$--$(H4)$ and $(E_{\sup})$
		let  a process $(\hjk{y},\hjk{u})$ be locally weakly overtaking  optimal for problem \eqref{sys0_}--\eqref{sysK_}.
		Let also	$\hjk{y}(0)\in\internary\mathcal{C}_{\as}$; this holds in particular when $\mathcal{C}_{\infty}=\mathbb{X}$.
		
		Then the conclusion of Theorem~\ref{9} holds.
		
		Furthermore, 
		for all 
		$\varkappa>0$, for all natural  $i\in[1\!:\!\dim \mathbb{X}+1]$, 
		there exists a time instant $\theta_i>1/\varkappa$, a point $x_i\in\mathbb{X}$, a gradient $\zeta_i\in\hjk{\partial}\hjk{J}(x_i;\theta_i)$,
		and nonnegative numbers $\lambda_i$ and $\alpha_i$  such that one has $\sum_{k=1}^{\dim \mathbb{X}+1}\alpha_k=1$,
		$\big\|\hjk{\psi}(0)+\sum_{k=1}^{\dim \mathbb{X}+1}\alpha_k\lambda_k\zeta_k\big\|< \varkappa$, and
		$|\lambda_i-\hjk{\lambda}|+ \|x_i-\hjk{y}(0)\|+|\hjk{J}({x}_i;\theta_i)-\hjk{J}(\hjk{y}(0);\theta_i)|<\varkappa$ for all $i\in[1\!:\!\dim \mathbb{X}+1]$.
		
		If in addition $(H5)$--$(H6)$ are fulfilled and the maps 
		$x\mapsto\frac{\partial \hjk{J}}{\partial x}(x;\theta)$, $\theta>0$,
		are well-defined and bounded on a given neighbourhood of  $\hjk{y}(0)$, one can put ${\lambda}_i=\hjk{\lambda}=1$ and $\zeta_i=\frac{\partial \hjk{J}}{\partial x}({x}_i;\theta_i)$  for all $i\in[1\!:\!\dim \mathbb{X}+1]$. 
	\end{corollary}
	\begin{corollary}\label{727}
		Under conditions $(H0)$--$(H4)$ and $(E_{\inf})$
		let  a process $(\hjk{y},\hjk{u})$ be locally  overtaking  optimal for problem \eqref{sys0_}--\eqref{sysK_}.
		Let also	$\hjk{y}(0)\in\internary\mathcal{C}_{\as}$; this holds in particular when $\mathcal{C}_{\infty}=\mathbb{X}$.
		
		Then the conclusion of Theorem~\ref{8} holds.
		
		Furthermore, 
		for every positive
		$\varkappa$ and unbounded increasing sequence of $\theta_n$,  
		there exists a natural $n>1/\varkappa$, a point $x\in\mathbb{X}$,   a gradient $\zeta\in\hjk{\partial}\hjk{J}(x;\theta_n)$,
		and nonnegative $\lambda$ such that one has 
		$|\lambda-\hjk{\lambda}|+\|x-\hjk{y}(0)\|+|\hjk{J}({x};\theta_n)-\hjk{J}(\hjk{y}(0);\theta_n)|<\varkappa$ and 
		$\big\|\hjk{\psi}(0)+\lambda\zeta\big\|< \varkappa$.
		
		If in addition $(H5)$--$(H6)$ are fulfilled and there exists an unbounded increasing sequence of $\theta_n$ such that the maps 
		$x\mapsto\frac{\partial \hjk{J}}{\partial x}(x;\theta_n)$, $n\in\mathbb{N}$,
		are well-defined and bounded on a given neighbourhood of  $\hjk{y}(0)$, one can put $\hjk{\lambda}={\lambda}=1$ and $\zeta=\frac{\partial \hjk{J}}{\partial x}({x};\theta_n)$. 
	\end{corollary}

	\medskip
	
	Corollaries~\ref{72} and~\ref{727}  make it possible to apply \eqref{WAKKa} and \eqref{AKKa} without any asymptotic constraints, but
	this condition  may be satisfied by a continuum of the co-state arcs.
	Consider  another approach: let us start by searching for an explicit transversality condition, i.e., an asymptotic condition that would select exactly one co-state arc for each optimal process.  
	
	For this purpose, \cite[Theorem~8.1]{norv} proposed to find $\hjk{\psi}$
	such that it is the pointwise
	limit of a sequence of the co-state arcs that equal zero on a certain unbounded sequence  of time instants $\theta_n$. 
	The corresponding necessary condition  was proved 
	for the infinite-horizon control problem under some  strong assumptions on the asymptotics of $\rty$, $J$, and their gradients.
	Under these assumptions, the proposed condition  is equivalent to \eqref{WAKKa}.
	Later, for the same purpose, in 
	\cite{kr_as2004} and then in \cite{kr_as,JDCS,optim,av_new,Tauchnitz,belyakov2018,aseev2019new}, many assumptions on the asymptotic behavior of $f,f_0,J$, and their gradients were considered. Under these assumptions, the solution $(\hjk{\psi},\hjk{\lambda})$ to \eqref{sys_psi}--\eqref{maxH}  is determined by  the following formula: 
	\begin{subequations}
		\begin{equation} \label{AK}
		-\hjk{\psi}(0)
		=\lim_{\theta\uparrow \infty}\int_0^\theta
		\frac{\partial f_0}{\partial x}
		\big(\tau,\hjk{y}(\tau),\hjk{u}(\tau)\big)
		\, \hjk{A}(\tau)
		\,d\tau,\  \hjk{\lambda}=1.
		\end{equation}
		Here,      $\hjk{A}\in C({\mathbb{R}_+}, {\mathbb{R}^{\dim \mathbb{X}\times \dim \mathbb{X}}})$ is the solution to the Cauchy problem
		\begin{equation*}
		\frac{d{\hjk{A}}(t)}{dt} =\frac{\partial f }{\partial x}
		\big(t,\hjk{y}(t),\hjk{u}(t)\big)
		\hjk{A}(t),\quad \hjk{A}(0)=\textrm{Id}. 
		\end{equation*}
		Let us also note two equivalent representations of this formula.  The first one, obtained in \cite{JDCS}, is expressed as
		\begin{equation} \label{PsiA}
		\lim_{\theta\uparrow\infty}\hjk{\psi}(\theta) \hjk{A}(\theta)=0,
		\  \hjk{\lambda}=1
		\end{equation}
		and closely echoes the famous Shell's condition 
		\cite{shell,sagara} and 
		the Arrow-like condition  \cite{sagara,ssbook,belyakov2020}.
		The second  equivalent to \eqref{AK} expression
		\begin{equation} \label{AKd}
		-\hjk{\psi}(0)=\lim_{\theta\uparrow\infty} \frac{\partial \hjk{J}}{\partial x}(\hjk{y}(0);\theta),
		\  \hjk{\lambda}=1
		\end{equation}
		is useful in light of conditions \eqref{WAKKa} and \cite[(38b)]{norv}.
	\end{subequations}      
	
	As shown below in Example~\ref{aa}, condition \eqref{AK} may  be inconsistent with system \eqref{sys_psi}, \eqref{maxH} corresponding to an overtaking optimal process when
	the gradient at the initial state of the limit of $\hjk{J}$ does not coincide with the limit of
	gradients of $\hjk{J}$ at this state.
	This commutativity as a basic hypothesis for deducing some transversality condition was considered, in particular, in \cite[(3.4)]{kami}.
	Under similar assumptions,  the corresponding results in \cite{kr_as,JDCS,optim,belyakov2018,KhlopinIMM2018,aseev2019new,Tauchnitz2020}
	do not imply
	the following result.  
	\begin{corollary}\label{55}
		Under conditions $(H0)$--$(H6)$
		let  a process $(\hjk{y},\hjk{u})$ be locally weakly overtaking  optimal for problem \eqref{sys0_}--\eqref{sysK_}. Assume also that 
		$\hjk{y}(0)\in\internary\mathcal{C}_{\as}$; this holds in particular when $\mathcal{C}_{\infty}=\mathbb{X}$. Let there also exists a finite limit
		\begin{equation}
		\label{4004}
		\lim_{\theta\uparrow\infty,\ x\to \hjk{y}(0)
		}\frac{\partial \hjk{J}}{\partial x}(x;\theta).
		\end{equation}

		Then, 
		the system of relations \eqref{sys_psi}--\eqref{maxH}, \eqref{AK} has exactly one solution $(\hjk{\psi},\hjk{\lambda})$. Furthermore, this solution also satisfies conditions \eqref{400}, \eqref{PsiA}, \eqref{AKd},
		\begin{equation}      \label{4E}
		\lim_{\theta\uparrow\infty}\frac{\partial J}{\partial x}(x,t,\hjk{u};\theta)\Big|_{x=\hjk{y}(t)}
		=-\hjk{\psi}(t) \qquad \forall t\geq 0,
		\end{equation}
		\begin{multline}
		\inf_{u\in\mathcal{U}}\liminf_{\theta\uparrow\infty}\Big[{H}\Big(\hjk{y}(t),-\frac{\partial J}{\partial x}(x,t,\hjk{u};\theta),\hjk{u}(t),1,t\Big)\Big|_{x=\hjk{y}(t)}\\ 
		-
		{H}\Big(\hjk{y}(t),-\frac{\partial J}{\partial x}(x,t,\hjk{u};\theta)\Big|_{x=\hjk{y}(t)},u(t),1,t\Big)\Big]\geq 0 \quad\ {a.e.\ }
		\label{Anton}
		\end{multline}
	\end{corollary}
	\doc 
		Note that $(E'_{\sup})$ and $(E_{\sup})$ as well as
		the existence and the finiteness of the limit in~\eqref{AK}  is an immediate consequence of~\eqref{4004} and equalities
		\begin{equation}      \label{4C}
		\frac{\partial \rty}{\partial x}(\hjk{y}(0),0,\hjk{u};\theta)=\hjk{A}(\theta),\ 
		\frac{\partial \hjk{J}}{\partial x}(\hjk{y}(0);\theta)
		=\int_0^\theta
		\frac{\partial f_0}{\partial x}
		\big(\tau,\hjk{y}(\tau),\hjk{u}(\tau)\big)
		\, \hjk{A}(\tau)
		\,d\tau
		\quad  \forall \theta\in{\mathbb{R}_+}.
		\end{equation}
		
		By Corollary~\ref{72},  one can find a  solution $(\hjk{\psi},1)$ of the corresponding to $(\hjk{y},\hjk{u})$ system \eqref{sys_psi}--\eqref{maxH} such that  $-\hjk{\psi}(0)$ is a  convex combination of partial limits of $\frac{\partial \hjk{J}}{\partial x}(x_n;\theta_n)$
		for certain sequences $x_n\to\hjk{y}(0), \theta_n\uparrow\infty$.
		Then, by \eqref{4004}, this is the limit of $\frac{\partial \hjk{J}}{\partial x}(\hjk{y}(0);\theta)$ as $\theta\uparrow\infty$. So, we have proved \eqref{AKd}.        
		Now, from
		\eqref{4C}, we see that \eqref{AK} holds for $\hjk{\psi}$;
		moreover, condition~\eqref{AK} makes it possible to reconstruct $\hjk{\psi}$ uniquely. 
		At the same time, \eqref{maxH} holds for all $t\geq 0$ except a possibly empty subset $\mathcal{N}\subset\mathbb{R}_+$ of  zero measure. Fix this set.
		
		To prove that \eqref{PsiA} and \eqref{4E}, note that, since 
		$\hjk{\psi}$ as a solution to  \eqref{sys_psi} satisfies  the Cauchy formula
		\begin{equation}      \label{800}
		\hjk{\psi}(\theta)\hjk{A}(\theta)-\hjk{\psi}(t)\hjk{A}(t)=\int_t^\theta
		\frac{\partial f_0}{\partial x}
		\big(\tau,\hjk{y}(\tau),\hjk{u}(\tau)\big)
		\, \hjk{A}(\tau)
		\,d\tau\qquad\forall\theta>t, 
		\end{equation}
		the passage to the limit  as $\theta\uparrow\infty$ with $t=0$ leads to \eqref{PsiA}. Further, for a nonnegative $t$ and $\theta>t$,
		one has  the equality $J\big(\rty(x,0,\hjk{u};t),t,\hjk{u};\theta\big)=\hjk{J}(x;\theta)-\hjk{J}(x;t)$.
		Differentiating it in $x$ at $\hjk{y}(0)$,  we have
		\begin{align*}      
		\int_t^\theta
		\frac{\partial f_0}{\partial x}
		\big(\tau,\hjk{y}(\tau),\hjk{u}(\tau)\big)
		\, \hjk{A}(\tau)
		\nonumber
		\,d\tau\ravref{4C}\frac{\partial \hjk{J}}{\partial x}(x;\theta)\Big|_{x=\hjk{y}(0)}-\frac{\partial \hjk{J}}{\partial x}(x;t)\Big|_{x=\hjk{y}(0)}\\
		\nonumber
		\;=\;\frac{\partial J(\rty(z,0,\hjk{u};t),t,\hjk{u};\theta)}{\partial z}\Big|_{z=\hjk{y}(0)}
		\ravref{4C}\frac{\partial J}{\partial x}(x,t,\hjk{u};\theta)\Big|_{x=\hjk{y}(t)}
		\hjk{A}(t).
		\end{align*}
		Combining it  with  \eqref{800} leads to
		\(\frac{\partial J}{\partial x}(\hjk{y}(t),t,\hjk{u};\theta)
		=\big(\hjk{\psi}(\theta)\hjk{A}(\theta)-\hjk{\psi}(t)\hjk{A}(t)\big)\hjk{A}^{-1}(t).\)
		Passing to the limit as $\theta\uparrow\infty$, by \eqref{PsiA}, we obtain \eqref{4E} for all nonnegative $t$.
		
		Let us prove condition~\eqref{Anton}. Suppose it is false. Then, there could exists a  $\tau\in \mathbb{R}_+\setminus \mathcal{N}$, a  $u\in \mathcal{U}$, an~$\epsi>0$, and  an unboundedly increasing sequence of $\theta_n>0$ satisfying
		\begin{multline*}
		{H}\Big(\hjk{y}(\tau),-\frac{\partial J}{\partial x}(\hjk{y}(\tau),\tau,\hjk{u};\theta_n),\hjk{u}(\tau),1,\tau\Big)
		\leq
		{H}\Big(\hjk{y}(\tau),-\frac{\partial J}{\partial x}(\hjk{y}(\tau),\tau,\hjk{u};\theta_n),u(\tau),         1,\tau\Big)-\epsi.    
		\end{multline*}
		By \eqref{4E}, $\hjk{\psi}(\tau)$ would be the pointwise 
		limit of $-\frac{\partial J}{\partial x}(\hjk{y}(\tau),\tau,\hjk{u};\theta_n)$ as $n\uparrow\infty$; therefore,  one could have
		\[{H}\big(\hjk{y}(\tau),\hjk{\psi}(\tau),\hjk{u}(\tau),1,\tau\big)
		\leq
		{H}\big(\hjk{y}(\tau),\hjk{\psi}(\tau),u(\tau),1,\tau\big)-\epsi.\]
		This would contradict  condition~\eqref{maxH} for $\tau\in\mathbb{R}_+\setminus \mathcal{N}$. Thus, condition \eqref{Anton} has been proved.
	\bo       
	\section{Examples}
	
	The first example will show the direct calculation of the co-state arc and optimal control by  Corollary~\ref{72}. Earlier this example was considered in \cite{GBG,GBG2022} in the case when $r$ is positive and  ${\delta}$ is positive $1$-periodic stepwise function with two regimes.
	
	\begin{example}\label{ggg} 
			\begin{align*}        
		\textrm{Minimize\ }            
		&\int_{0}^{\infty} e^{-\varrho \tau}\Big[\frac{u^2(\tau)}{2}+y(\tau)-u(\tau)\Big]\, d\tau\\ 
		\textrm{subject to\ }&\frac{d{y}(t)}{dt}=\beta u(t)-{\delta}(t)y(t)\ \textrm{a.e.},\ y(0)=x_*> 0,\ u(t)\in U\rav [0;1].
		\end{align*}
		Here, $\varrho$ and $\beta$ are  constants and a function $\delta:\mathbb{R}\to\mathbb{R}$ is Lebesgue measurable,  locally summable, and $1$-periodic. To simplify, we will focus on the case $\beta\neq 0$.
		Let $(\hjk{y},\hjk{u})$ be a locally weakly overtaking optimal in this problem. 
		
		The Hamilton-Pontryagin function in this example is
		$$\psi\big(\beta u-\delta(t) x\big)-\lambda e^{-{{\varrho}} t}\big(\frac{u^2}{2}+x-u\big)$$
		with the adjoint equation  
		$\frac{d\psi(t)}{dt}=\lambda e^{-\varrho t}+\psi(t)\delta(t);$
		so,  its solutions are of the form 
		$$\psi(t)=e^{\int_{0}^t\delta(\tau)\,d\tau}(\psi(0)+\lambda \int_0^t e^{-\int_{0}^\tau(\varrho+\delta(s))\,ds}\,d\tau)\qquad\forall t\geq 0.$$
		
	  	{Define $R\rav \varrho + \int_{0}^1\delta(\tau)\,d\tau$.} 
   \newcorr{In the case $R>0$ the function  $\psi$ is bounded, $\hjk{J}$ is Lipschitz continuous and assumptions $(E'_{\sup})$ and $(E_{\sup})$ as well as  $(E'_{\inf})$ and $(E_{\inf})$ are fulfilled.} 
   \corr{It is easy to see that this problem satisfies $(H0)$--$(H6)$.  It follows that Corollary~\ref{72} is applied.}
		By  Corollary~\ref{72}, we must find all partial limits  of   $(\lambda_n {\psi}_n,{\lambda}_n)$ (as $\theta_n\uparrow\infty$) satisfying  $\psi_n(\theta_n)=0$ with $\lambda\downarrow\hjk{\lambda}$.
		Dividing this equality by $e^{\int_{0}^t\delta(\tau)\,d\tau}$, we obtain
	 \begin{align*}
	 \psi_n(0)+\lambda_n \int_0^{\theta_n} e^{-\int_{0}^\tau(\varrho+\delta(s))\,ds}\,d\tau=0.
	 \end{align*}	 
      Thus, 
       the pair $(\hjk{\psi}(0),\hjk{\lambda})$ is a convex combination of   partial limits of the sequence of
		$$ (\lambda_n \psi_n(0),\lambda_n)=(- \lambda^2_n\int_0^{\theta_n}  e^{-\int_{0}^\tau(\varrho+\delta(s))\,ds}\,d\tau,\lambda_n).$$ In particular, $\hjk{\psi}(0)$ is nonpositive.
    	
		In the case $R>0$, due to  $\int_0^{\theta_n}  e^{-\int_{0}^\tau(\varrho+\delta(s))\,ds}\,d\tau \sim e^{-R\theta_n}$, we obtain that the sequence of $\psi_n(0)$ is bounded. By $\hjk{\lambda}\in\{0,1\}$, we get  $\hjk{\lambda}=1$. This entails
	\begin{align*}
		\hjk{\psi}(t)e^{\varrho t}&=
		-e^{\int_{0}^t(\varrho+\delta(\tau))\,d\tau}\int_t^\infty e^{-\int_{0}^\tau(\varrho+\delta(s))\,ds}\,d\tau=
		-\int_0^\infty e^{-\int_{0}^\tau(\varrho+\delta(t+s))\,ds}\,d\tau\\
	\textrm{and\ }
	    \hjk{u}(t)&=\frac{1}{2}\max\big(0,\min(1,1+\hjk{\psi}(t)\beta e^{\varrho t}\big)\big)\\&
	    =\max\Big(0,\min\Big(1,\frac{1}{2}-\frac{\beta}{2}\int_0^\infty e^{-\int_{0}^\tau(\varrho+\delta(t+s))\,ds}\,d\tau\Big)\Big)
	\end{align*}	
		for almost all positive $t$.
		

		 Thus,  in this example  the condition \eqref{WAKKa} points to the unique solution of relations  \eqref{sys_x}--\eqref{maxH}; therefore  there can be at most
one locally weakly overtaking optimal process; in addition, this process must be generated by  $1$-periodic function.

    \end{example}
	
	The following two examples were inspired by the optimal growth theory:
	the 
	 Ramsey-like problem and
	the  Beltratti--Chichilnisky--Heal problem of  sustainable growth \cite{bch} with logistic renewal function.
	In addition,  these examples make it possible entirely and directly to apply Theorem~\ref{9}. 
	Along the road, we will consider a   set of tools and methods, different from the direct calculation of any motions,  costs, and their  gradients.  We will see that, on the one hand,  this road is not simple and user-friendly, and so these tools and methods should be improved. On the other hand, after proper preparation, its application is quick and comfortable.

	 To this aim, we prepare several  facts for control systems corresponding to some economic applications.
\begin{proposition}\label{done}
Let 
		$U\subset\mathbb{R}_+$ be interval with $\internary U\neq \varnothing$.
   		Let  continuous functions $g:\mathbb{R}\to\mathbb{R}$ and $g_0:U\to\mathbb{R}$  satisfy
   		\begin{align}\label{plus}
		g(0)=0,\ g''(x)<0,\ g'_0(\upsilon)<0, \textrm{and}\ 
		g''_0(0)>0\quad\forall x\in(0;\infty),\upsilon\in U.
		\end{align}
		Let  a number $\varrho\in\mathbb{R}$  and initial position  $x_*>0$ be given. 
   		
   		Assume that a pair $(\hjk{y},\hjk{u})$ is a  locally weakly overtaking optimal process to the following problem:
		\begin{align*}        
		\textrm{minimize\ }            
		&\int_{0}^{\infty} e^{-\varrho \tau} g_0(u(\tau))\, d\tau\\ 
		\textrm{subject to\ }&\frac{d{y}(t)}{dt}=g(y(t))-u(t)\ \textrm{a.e.},\\
		& y(0)=x_*>0,\ y(t)\in\mathbb{R},\ u(t)\in U,\ \Limsup_{\theta\uparrow\infty} \{\sign y(\theta)\}\subset 1.
		\end{align*}
      Assume also that   $\hjk{y}$ is positive.
   
   Then,
    \begin{enumerate}
        \item hypotheses $(H0)$--$(H4)$ are satisfied for $(\hjk{y},\hjk{u})$;
        \item  there exists a  solution $(\hjk{\psi},\hjk{\lambda}=1)$ to relations \eqref{sys_psi}, \eqref{maxH}, and \eqref{WAKKa} with $(\hjk{y},\hjk{u})$.  
	 \item   one has
         \begin{align}\label{1040}
         \hjk{u}(t)=\eta(e^{\varrho t}\hjk{\psi}(t)) \qquad \textrm{a.e.},
         \end{align} 
         here $\eta(q)=\arg\max_{u\in U}[-qu-g_0(u)]$ for all $q\in\mathbb{R}$;
        \item the pair $(\hjk{y},\hjk{p}\rav\hjk{\psi}e^{\varrho\cdot})$ solves
        the Hamiltonian system
         \begin{align}\label{ds}        
		\frac{dy(t)}{dt}=g(y(t))-\eta(p(t)),\ 
		\frac{d{p}(t)}{dt}=p(t)\big(\varrho-g'\big({y}(t)\big)\big),
		\end{align}
         	\item  either $\hjk{\psi}\equiv 0$ 
	       or $\hjk{\psi}$ is positive.
	 \end{enumerate}
     Furthermore, in the case of positive $\hjk{\psi}$ 
     \begin{enumerate}
         \item   the sign  of the arc  ${y}$ in unstable at the solution  $\hjk{y}$ to this system $\frac{dy(t)}{dt}=g(y(t))-\hjk{u}(t)$: 
         there exists a converging to $\hjk{y}$ sequence of  solutions $y_n$ to this system such that  inequalities $y_n(\theta_n)\leq 0$ hold for  a sequence of  positive $\theta_n$;
         \item the sign   of the arc $\hjk{y}$  in  unstable at the solution $(\hjk{y},\hjk{p})$ to \eqref{ds}.
     \end{enumerate}
   \end{proposition}
	\doc 
	    It is easy to see that $(H0)$--$(H1)$ are  fulfilled.

	    Put $R\rav1+\sup_{t\in\mathbb{R}_+} \hjk{y}(t)\in[-\infty;+\infty]$, $r\rav \inf_{t\in\mathbb{R}_+} \hjk{y}(t)\in[-\infty;+\infty]$. By condition on $\hjk{y}$, the function $g$ is continuously differentiable on $(0;R+1)$  and the map $t\mapsto\frac{dg}{dx}(\hjk{y}(t))$ is locally summable.  
	    
	    We claim that the hypotheses $(H2)$--$(H4)$ as well  $(E'_{\sup})$ and $(E_{\sup})$ are fulfilled. Indeed, in the case where 
	     there exists a positive  $t_0$ with $\hjk{y}(t_0)=r$, the function $g$ is bounded and differentiable on $[r-\epsi;+\infty)$ for a positive $\epsi$; therefore, the hypotheses $(H2)$--$(H4)$ are also satisfied with $\mathbb{G}\rav\mathbb{R}_+\times [r-\epsi;R+1]$.
	      In the case $\hjk{y}>r$, we put $\epsi\rav0$,
	    $\mathbb{G}\rav\{(t,x)\in\mathbb{R}_+\times\mathbb{R}\mid {\hjk{y}(t)+r} \leq 2x\leq 2R+2\}$. So, $(H2)$--$(H4)$ are verified.

	     Further,  if $f$ changes outside $G$, the pair $(\hjk{y},\hjk{u})$ will remain weakly overtaking optimal.
	      Redefining  $g$ on $(-\infty;0)$, we get that $g$ is bounded and continuous on $(-\infty;0]$. Since $\hjk{y}$ is positive, we have that $\int_{0}^t|\hjk{u}(t)|\,d\tau\leq x_*+\int_{0}^t{g}(\hjk{y}(\tau))\,d\tau$ for all positive $t$. At the same time, the concavity of $g$ leads to  $0\leq g(x)\leq |g(x_*)|+|g'(x_*)|\cdot|x-x_*|$  for all positive $x$. Therefore, one finds a locally summable function $M:\mathbb{R}_+\to\mathbb{R}_+$ such that $|g(x)|+|\hjk{u}(t)|\leq M(t)(1+|x|)$ for all $(t,x)\in\mathbb{R}_+\times\mathbb{R}$.  Therefore every solution $y(\cdot)\rav\rty(x,0,\hjk{u};\cdot)$ is defined on $\mathbb{R}_+$. Since
	      $\int_{0}^\theta f_0(\hjk{u}(\tau))\,d\tau$ is finite for all positive $\theta$,  the corresponding control process $(y,\hjk{u})$ is admissible; it follows that $\hjk{J}(x;\theta)=J(x,0,\hjk{u};\theta)=J(x_*,0,\hjk{u};\theta)$. So,  ${J}(x;\theta)$ is independent of $x$.
	      Since	      
 $\hjk{J}(x;\theta)$ is independent of $x$, the assumptions $(E'_{\sup})$ and $(E_{\sup})$ hold true.
		
		By Theorem~\ref{9},
		there exists a nonzero solution $(\hjk{\psi},\hjk{\lambda})$ to relations \eqref{sys_psi}, \eqref{maxH}, and \eqref{WAKKa}.

		Note that the set $\mathcal{C}_{\as}$ is  an interval unbounded above. Further, since
		$\hjk{J}(x;\theta)$ is independent of $x$,  for all $(x,\theta)\in\mathbb{R}\times\mathbb{R}_+$ one has
		$\hjk{\partial}^\infty_x \hjk{J}(x;\theta)=\varnothing$,
		$\hjk{\partial}_x \hjk{J}(x;\theta)=\{0\}$, and 
		$N(x,\hjk{J}(x;\theta);\epi \hjk{J}(\cdot;\theta))=\{0\}\times (-\mathbb{R})$. Then,
		from  \eqref{WAKKa} it follows that
		$$ -(\hjk{\psi}(0),\hjk{\lambda})\in N(\hjk{y}(0);\cl\mathcal{C}_{as})\times(-\mathbb{R}_+).$$
		 In particular, we have been proved that $\hjk{\psi}$ is nonnegative.
		
		Let us prove $\hjk{\lambda}>0$. Suppose it is false, $\hjk{\lambda}=0$. Then, $\hjk{\psi}$ is positive, Since  the value of $\hjk{\psi}(t)\upsilon$ would attain a minimum for each $t$, the control   $\hjk{u}(t)$  would also be minimal. Then, another admissible  control process $(y,u)$ satisfies $J(x_*,0,u;t)>J(x_*,0,0;t)$ for all large $t$. Therefore, the process  $(x_*,0)$ would be  a unique admissible control process. It means that $x_*$ would be zero, in contradiction to $x_*>0$. Thus, we have $\hjk{\lambda}>0$. 
		
		Due to $\hjk{\lambda}\in\{0,1\}$, we obtain $\hjk{\lambda}=1$.

		  Set $P\rav\{-\frac{dg_0(\upsilon)}{d\upsilon}\,|\,\upsilon\in U\}$. Since $g_0$ is continuously differentiable on interval $U$, $P$ is an interval too. Furthermore, $g_0$ is strongly convex, the map $U\ni \upsilon\mapsto -\frac{dg_0(\upsilon)}{d\upsilon}$ is continuous and increasing. Then,
		  the inverse map  is also continuous and increasing, this map can be extended by continuity to the nondecreasing map $\mathbb{R}_+\ni p\mapsto \eta(p)\in U.$  Consider the Legendre transform of
		 the convex function $g_0+{\imath}_{U}$, the map
		 $-\mathbb{R}_+\ni p\mapsto
		 \max_{\upsilon\in   U}(p\upsilon-g_0(\upsilon))$.
		 This function is  smooth and strongly convex, therefore, for all $p_0\in \mathbb{R}_+$ and $\upsilon_0\in U$, one has  $\eta(p_0)=-{g'_0}(\upsilon_0)$ iff $-p_0\upsilon_0-g_0(\upsilon_0)=\max_{\upsilon\in  U}(-p_0\upsilon-g_0(\upsilon))$ holds.
		 
		  Hence, on the one hand, from \eqref{maxH} and $\hjk{\lambda}=1$ it follows that  \eqref{1040}; on the other hand, from
		 \eqref{sys_psi} and $\hjk{\lambda}=1$ it follows that $\hjk{\psi}$ satisfies
		$
\frac{d{\psi}(t)}{dt}=-\frac{\partial \mathcal{H}}{\partial x}(\hjk{y}(t),\hjk{\psi}(t),t)$ with 
$\mathcal{H}(x,\psi,t)\rav \psi g(x)-\psi\eta(e^{\varrho t}\psi)-\lambda e^{-\varrho t} g_0(e^{\varrho t}\psi)$. Now, the following function 
$\hjk{p}(t)=e^{\varrho t}\psi(t)$
solves on $\mathbb{R}_+$ the equation
$\frac{d{p}(t)}{dt}=\varrho p(t)-\frac{\partial \mathcal{H}}{\partial x}(\hjk{y}(t),p(t),t)$. By \eqref{1040},  $\hjk{y}$ solves 
$\frac{dy(t)}{dt}=g(y(t))-\eta(p(t))$ and the pair $(\hjk{y},\hjk{p})$ solves \eqref{ds}.

		Assume that on finds  a nonnegative $t$ with $\hjk{\psi}(t)=0$.  Then, by \eqref{sys_psi}, the co-state arc $\hjk{\psi}$ as a solution to $\frac{d\psi(t)}{dt}=-\psi(t)\frac{\partial f(\hjk{y}(t))}{\partial x}$ is zero. So, $\hjk{\psi}\equiv 0$ and, by \eqref{maxH},  control $\upsilon=\hjk{u}(t)$ should minimize the strongly decreasing on $\mathbb{R}_+$ function $\upsilon\mapsto g_0(\upsilon)$; therefore,
		$\hjk{u}\equiv\max_{\upsilon\in U} \upsilon$.
		
		Consider the case where $\hjk{\psi}$ is positive. 
		
		Since
		$-\hjk{\psi}(0)$ lies in $N(x_*;\mathcal{C}_{\as}),$ the point  $x_*=\hjk{y}(0)$ lies on the boundary of $\mathcal{C}_{\as}.$
		Then,   it  is the minimal  of all positive  motions generated by $\hjk{u}$; it means that for  each  positive $\epsi$ small enough, the motion $\rty(x_*-\epsi,0,\hjk{u};\cdot)$ doesn't save its sign. Fix a such $\epsi$  with its motion $y_\epsi(\cdot)\rav\rty(x_*-\epsi,0,\hjk{u};\cdot)$.	
		
	 Further,  let $T_0$ be minimal time instance satisfying
		$y_\epsi(T_0)=0$.
		Together with $y_\epsi$ consider the solution $({\tilde{z}_\epsi}(\cdot),\tilde{p}(\cdot))$ to \eqref{ds} satisfying the initial conditions ${\tilde{z}_\epsi}(0)=x_*-2\epsi$, $\tilde{p}(0)=\hjk{\psi}(0)-\epsi.$ Define also $\tilde{q}(\cdot)\rav \tilde{p}(\cdot)e^{-\varrho \cdot}$.
		Consider 
		the sequence of  
		\(T_n\rav \sup\{t\geq 0\,|\,\tilde{q}(\tau)<\hjk{\psi}(\tau),\ 1/n<{\tilde{z}_\epsi}(\tau)<y_\epsi(\tau)\quad\forall \tau\in[0;t]\}\); this sequence is nondecreasing and bounded by $T_0$.
		On the one hand, for all positive $t<T_n$,  according to $\hjk{\psi}(t)>\tilde{q}(t)$, $1/n<{\tilde{z}_\epsi}(t)<y_\epsi(t)<\hjk{y}(t)$, and 
		$g'(\hjk{y}(t))\leq g'({\tilde{z}_\epsi}(t))\leq g'(1/n)$, we obtain
		\begin{align*}
		\frac{d(\tilde{q}(t)-\hjk{\psi}(t))}{dt}&=
		\hjk{\psi}(t)g'(\hjk{y}(t))-
		\tilde{q}(t)g'({\tilde{z}_\epsi}(t))\\
		&\leq (\hjk{\psi}(t)-\tilde{q}(t))g'({\tilde{z}_\epsi}(t))\leq (\hjk{\psi}(t)-\tilde{q}(t))g'(1/n).
		\end{align*}
		Hence one has $\hjk{\psi}(t)-\tilde{q}(t)\geq\epsi e^{-g'(1/n)t}$
		and 
		$\eta(\hjk{\psi}(t)e^{\varrho t})\geq
		\eta(\tilde{q}(t)e^{\varrho t})=\eta(\tilde{p}(t))$ for all positive $t\leq T_n$.
		On the other hand, for a such $t$, we have 
		\begin{align*}
		\frac{d(y_\epsi(t)-{\tilde{z}_\epsi}(t))}{dt}&=g({\tilde{z}_\epsi}(t))-g(y_\epsi(t))+\eta(\tilde{q}(t)e^{\varrho t})-\eta(\hjk{\psi}(t)e^{\varrho t})\\
		&\leq (y_\epsi(t)-{\tilde{z}_\epsi}(t))g'(1/n).
		\end{align*}
		Then, we obtain $y_\epsi(t)-{\tilde{z}_\epsi}(t)\geq\epsi e^{-g'(1/n)t}$ and  $y_\epsi(t)>{\tilde{z}_\epsi}(t)$ for all positive $t\leq T_n$. 
		According to $\tilde{q}(T_n)<\hjk{\psi}(T_n)$
		and  ${\tilde{z}_\epsi}(T_n)<{y}_\epsi(T_n)$, we have  ${\tilde{z}_\epsi}(T_n)=1/n$. Since the sequence of $T_n$  is nondecreasing and bounded,  ${\tilde{z}_\epsi}(T')$ is zero  for a positive  $T'$. Thus, ${\tilde{z}_\epsi}$ also doesn't save its sign for  positive $\epsi$ small enough.

		Accordingly, since $||\hjk{y}(0)-{\tilde{z}_\epsi}(0)||+||\hjk{\psi}(0)-\tilde{q}(0)||=||\hjk{y}(0)-{\tilde{z}_\epsi}(0)||+||\hjk{p}(0)-\tilde{p}(0)||$ can be chosen arbitrary small, the sign  of  ${y}$  is also  unstable at the solution $(\hjk{y},\hjk{p})$ to   \eqref{ds}. 
		
		Proposition~\ref{done} has been proved.
	\bo 	
		
    
 	Note that condition \eqref{plus}  is only part of the Inada conditions, typical conditions on the renewal (production) 
		and utility functions 
		in optimal growth models; see  \cite[(3.19a)-(3.19c)]{feichtinger} 
		and \cite[Assumptions 3]{bch}.
   In the following example, Ramsey-like problem, $g$ and $g_0$ satisfy all Inada conditions, including $g'(0+)=+\infty$. It  follows that the Pontryagin Maximum Principle is useless for $x=0$. In particular,   in this example hypotheses $(H5)$ and $(H6)$ are difficult to verify.

	\begin{example}\label{aaalike} 
	Fix  a number $\varrho\in\mathbb{R}$  and a  positive number $x_*$.  Take	\[g(x)\rav \sqrt{\max(x,0)}-x,\  g_0(\upsilon)\rav-\sqrt{\upsilon} \qquad \forall x\in \mathbb{R},\upsilon\in \mathbb{R}_+.\]
		Consider the following problem:
	\begin{align*}
	\textrm{maximize\ }            
	&\int_{0}^{\infty} e^{-{{\varrho}} \tau} \sqrt{{u(\tau)}}\, d\tau\\ 
	\textrm{subject to\ }&\frac{d{y}(t)}{dt}=\sqrt{{y(t)}}-y(t)-u(t)\ \textrm{a.e.},\\
	& y(0)=x_*>0,\ y(t)\in\mathbb{R},\ {u(t)\in U\rav\mathbb{R}_+},\ {\Limsup_
		{\theta\uparrow\infty} \{\sign y(\theta)\}\subset\{1\}}.
	\end{align*}
	Let us consider a locally weakly overtaking optimal process $(\hjk{y},\hjk{u})$ with a positive $\hjk{y}$.
	Now, all  assumptions of Proposition~\ref{done} satisfy. In particular, $(H1)$--$(H4)$ are fulfilled.

	Due to Proposition~\ref{done}, we know that there exists a corresponding to $(\hjk{y},\hjk{u})$  pair $(\hjk{\psi},\hjk{\lambda}=1)$. Since $\hjk{\psi}\equiv 0$ could lead to $\hjk{u}\equiv \sup\{\upsilon\, |\,\upsilon \in U\}=\infty$, we obtain that $\hjk{\psi}$ is positive and there exists a converging to $\hjk{y}$ sequence of  generated by $\hjk{u}$  solutions $y_n$ with $\sign y_n\not\equiv 1$. Notice that,  we might try to seek out
	all positive ${\psi}(0)$ and, calculating the corresponding solution $(y,\psi,u,1)$ to \eqref{sys_x}--\eqref{maxH}, verify the unstability of the sign  of ${y}$. However, since $\hjk{u}$ and $\hjk{\psi}$ are unknown, it is very difficult. 
	
	Thanks to last item of Proposition~\ref{done}, we may   seek out
	a solution $(y(\cdot),p(\cdot))$ with positive $p$ and the unstable  sign  of ${y}$ for the system
		\begin{align}\label{ds1}        
\frac{dy(t)}{dt}=\sqrt{y(t)}-y(t)-\frac{1}{4p^2(t)},\qquad 
\frac{d{p}(t)}{dt}=\Big(\varrho+1-\frac{1}{2\sqrt{y(t)}}\Big)p(t).
\end{align}		
Albeit  the control $\hjk{u}$ and the co-state arc $\hjk{\psi}(t)=\hjk{p}(t)e^{-\varrho t}$ are still unknown, in this way we must verify solutions
to just one dynamical system. Let's do it.

\begin{figure}[!ht]
	\begin{minipage}[h]{8cm}
		\center{\includegraphics[width=240px]{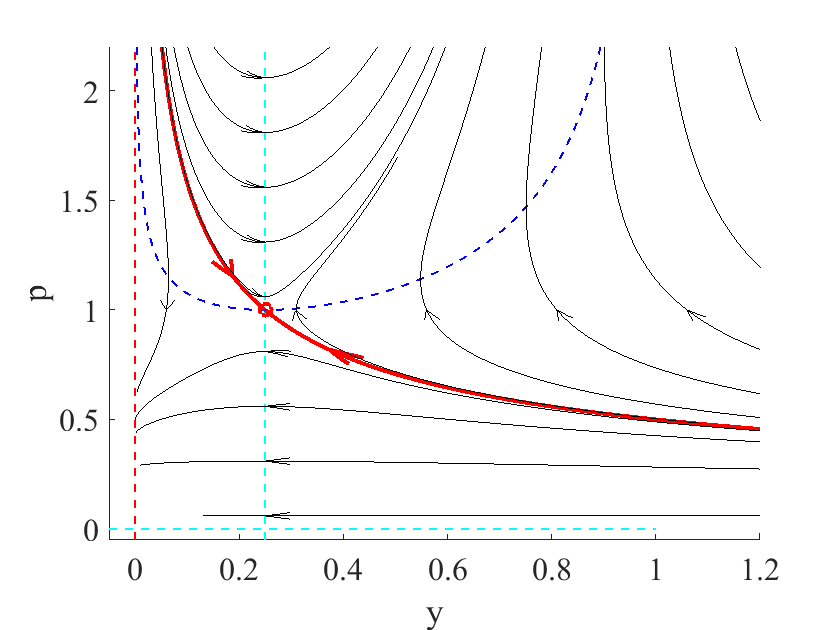}} \\(a) $\varrho=0$.
	\end{minipage}
	\hfill
	\begin{minipage}[h]{8cm}
		\center{\includegraphics[width=240px]{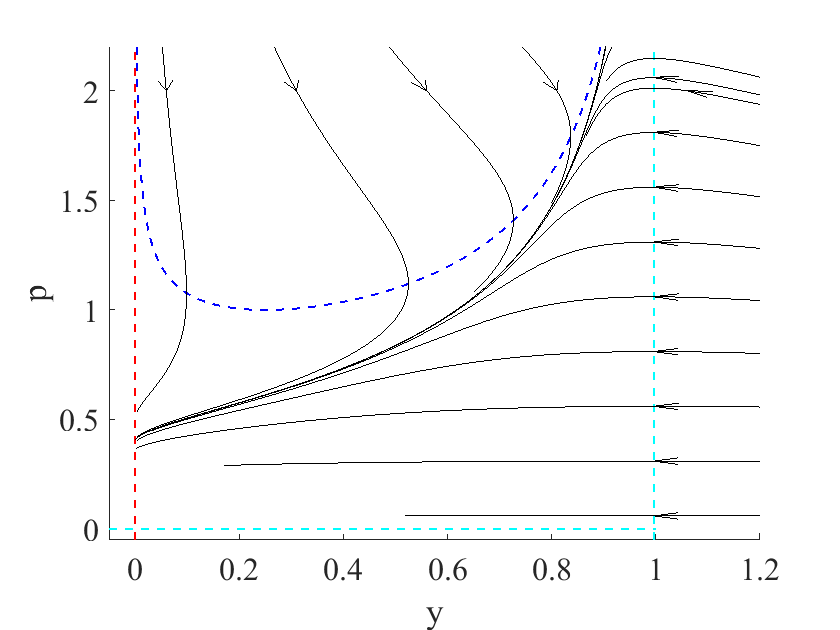}}\\ (b) $\varrho=-1/2$.
	\end{minipage}
	\caption{The typical phase diagrams for solutions $(y(\cdot),p(\cdot))$ to \eqref{ds1} if (a) $\varrho> -1/2$; (b) $\varrho\leq -1/2$.}
	\label{aaalike1}
\end{figure}

		At the beginning, consider  the case  $\varrho>-1/2$ (see Fig.~\ref{aaalike1}(a)).  Since $g_0$ and $-g$ are strongly convex on $\mathbb{R}_+$, 
there exists a  unique stationary point $(x_0,p_0)=((\varrho+1)^{-2}/4,(\varrho+1)/\sqrt{2\varrho+1})$ of dynamical system \eqref{ds} lying  in $\mathbb{R}^2$; furthermore, this point is a saddle point.
Then, the image of the path $t\mapsto(\hjk{y}(t),\hjk{p}(t))$, the set
$\{(\hjk{y}(t),\hjk{p}(t))\,|\,t\in\mathbb{R}\}$, is contained in the unstable manifold \cite{hartman} of \eqref{ds1}.
Hence, $(\hjk{y},\hjk{p})$ is
either  $\{(x_0,p_0)\}$ or  one of two unstable paths of the system,  converging to $(x_0,p_0)$ (see Fig.~\ref{aaalike1}(a)). 	So, in this case,  the transversality condition  \eqref{WAKKa} elicits a unique motion and for each weakly overtaking optimal process $(\hjk{y},\hjk{u})$ with positive motion $\hjk{y}$, its motion
 converges to $(\varrho+1)^{-2}/4$ as $t\uparrow \infty$.	

 In the case  $\varrho\leq -1/2$, the system \eqref{ds1} has no solution
 $(y,p)$ with 
 positive  $\hjk{y}(\cdot)$ (see Fig.~\ref{aaalike1}(b)); therefore,  in the case    $\varrho\leq -1/2$ there exists no locally weakly overtaking optimal process in the optimal control problem.
 \end{example}

	\begin{example}\label{aaa} 
		Again consider  a number $\varrho\in\mathbb{R}$  and a  positive number $x_*$.  Now take	\[g(x)\rav 2x-x^2,\  g_0(\upsilon)\rav-\sqrt{\upsilon} \qquad \forall x\in \mathbb{R},\upsilon\in \mathbb{R}_+.\]
		It gives  the following problem:
		\begin{align*}
		\textrm{maximize\ }            
	&\int_{0}^{\infty} e^{-{{\varrho}} \tau} \sqrt{{u(\tau)}}\, d\tau\\ 
	\textrm{subject to\ }&\frac{d{y}(t)}{dt}={2y(t)-y^2(t)}-u(t)\ \textrm{a.e.},\\
	& y(0)=x_*>0,\ y(t)\in\mathbb{R},\ {u(t)\in U\rav [0;1]},\ {\Limsup_
		{\theta\uparrow\infty} \{\sign y(\theta)\}\subset\{1\}}.
	\end{align*}
	Note that a locally weakly overtaking optimal process  $(\hjk{y},\hjk{u})$ with positive $\hjk{y}$ satisfies the conditions of Proposition~\ref{done}. In particular, $(H1)$--$(H4)$ hold true.
   	
	Due to Proposition~\ref{done}, we know that either $\hjk{u}\equiv 1$ (with $\hjk{\lambda}=1$, $\hjk{\psi}=\hjk{p}=0$), or there exists sequence of  solutions $(y_n,p_n)$ to system
			\begin{align}\label{ds2}        
	\frac{dy(t)}{dt}=2y(t)-y^2(t)-\frac{1}{\max(1,4p^2(t))},\qquad 
	\frac{d{p}(t)}{dt}=(\varrho-2+2y(t))p(t),
	\end{align}		
	 converging to $(\hjk{y},\hjk{p})$  such that, for all natural $n$, one finds a positive  $t$ with
	 $y_n(t)=0$. So, we must seek out the  solutions $(\hjk{y},\hjk{p})$ to \eqref{ds2} with positive $\hjk{y}$ such that either $\hjk{p}\equiv 0$, or $\hjk{p}> 0$ and  the  sign of $\hjk{y}$ is unstable in \eqref{ds2}.

	\begin{figure}[ht]
	\center{\includegraphics[width=340px]{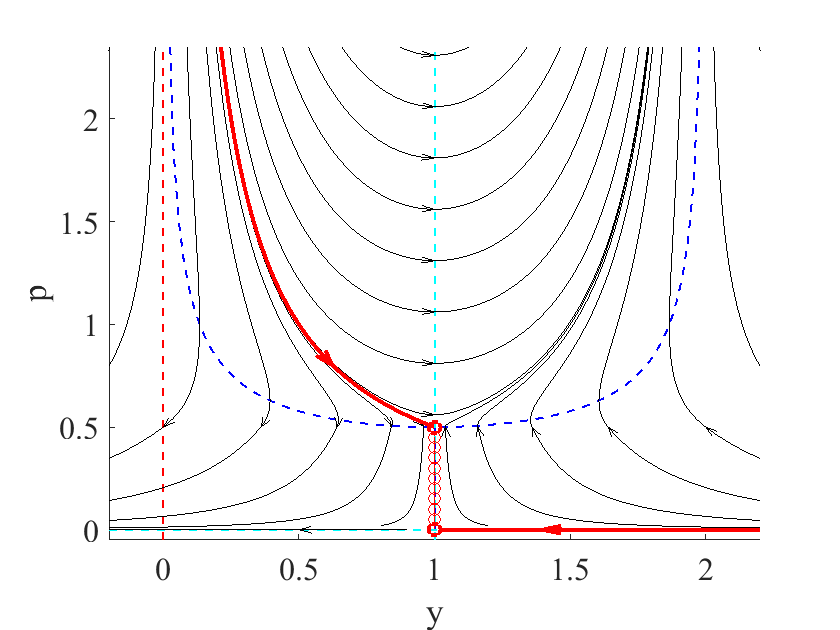}}
	\caption{
		The phase diagram for solutions $(y(\cdot),p(\cdot))$ to \eqref{ds2} in the case $\varrho=0$.
	}\label{aaa0}
\end{figure}

     At the beginning, 
     introduce a solution $({y}^0,0)$ to \eqref{ds2} with  $\hjk{p}\equiv 0$ and $y^0>0$. The direct calculations give
     ${y}^0(t)=\frac{(x_*-1)(t+1)+1}{(x_*-1)t+1}$ for all $t\in\mathbb{R}$; in particular,  ${y}^0\equiv 1$ if $x_*=1$.
          Notice also that, in the case  $|\varrho|<2$, $\varrho\neq 0$,
     there exists a  positive  stationary point 
     $(x_0,p_0)=(1-\varrho/2;\frac{1}{\sqrt{4-\varrho^2}})$; furthermore, this point is  unique in $(0;+\infty)\times(0;+\infty)$ and  is a saddle point. Hence, there exists two  unstable paths of this Hamiltonian system,  converging to $(x_0,p_0)$ (see Fig.~\ref{aaa-1} and~\ref{aaa1}); denote these paths by $y_{\rm{left}}$ and $y_{\rm{right}}$.

     In the case $\varrho=0$ (see  Fig.~\ref{aaa0}) the stationary points of \eqref{ds2}  constitute the interval $\{1\}\times[0;1/2]$ and for all $x_*<1$ there exists a unique path $(y_{\downarrow},p_{\downarrow})$ of  \eqref{ds2}, converging to the point $(1,1/2)$. In this case define 
     $$S\rav\{(y_{\downarrow}(t),p_{\downarrow}(t))\in\mathbb{R}_+\times\mathbb{R}_+\,|\,t\in\mathbb{R}\}\cup([1;\infty)\times\{0\})\cup(\{1\}\times[0;1/2]).$$
     Later we will prove that the image  of $t\mapsto(\hjk{y},\hjk{p})$ is contained in $S$, in particular, $\hjk{y}$ is $y_{\downarrow}$ if $x_*<1$  and  ${y}^0$ otherwise.

     In the case $\varrho<0$ (see  Fig.~\ref{aaa-1}(a,b)) there exists a unique path $(y_{\downarrow},p_{\downarrow})$ of  \eqref{ds2}, converging to the point $(1,0)$. In this case define 
$$S\rav\{(y_{\downarrow}(t),p_{\downarrow}(t))\in\mathbb{R}_+\times\mathbb{R}_+\,|\,t\in\mathbb{R}\}\cup([1;\infty)\times\{0\}).$$
Later  we will prove  that      the image  of $t\mapsto(\hjk{y},\hjk{p})$ is contained in $S$; in particular, $\hjk{y}$ is $y_{\downarrow}$ if $x_*<1$  and  ${y}^0$ otherwise.

\begin{figure}[!ht]
	\begin{minipage}[h]{8cm}
		\center{\includegraphics[width=240px]{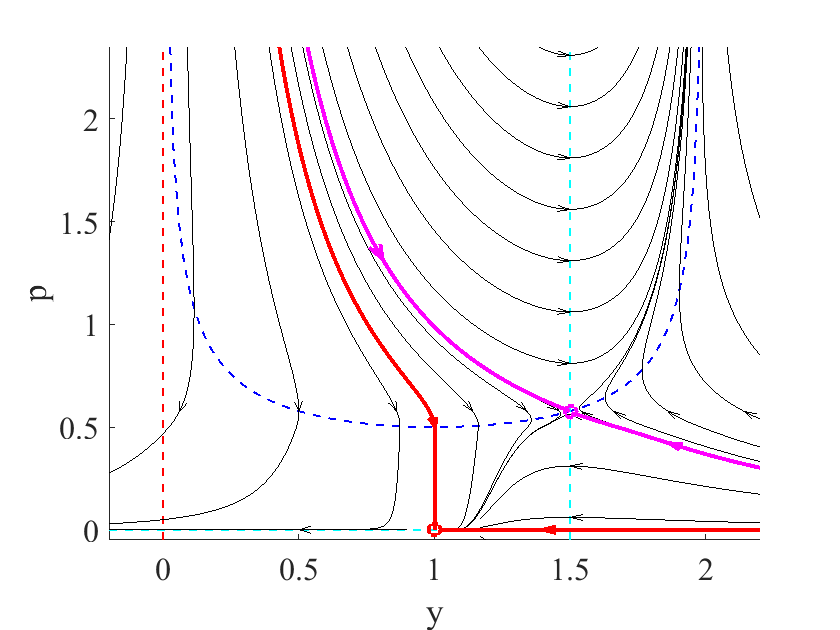}} \\(a) $\varrho=-1$.
	\end{minipage}
	\hfill
	\begin{minipage}[h]{8cm}
		\center{\includegraphics[width=240px]{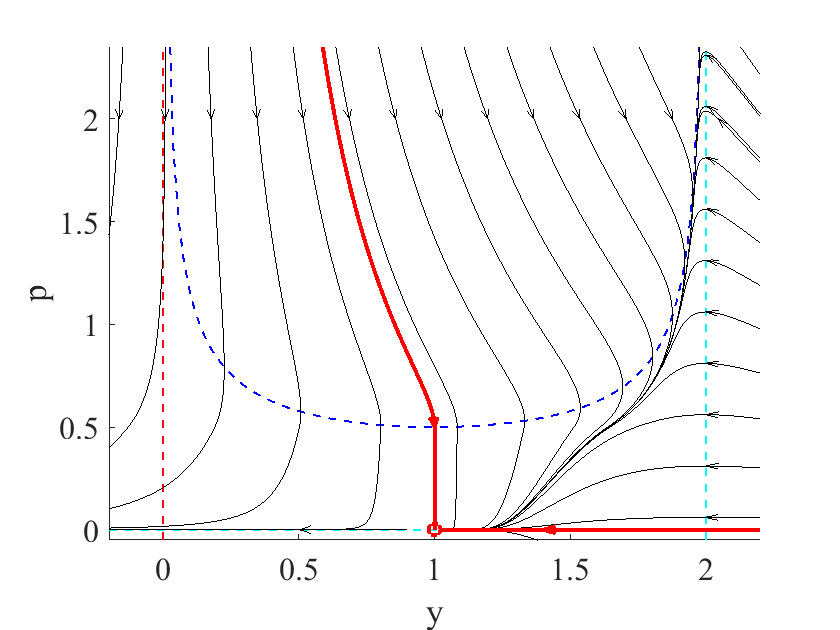}}\\ (b) $\varrho=-2$.
	\end{minipage}
	\caption{The typical phase diagrams for solutions $(y(\cdot),p(\cdot))$ to \eqref{ds2} if (a) $\varrho\in(-2;0)$; (b) $\varrho\leq -2$.}
	\label{aaa-1}
\end{figure}

       In the case $\varrho\in(0;2)$ (see  Fig.~\ref{aaa1}(a))  the path $(y_{\rm{right}},p_{\rm{right}})$ is continuated to $-\mathbb{R}_+$ and its image connects $(1;0)$ and the saddle point $(x_0,p_0)$. In this case define 
\begin{align*}
    S\rav&\{(y_{\rm{left}}(t),p_{\rm{left}}(t))\in\mathbb{R}_+\times\mathbb{R}_+\,|\,t\in\mathbb{R}\}
\cup\{(x_0,p_0)\}\\&\cup\{(y_{\rm{right}}(t),p_{\rm{right}}(t))\in\mathbb{R}_+\times\mathbb{R}_+\,|\,t\in\mathbb{R}\}\cup([1;\infty)\times\{0\}).
\end{align*}
Later  we will prove  that      the image  of $t\mapsto(\hjk{y},\hjk{p})$ is contained in $S$, in particular, 
$\hjk{y}$ is $y_{\rm{right}}$ if $x_0<x_*<1$, $y_{\rm{left}}$ if $x_0>x_*$, $x_0$ if $x_0=x_*$,  and  ${y}^0$ otherwise.

       In the case $\varrho\geq 2$ (see  Fig.~\ref{aaa1}(b)) consider the union of all images of the paths $(\check{y},\check{p})$ with positive $\check{p}$ that  intersect the axe $\{0\}\times(0;+\infty)$. This set is open and
       connected subset of $\mathbb{R}\times(0;+\infty)$. 
        Therefore the right boundary of this subset is the union of the images of paths.
         Since the stationary point $(1,0)$ is unstable and there exists no stationary point on $(0;\infty)^2$, this boundary is the image of a certain path $t\mapsto(y_{\uparrow},p_{\uparrow})$. So, in the case $\varrho\geq 2$ define 
$$S\rav\{(y_{\uparrow}(t),p_{\uparrow}(t))\in\mathbb{R}_+\times\mathbb{R}_+\,|\,t\in\mathbb{R}\}\cup([1;\infty)\times\{0\}).$$
We will prove  that      the image  of $t\mapsto(\hjk{y},\hjk{p})$ is contained in $S$, in particular, 
$\hjk{y}$ is $y_{\uparrow}$ if $x_*<1$ and  ${y}^0$ otherwise.

\begin{figure}[!ht]
	\begin{minipage}[h]{8cm}
		\center{\includegraphics[width=240px]{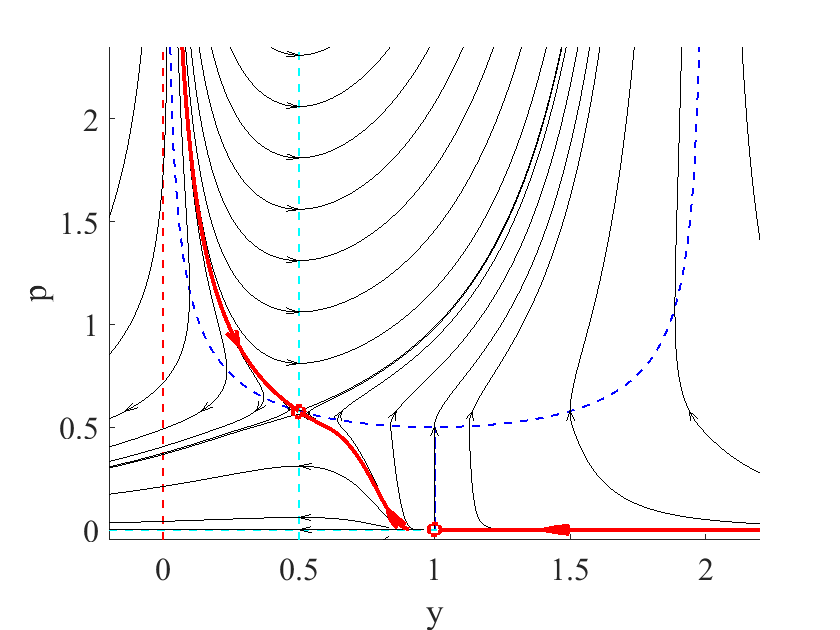}} \\(a) $\varrho=1$.
	\end{minipage}
	\hfill
	\begin{minipage}[h]{8cm}
		\center{\includegraphics[width=240px]{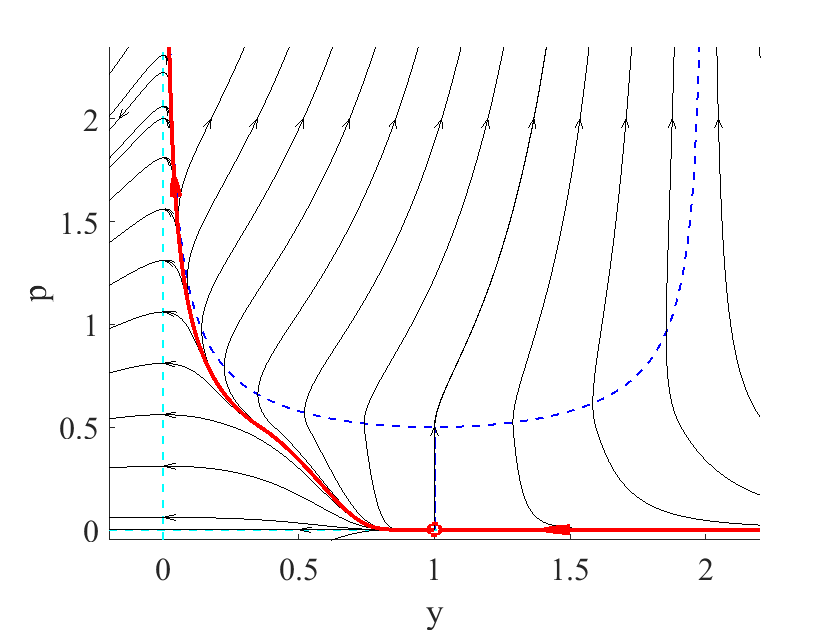}}\\ (b) $\varrho=2$.
	\end{minipage}
	\caption{The typical phase diagrams for solutions $(y(\cdot),p(\cdot))$ of \eqref{ds2} if (a) $\varrho\in(0;2)$; (b) $\varrho\geq 2$.}
	\label{aaa1}
\end{figure}
     
     At last, we begin to prove that  the image  of $t\mapsto(\hjk{y}(t),\hjk{p}(t))$ is contained in $S$ in 
     all five cases.
     Indeed, on the one hand, consider  any solution $(y,p)$ to  \eqref{ds2} such that $(x_*,p(0))$ is  to the left of  $S$. Then, $(y(t),p(t))$ is located to the left of  $S$ for all $t$ and, for large $t$, $p(t)$ is closed to zero and $y(t)$ is negative. Since $\hjk{p}$ must be nonnegative,  $(y,p)$ can't be $(\hjk{y},\hjk{p})$.
     On the other hand,  consider any solution $(y,p)$ to  \eqref{ds2} such that $(x_*,p(0))$ to the right of $S$. For all solutions $(\check{y},\check{p})$ at to $(y,p)$  its image is not intersects with $S$. In particular, the sign of $\check{y}$ as well as the sign of ${y}$ is positive for all positive $t$. It  means that the sign of $y$ is stable. So,  since $p$ is positive,  $(y,p)$  can't be $(\hjk{y},\hjk{p})$ again. Thus,
     we have proved that the image  of $t\mapsto(\hjk{y},\hjk{p})$ is contained in $S$. 
     
     So,  for all initial position $x_*>0$ and every discount rate $\varrho$ there can be at most one  locally weakly overtaking optimal process, because 
     there exists at most one
      solution to the Hamiltonian system satisfying the necessary conditions of Proposition~\ref{done}.
     Since these conditions are the direct consequence of \eqref{WAKKa}, for every
     locally weakly overtaking optimal process to this control problem the
     corresponding relations \eqref{sys_psi}-\eqref{maxH} with boundary condition \eqref{WAKKa} is a complete system of necessary
     conditions. 
     
     Finally, notice that for $\varrho\leq 0$ the motion ${y}^0$ and ${y}_{\downarrow}$ are overtaking optimal because $\hjk{u}=1$ for all sufficiently large $t$. Besides, in the case $\varrho\in(-2;0)$ (see Fig.~\ref{aaa-1}(a)) the constant motion $x_0$ as well as    converging  to  $x_0$ motions $y_{\rm{left}}$ and $y_{\rm{right}}$ can't be  locally weakly overtaking optimal.
     
	\end{example}
	
	The remaining examples illustrate the  borderline of consistent with \eqref{maxH} conditions on  co-state arcs  in the case of the lack of asymptotic constraints.
	Before  we proceed   further, we investigate system \eqref{sys_x}--\eqref{maxH},   corresponding to both examples.
	
	Fix a finite-dimensional real vector space $\mathbb{X}$ and a $C^2$-smooth function $S:\mathbb{R}\times\mathbb{X}\to\mathbb{R}$.  
	For  each positive $\theta$ and points  $x_*,\ x_\theta\in\mathbb{X}$,
	consider the following auxiliary problem:
	\begin{align*}      
	\textrm{minimize\ }            
	&\int_{0}^{\theta}\Big[ \frac{1}{4}\|u(\tau)\|^2+
	\frac{\partial S}{\partial x}\big({\tau},y(\tau)\big)e^{-\tau}u(\tau)+
	\frac{\partial S}{\partial t}\big({\tau},y(\tau)\big)\Big]\, d\tau\\
	\textrm{subject to\ }&\frac{d{y}(t)}{dt}=e^{-t}u(t),\ t\in[0;
	\theta]\ \textrm{a.e.},\\
	& y(0)=x_*,\quad y(\theta)=x_\theta,\quad y(t)\in\mathbb{X},\quad u(t)\in U\rav\mathbb{X}.
	\end{align*}
	
	\begin{subequations}
		Consider  a solution  $({\psi},{\lambda})$ (on $[0;\theta]$) to relations \eqref{sys_psi}--\eqref{maxH} 
		corresponding to a given admissible  
		process $({y},{u})$ of this  auxiliary problem.   It follows   from \eqref{sys_psi} that
		\[  \frac{d\psi(t)}{dt}={\lambda}
		\frac{\partial^2 S}{\partial x^2}\big({t},{y}(t)\big)e^{-t}u(t)+{\lambda}
		\frac{\partial S}{\partial t\partial x}\big({t},{y}(t)\big)=
		{\lambda}\frac{d}{dt}\Big(\frac{\partial S}{\partial x}\big({t},{y}(t)\big)\Big)\]
		for almost all $t\in[0;\theta]$. Then,  one can find a $C\in\mathbb{X}$ satisfying ${\psi}(t)=\lambda\frac{\partial S}{\partial x}(t,{y}(t))+{C}$ for all $t\in[0;\theta]$. 
		Because $u(t)$ minimizes 
		$\frac{\lambda}{4}\|\upsilon\|^2+
		e^{-t}\big(\lambda\frac{\partial S}{\partial x}({t},y(t))-\psi(t)\big)\upsilon=
		\frac{\lambda}{4}\|\upsilon\|^2-e^{-t}C\upsilon$ over all $\upsilon\in\mathbb{X}$,  we have $\lambda>0$. Then, putting $\lambda=1$, one has
		\begin{eqnarray}
		{\psi}(t)=\frac{\partial S}{\partial x}(t,{y}(t))+{C},\quad{u}(t)=2e^{-t}{C},\quad {y}(t)=x_*+(1-e^{-2t}){C},
		\label{946}\\
		J(x_*,0,{u};t)-S(t,y(t))+S(0,x_*)=\frac{1-e^{-2t}}{2}||{C}||^2=\frac{||{y}(t)-x_*||^2}{2(1-e^{-2t})}
		\label{953}
		\end{eqnarray}  
		for all $t\in[0;\theta]$. In particular, since $J(x,0,{u};t)-S(t,\rty(x,0,{u},t))+S(0,x)$ is independent of $x\in\mathbb{X}$, we also have   
		\begin{equation}\label{951}
		\frac{\partial J}{\partial x}(y(0),0,{u};t)=\frac{\partial S}{\partial x}(t,y(t))-
		\frac{\partial S}{\partial x}(0,x_*)=\psi(t)-\psi(0)\qquad \forall t\in[0;\theta].
		\end{equation}
		Note also that no co-state arc satisfies  relations  \eqref{sys_psi}--\eqref{maxH} with ${\lambda}=0$.
	\end{subequations} 
	
	Having all required formulae,
	let us finish considering this auxiliary problem and return to examples of infinite-horizon control problems.
	
	The following example will show that condition \eqref{AK} might  fail if       the gradient  of the limit of $\hjk{J}$ does not coincide with the limit of gradients of $\hjk{J}$  at $x_*$ (cf. \eqref{4004}).
	
	\begin{example}\label{aa}
		Assume that   the function  \(\mathbb{X}\ni x\mapsto\limsup_{\theta\uparrow\infty} S(\theta,x)\)
		attains its minimum  at the point $x_*=0$.

		Then, the assumptions $(E'_{\sup})$ and $(E_{\sup})$ hold true and the admissible control process $(\hjk{y},\hjk{u})=(0,0)$ is weakly overtaking optimal in the problem
		\begin{align*}        
		\textrm{minimize\ }            
		&\int_{0}^{\infty}\Big[ \frac{1}{4}\|u(\tau)\|^2+ 
		\frac{\partial S}{\partial x}\big({\tau},y(\tau)\big)e^{-\tau}u(\tau)+
		\frac{\partial S}{\partial t}\big({\tau},y(\tau)\big)\Big]\, d\tau\\
		\textrm{subject to\ }&\frac{d{y}(t)}{dt}=e^{-t}u(t)\ \textrm{a.e.},\\
		& y(0)=0,\ y(t)\in\mathbb{X},\ u(t)\in U=\mathbb{X}.
		\end{align*}
		Hence there exists a co-state arc $\hjk{\psi}\in C(\mathbb{R}_+,\mathbb{X}^*)$ satisfying the corresponding to $(\hjk{y},\hjk{u})$ relations \eqref{sys_psi}-\eqref{maxH} with $\hjk{\lambda}=1$; moreover, \eqref{946} holds for a  certain $\hjk{C}$.  It follows from $\hjk{u}\equiv 0$ that $\hjk{C}=0$ and $\hjk{\psi}(t)=\frac{\partial S}{\partial x}(t,0)$ for all positive $t\geq 0$.
		
		Consider the transversality condition \eqref{AK}. By
		$\frac{\partial \hjk{J}}{\partial x}(0;\theta)=\psi(\theta)-\psi(0)$,
		the co-state arc $\hjk{\psi}$ satisfies this condition iff the co-state arc
		$\hjk{\psi}(\theta)=\frac{\partial S}{\partial x}(\theta,0)$ converges to zero as $\theta\uparrow \infty$.
		Note that, in this example, condition~\eqref{AK} is equivalent  to Shell's condition \cite{shell,sagara}:  $\psi(\theta)\to 0$. By comparison,  $\frac{\partial S}{\partial x}(\theta,0)=0$ holds for all $\theta$ large enough iff  the
		Arrow-like condition \cite{ssbook,sagara} holds:
		\(\displaystyle \limsup\nolimits_{\theta\uparrow\infty}\hjk{\psi}(\theta)(\hjk{y}(\theta)-{y}(\theta))\geq 0\)
		for all admissible control processes $(y,u)$. 
		
		So, in Example~\ref{aa}, the corresponding to a weakly overtaking optimal process $(\hjk{y},\hjk{u})=(0,0)$ relations \eqref{sys_psi}--\eqref{maxH} have  a solution satisfying the transversality condition \eqref{AK} iff the function $S$ satisfies the following additional asymptotic assumption: the gradients
		$\frac{\partial S}{\partial x}(\theta,0)$ converge to zero as $\theta\uparrow\infty.$
		For instance, this  is false if we take \[\mathbb{X}\rav\mathbb{R},\qquad\mathbb{R}_+\times\mathbb{R}\ni(t,x)\mapsto S(t,x)\rav e^{-t}\sin(e^{t}x)-e^{-x^2}.\]
		
		In this case,  the cost functional $J$ and all its derivatives are bounded for all control processes. In particular, there could not be $\lambda=0$, i.e., 
		all solutions to the Pontryagin Maximum Principle are not abnormal whenever $x_*$. Furthermore, the corresponding infinite-horizon problem has the unique overtaking optimal  process 
		(moreover,  a unique strongly optimal \cite{car1}, a unique classical optimal \cite{slovak}, and a unique (O)-optimal \cite{stern}  process), this process has a unique (up to a positive factor) solution to  \eqref{sys_x}--\eqref{maxH}, and the corresponding transversality condition \eqref{AK} is well defined;
		after all, this condition is not  consistent with \eqref{maxH}. Thus, the necessity of \eqref{AK} depends primarily on the asymptotics of $\frac{\partial \hjk{J}}{\partial x}$ rather than the qualitative properties of system  \eqref{sys_x}--\eqref{maxH}.
	\end{example}

	The dynamics and the integrand  in the last example are linear in $x$, whilst the functions  $S$ and $\frac{\partial S}{\partial x}$ are periodic. It is enough that for a given $x_*$  in the corresponding infinite-horizon control problem the dimension of the set of optimal processes coincides with $\dim \mathbb{X}$. 
	By this reason, 
	their co-state arcs at $t=0$ also form a ball in $\mathbb{X}^*$.  Inclusion \eqref{WAKKa} as a unified necessary transversality condition  has to 
	take this into account.
	\begin{example}\label{1}
		Put $\mathbb{X}\rav\mathbb{R}^2$. Define the map $S$ as follows: 
		\[\mathbb{R}_+\times\mathbb{R}^2\ni(t,x_1,x_2)\mapsto S(t,x_1,x_2)\rav {x}_1\sin(t)-{x}_2\cos(t).\]  
		Then, the corresponding infinite-horizon control problem is	
		\begin{align*}    
		\textrm{minimize\ }            
		&\int_{0}^{\infty}\Big[ \frac{\|u(\tau)\|^2}{4}+(e^{-\tau}{u}_1(\tau)+{y}_2(\tau))\sin(\tau)\\
		&\qquad\qquad\qquad\qquad+({y}_1(\tau)-e^{-\tau}{u}_2(\tau))\cos(\tau)
		\Big]\, d\tau\\
		\textrm{subject to\ }&\frac{d{y_1}(t)}{dt}=e^{-t}u_1(t), \frac{d{y_2}(t)}{dt}=e^{-t}u_2(t)\ \textrm{a.e.\ }t\geq 0,\\
		&(y_1,y_2)(0)=x_*,\ (y_1,y_2)(t)\in\mathbb{R}^2,\ (u_1,u_2)(t)\in U=\mathbb{R}^2.
		\end{align*}
		
		Consider a weakly overtaking optimal process $(\hjk{y},\hjk{u})$   in this problem. Note that $(H0)$--$(H6)$ and $(E'_{\sup})$ are fulfilled.
		By Theorem~\ref{9}, there exists a nonzero solution $(\hjk{\psi},\hjk{\lambda})$ to system \eqref{sys_psi}--\eqref{maxH} and \eqref{WAKKc}, corresponding to this process. 
		Note that for each positive $\theta$, this pair is also the solution to this system for the auxiliary control problem  with $x_\theta=\hjk{y}(\theta).$
		Then, first, $\hjk{\lambda}$ is positive, and one can take $\hjk{\lambda}=1$; second, by \eqref{946},
		one find a  $\hjk{C}\in\mathbb{X}^*$ satisfying
		$\hjk{\psi}(t)=\frac{\partial S}{\partial x}(t,\hjk{y}(t))+\hjk{C}$ for all nonnegative $t$. 
		Hence, by \eqref{951}, condition \eqref{WAKKc} is equivalent to
		\[-\frac{\partial S}{\partial x}(0,x_*)-\hjk{C}=-\hjk{\psi}(0)\in \co\big\{(\sin\theta,-\cos\theta)\,\big|\,\theta\in[0;2\pi]\big\}-\frac{\partial S}{\partial x}(0,x_*),\] i.e.,
		$||\hjk{C}||\leq 1$.
		Thus, according to \eqref{WAKKc}, 
		every weakly overtaking optimal process  satisfies \eqref{946} with some $\hjk{C}$, $||\hjk{C}||\leq 1$.
		
		We claim that each process with this property is a weakly overtaking optimal. 
		Fix an admissible 
		process $(\bar{y},\bar{u})$
		satisfying \eqref{946} for some $\bar{C}$, $||\bar{C}||\leq 1$.
		Consider another admissible control process $(y,u)$.
		By  \eqref{953}, the function
		$J(x_*,0,\bar{u};\theta)-S(\theta,\bar{y}(\theta))+S(0,x_*)$  converges to $||\bar{C}||^2/2$  as $\theta\uparrow\infty$. 
		Then, it is necessary to find a $\phi\in[0;2\pi]$ satisfying 
		\[\limsup_{n\to\infty}\big[J(x_*,0,u;\phi+2\pi n)-S(\phi,x_*+\bar{C})+S(0,x_*)-{||\bar{C}||^2}/{2}\big]\geq 0.\]
		Note that $J(x_*,0,u;\theta)$ tends to infinity as $\theta\uparrow\infty$ if the total variation of $y$ is unbounded. Therefore, one can assume that the motion $y$ converges  to a $x_*+C_\infty\in\mathbb{X}$ as $t\to\infty$. 
		For every positive $\theta$, there exists an admissible control process $(y_\theta,u_\theta)$ optimal in the auxiliary problem with $x_\theta=y_\theta(\theta)$; 
		in particular, $J(x_*,0,u;\theta)\geq J(x_*,0,u_{\theta};\theta)$. Hence, by \eqref{953}, it follows from $y_\theta(\theta)\to C_\infty$ that one has
		\[\limsup_{n\to\infty}J(x_*,0,u;\phi+2\pi n)\geq S(\phi,
		x_*+C_\infty
		)-S(0,x_*)+
		{||C_\infty||^2}/{2}.\]  
		Then, 
		we must find a $\phi\in[0;2\pi]$ so that the number
		\[
		S(\phi,x_*+C_\infty)-S(\phi,x_*+\bar{C})+\frac{||C_\infty||^2-||\bar{C}||^2}{2}
		=S(\phi,C_\infty-\bar{C})+\frac{||C_\infty||^2-||\bar{C}||^2}{2}\]
		is nonnegative. Note that for each vector $z\in\mathbb{X}$, one finds a $\phi\in[0;2\pi]$ such that $S(\phi,z)=||z||$. Therefore, it is required  to prove only 
		$2||x-\bar{C}||\geq ||\bar{C}||^2-||x||^2=\big(\|\bar{C}\|-\|x\|\big)\big(\|\bar{C}\|+\|x\|\big)$ for all vectors $x=C_\infty\in\mathbb{R}^2$.
		The latter  is evident at all $x$ where $||x||+||\bar{C}||\leq 2$,  while at vectors $x$ where $||x||>2-||\bar{C}||$ it follows from $||\bar{C}||^2\leq 1\leq (2-||\bar{C}||)^2<||x||^2$ because of $||\bar{C}||\leq 1$. 
		
		Thus, it is shown that  every admissible control process satisfying \eqref{946} with some $\bar{C}$, $||\bar{C}||\leq 1$,
		is weakly overtaking optimal.  Then,  each co-state arc satisfying \eqref{WAKKa} generates the unique  weakly overtaking optimal process and each weakly overtaking optimal process  possesses its own co-state arc satisfying \eqref{WAKKa}.

		Summing up,  for every  weakly overtaking optimal process in this example, there exists a unique co-state arc satisfying \eqref{maxH}. Therefore, each  necessary condition on co-state arcs must either satisfy each of them
		or distinguish them. Due to  the linearity  of the dynamics and integrand in $x$, any  necessary condition depending only  on a co-state arc is satisfied  for all such co-state arcs. In this example, all such co-state arcs $\psi$ satisfy  $||\psi(0)+(0,1)||\leq 1$. Since
		inclusion \eqref{WAKKa} also points to the same set, in this example this inclusion    becomes  the tightest of all boundary conditions  on the co-state arcs consistent with \eqref{maxH}.
	\end{example}
	
	\section{Subdiffentials of lower and upper limits of scalar functions}
\label{e}

Let a family of maps ${W}_\theta$,  $\theta\geq 0$, from a finite-dimensional  
space   $\mathbb{X}$ to $\mathbb{R}\cup\{+\infty\}$ be given.
Define the maps $W_{\inf}$  and $W_{\sup}$ from $\mathbb{X}$ to $\mathbb{R}\cup\{-\infty,+\infty\}$ by the rules
\begin{equation}\label{infsup}W_{\inf}(x)\rav \liminf_{\theta\uparrow\infty} {W}_\theta (x),\qquad
W_{\sup}(x)\rav \limsup_{\theta\uparrow\infty} {W}_\theta (x)
\qquad\forall x\in\mathbb{X}.
\end{equation}
 
 For every nonempty subset 
 ${{\mathcal{X}}}$  of a finite-dimensional  
space   $\mathbb{X}$, define the maps 
$\lscinf{\mathcal{X}} W$ and $\lscsup{{\mathcal{X}}} W$ from $\mathbb{X}$ to $\mathbb{R}\cup\{-\infty,+\infty\}$ as follows:
\begin{equation}\label{infsupOmega}
(\lscinf{\mathcal{X}} W)(\check{x})\rav \lsc \liminf_{\substack{\theta\uparrow\infty,\\ {{\mathcal{X}}}\ni x\to \check{x}}} {W}_{\theta} (x),\quad 
(\lscsup{\mathcal{X}} W)(\check{x})\rav \lsc \limsup_{\substack{\theta\uparrow\infty,\\ {{\mathcal{X}}}\ni x\to \check{x}}} {W}_{\theta} (x)
\end{equation}
if $\check{x}\in \cl {{\mathcal{X}}}$ and $+\infty$ otherwise. Note that both maps are  lower semicontinuous.

First, consider subdifferentials of the lower epigraphical limit of $W_{\theta}$. 

Following \cite[Sect. 7.3]{RW}, for every sequence of $g_n:{\mathbb{X}}\to\mathbb{R}\cup\{-\infty,+\infty\}$, denote by $\displaystyle\eliminf_{n\uparrow\infty} g_n$ 
	and $\displaystyle\elimsup_{n\uparrow\infty} g_n$ 
	the epigraphical lower and upper
	limits of $g_n$: 
	for all $x\in\mathbb{X}$
	\begin{align*}
	\eliminf_{n\uparrow\infty} g_n(x)&\rav \lim_{\varkappa\downarrow
		0}\liminf_{n\uparrow\infty}\inf_{||x-z||\leq \varkappa}g_n(z)\in\mathbb{R}\cup\{-\infty,+\infty\},\\
	\elimsup_{n\uparrow\infty} g_n(x)&\rav\lim_{\varkappa\downarrow 0}\limsup_{n\uparrow\infty}\inf_{||x-z||\leq \varkappa}g_n(z)\in\mathbb{R}\cup\{-\infty,+\infty\}.
	\end{align*}
	When these two functions coincide, we say  that the functions $g_n$ epi-converge  to 
	$$\elim_{n\uparrow\infty} g_n\rav\eliminf_{n\uparrow\infty} g_n=\elimsup_{n\uparrow\infty} g_n.$$
	Note that the epigraphical lower and upper limits is lower semicontinuous \cite[Proposition 7.4(a)]{RW}.
	

The first lemma is the consequence of the corresponding results of \cite{slope}. 
\begin{lemma}\label{mymy19}
	Let  for a given neighborhood  ${G}\subset\mathbb{X}$ maps ${W}_\theta:\mathbb{X}\to\mathbb{R}$  be lower semicontinuous on $G$ for all positive $\theta$. 
	Assume that 
	\begin{align}\label{hitro1}
	\eliminf_{\theta\uparrow\infty} W_\theta(z)= W_{\inf}(z)\qquad \forall z\in {G}.
	\end{align}
	
	Then, for all $\check{z}\in {G}$ satisfying $|W_{\inf}(\check{z})|<+\infty$, one has 
	\begin{equation}\label{14666}
	{\partial} W_{\inf}(\check{z})\subset
	\bigcap_{\substack{(\theta_n)_{n\in\mathbb{N}}\in\mathbb{R}_+^{\mathbb{N}},\ \theta_n\uparrow\infty,\\ 
			\dist( \check{z},W_{\inf}(\check{z});\Gr W_{\theta_n})\to 0}}
	\Limsup_{\substack{n\uparrow\infty,\ z\to\check{z},\\ |W_{\theta_n}(z)-W_{\theta_n}(\check{z})|\to 0}
	}
	\hjk{\partial} W_{\theta_n}(z),
	\end{equation}
	i.e., for every limiting subgradient
	$\xi\in{\partial} W_{\inf}(\check{z})$ and 
	 every unbounded increasing sequence of positive numbers $\theta_n\in\mathbb{R}_+$ satisfying 
	$\dist( \check{z},W_{\inf}(\check{z});\Gr W_{\theta_n})\to 0$
	there exists sequences of
	points $z_n\in \mathbb{X}$ and gradients
	$\zeta_n\in\hjk{\partial} W_{\theta_n}(z_n)$  satisfying
	  $\|z_n-\check{z}\|\to 0$, $\|\zeta_n-\check{z}\|\to 0$, and 
	$|W_{\theta_n}(z_n)-W_{\theta_n}(\check{z})|\to 0$ as $n\uparrow\infty$.
\end{lemma}
\doc 
	Consider a point $\check{z}\in {G}$.
	
	Let $\mathcal{T}(\check{z})$ be the set of all unbounded increasing sequences of positive $\theta_n$ such that one has
	$\dist( \check{z},W_{\inf}(\check{z});\Gr W_{\theta_n})\to 0$ and the sequence of maps
	\begin{equation}
	\label{555535}
	\mathbb{X}\ni z\mapsto\ W_{\theta_n}(z)
	\end{equation}
	epi-converges. Fix such a sequence of $\theta_n$, consider the corresponding epi-limit 
	\begin{equation}\label{1265}
	\mathbb{X}\ni z\mapsto\widetilde{W}_{\{\theta_n\}}(z)\rav\elim_{n\uparrow\infty} W_{\theta_n}(z);
	\end{equation}
	this function is lower semicontinuous
	by \cite[Proposition 7.4(a)]{RW}.
	
	 Fix a subgradient
	$\xi\in\hjk{\partial} W_{\inf}(\check{z})$. 
	By \cite[Theorem~1.27]{mord5}, 
	since $\mathbb{X}$ is finite-dimensional, 
	there exists a $C^1$-smooth function $g:\mathbb{X}\to\mathbb{R}$ and a neighbourhood $G'\subset G$ of the point $\check{z}$
	such that one has $\xi=\frac{\partial g}{\partial z}(\check{z})$, ${W}_{\inf}(\check{z})=g(\check{z})$, and $\lsc{W}_{\inf}(z)-\lsc{W}_{\inf}(\check{z})\geq g(z)-g(\check{z})$  for all $z\in G'$.

	 Now, on the one hand, \eqref{hitro1} entails  	
	$\widetilde{W}_{\{\theta_n\}}({z})\geq W_{\inf}({z})$ for all $z\in G$; on the other hand,
	$\dist( \check{z},W_{\inf}(\check{z});\Gr W_{\theta_n})\to 0$
	leads to $\widetilde{W}_{\{\theta_n\}}(\check{z})\leq W_{\inf}(\check{z})$.
	Then, one has $\widetilde{W}_{\{\theta_n\}}(\check{z})=W_{\inf}(\check{z})=g(\check{z})$ and $\widetilde{W}_{\{\theta_n\}}(z)-\widetilde{W}_{\{\theta_n\}}(\check{z})\geq g(z)-g(\check{z})$ for all
	$z\in G'$. By \cite[Theorem~1.27]{mord5}, 
	$\xi=\frac{\partial g}{\partial z}(\check{z})$ lies in $\hjk{\partial}\widetilde{W}_{\{\theta_n\}}(\check{z})$.

	Recall that the sequence of functions  $W_{\theta_n}$ is 
	asymptotically locally equicoercive \cite[Definition~4.4]{evol}  	
	if for a bounded sequence of $x_k\in\mathbb{X}$, from $$\sup_{k\in\mathbb{N}} W_{\theta_{n(k)}}(x_k)<+\infty$$  for a given unbounded sequence of $n(k)$, it follows that there exists a converging subsequence of the sequence of $x_k.$
	Then,
	the sequence of $W_{\theta_n}$ has this property,  since  $\mathbb{X}$ is finite-dimensional.  
	According to \cite[Proposition 2.2(v) and Theorem~5.3(ii)]{slope}, it follows  that
	\[\xi\in\hjk{\partial} \widetilde{W}_{\{\theta_n\}}(\bar{z})\subset
	\Limsup_{{z\to \bar{z},\ n\uparrow\infty,\ W_{\theta_n}(z)\to \widetilde{W}_{\{\theta_n\}}(\check{z})}}
	\hjk{\partial} {W}_{\theta_n}(z)\qquad \forall \bar{z}\in {G}\]
	holds 
	for all sequences
	$(\theta_n)_{n\in\mathbb{N}}\in\mathcal{T}(\check{z})$ because of the epigraphical convergence  of \eqref{555535} to $\widetilde{W}_{\{\theta_n\}}$.
	Hence,
	the convergence     	 $W_{\theta_n}(\check{z})\to \widetilde{W}_{\{\theta_n\}}(\check{z})=
	{W}_{\inf}(\check{z})$ with
	$\xi\in\hjk{\partial}\widetilde{W}_{\{\theta_n\}}(\check{z})$
	gives
	\begin{equation} \label{573}    
	\xi\in
	\bigcap_{(\theta_n)_{n\in\mathbb{N}}\in\mathcal{T}(\check{z})}
	\Limsup_{\substack{n\uparrow\infty,\ z\to\check{z},\\ |W_{\theta_n}(z)-W_{\inf}(\check{z})|\to 0}
	}\hjk{\partial}{W}_{\theta_n}(z).
	\end{equation}
	Since the right side of this inclusion is upper semicontinuous, we obtain that \eqref{573} holds true for all 
	$\xi\in{\partial} W_{\inf}(\check{z})$.
	
	To prove \eqref{14666},   
	suppose to the contrary  that one can find a positive  $\epsi$ and an        
	unbounded increasing sequence of positive $\theta_n$ satisfying 
	$$\dist( \check{z},W_{\inf}(\check{z});\Gr W_{\theta_n})\to 0$$
	such that
	one has
	$
	|W_{\theta_n}({z}_n)-W_{\inf}(\check{z})|+||\zeta_n-\xi||+||z_n-\check{z}||>\epsi$ 	for all  sequences of
	points $z_n\in {G}$  and gradients $\zeta_n\in\hjk{\partial}{W}_{\theta_n}(z_n)$. 
	Then, the sequence of $\theta_n$ (and every its subsequence) would not satisfy \eqref{573}.
	Hence the sequence of maps \eqref{555535} would not have any epi-converging subsequence.
	However,
	since  $W_{\inf}(\check{z})$ is finite, the  sequence of ${W}_{\theta_n}$ does not escape epigraphically to the horizon, and,
	by \cite[Theorem 7.6]{RW}, possesses an
	epi-converging subsequence. We get a contradiction.  
	
	Thus,  Lemma \ref{mymy19} has been proved. 
\bo 
	

\begin{lemma}\label{mymy20}
	    Let    $C$  be a nonempty subset of $\mathbb{X}$.
		Let  
		maps ${W}_\theta:\mathbb{X}\to\mathbb{R}$  be  $1$-Lipschitz continuous on $\mathbb{X}$.

	Then, for all $\check{z}\in C$ satisfying $|(\lscinf{C} W)(\check{z})|<+\infty$, one has 
	\begin{equation}\label{14667}
	{\partial} (\lscinf{C} W)(\check{z})\subset N(\check{z};\cl C)+
		\bigcap_{\substack{(\theta_n)_{n\in\mathbb{N}}\in\mathbb{R}_+^{\mathbb{N}},\ \theta_n\uparrow\infty,\\ 
		|(\lscinf{C} W)({z})-W_{\theta_n}(z)|\to0}}
	\Limsup_{\substack{n\uparrow\infty,\ z\to\check{z},\\ |W_{\theta_n}(z)-W_{\theta_n}(\check{z})|\to 0}
	}
	\hjk{\partial} W_{\theta_n}(z).
	\end{equation}
\end{lemma}	
\doc 	
      
    First, note that
    1-Lipschitz continuity of all the $W_\theta$ as well as $W_{\inf}(z)$ yields that
 $$\lsc(W_\theta+{\imath}_C)=W_\theta+\lsc{\imath}_C=W_\theta+{\imath}_{\cl C},\quad\lsc(W_{\inf}+{\imath}_C)=W_{\inf}+{\imath}_{\cl C}.$$
     In particular, for all $z\in {\cl C}$ one has
     \begin{align}\label{hitro22}
     (\lscinf{C} W)({z})= 
	\eliminf_{\theta\uparrow\infty} (W_\theta(z)+{\imath}_C)(z)=
	W_{\inf}(z)+{\imath}_{\cl C}(z).
	\end{align}
	Furthermore, $\dist(z,(\lscinf{C} W)({z});\Gr (W_{\theta_n+\imath_{C}}))\to0$ iff
	$|(\lscinf{C} W)({z})-(W_{\theta_n}+\imath_{C}))(z)|\to0$ iff
	$|(\lscinf{C} W)({z})-W_{\theta_n}(z)|\to0$ for all $z\in C$.
	
	 Second, for all positive $\theta$, the fuzzy sum rule \cite[Theorem~3.3.3]{BorZhu} leads to
	 \[\hjk{\partial}(W_\theta+{\imath}_{\cl C})({z})\subset
	 \Limsup_{\cl C\ni\hjk{z}\to{z}}
	 \hjk{N}(\hjk{z};\cl C)+\Limsup_{\hjk{z}\to{z}}
	 \hjk{\partial}W_\theta(\hjk{z})\qquad\forall {z}\in \cl C.\]
	 Furthermore, by \cite[Theorem~1.6]{mord5} we also have
	 \begin{equation} \label{4220}
	 \hjk{\partial}(W_\theta+{\imath}_{\cl C})({z})\subset
	 N({z};\cl C)+\Limsup_{\substack{\hjk{z}\to {z},\\ |W_{\theta}(\hjk{z})-
	 		W_{\theta}({z})|\to 0}}
	 \check{\partial}W_\theta(\hjk{z})\qquad\forall {z}\in \cl {C}. 
	 \end{equation}

       At last, note that \eqref{hitro22} is condition    \eqref{hitro1} with $W_\theta+{\imath}_{\cl C}$ instead of $W_\theta$. Now, for all every $\check{z}\in C$,
inclusion \eqref{14666} with $W_\theta+{\imath}_{\cl C}$ instead of $W_\theta$ is
\begin{align*}{\partial} (\lscinf{C}W)(\check{z})&\subset
\bigcap_{\substack{(\theta_n)_{n\in\mathbb{N}}\in\mathbb{R}_+^{\mathbb{N}},\ \theta_n\uparrow\infty,\\ 
|(\lscinf{C} W)({z})-W_{\theta_n}(z)|\to0}}
\Limsup_{\substack{n\uparrow\infty,C\ni z\to\check{z},\\ |W_{\theta_n}(z)-W_{\theta_n}(\check{z})|\to 0}
}
\hjk{\partial} (W_{\theta}+{\imath}_{\cl C})(z).
\end{align*}    
Applying 
\eqref{4220}, we obtain \eqref{14667}.
	Lemma \ref{mymy20} has been proved. 
\bo 

	In \cite[Theorem 6.1(i)]{trieman},  the following upper-estimate of Fr\'echet subdifferentials of  the upper epigraphical limit of lower semicontinuous functions $W_\theta$ was shown:
\begin{align*}
\hjk{\partial} \Big(\elimsup\nolimits_{\theta\uparrow\infty} W_\theta\Big)(\check{z})&\subset \co\Limsup_{\theta\uparrow\infty, z\to \check{z}} \hjk{\partial}{W}_{\theta}(z)\\&= \bigcap_{\varkappa>0}\co\cl\bigcup_{\theta>1/\varkappa,\ {z\in \mathbb{X},\ ||z-\check{z}||<\varkappa}}
\hjk{\partial}{W}_{\theta}(z) \qquad \forall\check{z}\in\mathbb{X}.	
\end{align*}
Repeating the proof of  the previous Lemma  word-for-word and using this inclusion  instead to  \eqref{14666}, we 
obtain the following result.   
\begin{lemma}\label{mymy50}
	    Let    $C$  be a nonempty subset of $\mathbb{X}$.
		Let  
		maps ${W}_\theta:\mathbb{X}\to\mathbb{R}$  be  $1$-Lipschitz continuous on $\mathbb{X}$.
	
	Then, for all $\check{z}\in C$  satisfying $|(\lscsup{C} W)(\check{z})|<+\infty$, one has 
\begin{align}\label{1466}
{\partial} \big(\lscsup{C} W\big)(\check{z})\subset \co N(\check{z};\cl C)+\co\Limsup_{\theta\uparrow\infty, z\to \check{z}} \hjk{\partial}{W}_{\theta}(z).
\end{align}
\end{lemma}
 
	The following lemma is the keystone of this section.
	
	\begin{lemma}\label{key}
Let a family of nonempty closed  subsets $\Omega_\theta\subset \mathbb{X}$,  $\theta> 0$, 
be given.
Consider also the closed sets 
$\Omega_{\inf}$  and
$\Omega_{\sup}$ defined as follows:
\begin{eqnarray*}
\Omega_{\inf}\rav \Liminf_{\theta\uparrow\infty} \Omega_\theta=\cl\bigcup_{\theta>0}\bigcap_{\theta'>\theta}\Omega_{\theta'},\qquad
\Omega_{\sup}\rav \Limsup_{\theta\uparrow\infty} \Omega_\theta=\bigcap_{\theta>0}\cl \bigcup_{\theta'>\theta}\Omega_{\theta'}.\end{eqnarray*}
 
  Then, for every subset $C\subset \mathbb{X}$ and point $\check{z}\in C$
  \begin{align}
  \label{eq:0co}
    {N}(\check{z};\Omega_{\inf}\cap C)&\subset \co {N}(\check{z};\cl C)+\co\Limsup_{{z\to\check{z},\theta\uparrow\infty}}   \hjk{N}({z};\Omega_{\theta})\ \textrm{if}\ \check{z}\in\Omega_{\inf};\\
  \label{eq:0cap}
  {N}(\check{z};\Omega_{\sup}\cap C)&\subset {N}(\check{z};\cl C)+
 \bigcap_{\substack{(\theta_n)_{n\in\mathbb{N}}\in\mathbb{R}_+^{\mathbb{N}},\; \theta_n\uparrow\infty,\\ 
		\dist(\check{z};\Omega_{\theta_n})\to  0}}
 \Limsup_{{n\uparrow\infty,\;z\to\check{z}}
}  \hjk{N}({z};\Omega_{\theta_n})\ \textrm{if}\ \check{z}\in\Omega_{\sup}.
  \end{align}
	\end{lemma}
		\doc 
Define $W_\theta\rav\dist(\cdot;\Omega_{\theta})$ for all positive  $\theta$.
For all $z\in \mathbb{X}$ these functions satisfy the following equalities:
\begin{align*}
    \dist(z;\Omega_{\inf}\cap C)=\elimsup_{\theta\uparrow\infty}\dist(z;\Omega_{\theta}\cap C)=(\lscsup{C} W)(z),
    \\
\dist(z;\Omega_{\sup}\cap C)=\eliminf_{\theta\uparrow\infty}\dist(z;\Omega_{\theta}\cap C)=(\lscinf{C} W)(z).
\end{align*}
Now, from \eqref{14667} and \eqref{1466} it follows that 
\begin{align}
{\partial} \dist(\check{z};\Omega_{\inf}\cap C)&\subset N(\check{z};\cl C)
+
\bigcap_{\substack{(\theta_n)_{n\in\mathbb{N}}\in\mathbb{R}_+^{\mathbb{N}},\ \theta_n\uparrow\infty,\\ 
		\dist(\check{z};\Omega_{\theta_n})\to  0}}
 \Limsup_{\substack{n\uparrow\infty,\\ z\to\check{z}}
} \hjk{\partial}{W}_{\theta_n}(z)\quad \forall\check{z}\in\Omega_{\inf};\\
{\partial} \dist(\check{z};\Omega_{\sup}\cap C) &\subset \co N(x;\cl C)+\co\Limsup_{\theta\uparrow\infty, z\to \check{z}} \hjk{\partial}{W}_{\theta}(z) \qquad \forall\check{z}\in\Omega_{\sup}.
\end{align}
Now, \eqref{eq:0cap} and \eqref{eq:0co} follow from equalities \eqref{133}.

  The key lemma has been proved.
\bo 

The next two lemmata do not follow from similar results in \cite{trieman,mordukhovich2013subdifferentials,perez2019subdifferential}.
	\begin{lemma}\label{cap}
		Let  maps ${W}_\theta:\mathbb{X}\to\mathbb{R}\cup\{-\infty,+\infty\}$ be lower semicontinuous for all positive $\theta$.
 Given  a subset ${{\mathcal{X}}}\subset\mathbb{X}$, a  point $\check{z}\rav(\check{x},\check{y})\in \epi (\lscinf{\mathcal{X}} W)$, and 
		a normal $(\zeta,-\lambda)$ in ${N}(\check{z};\epi (\lscinf{\mathcal{X}} W))$.  
     Assume also that $\check{y}=W_\theta(\check{x})$	for all positive $\theta$.

\begin{subequations}
		Then:
	\begin{list}{}{}
\item{(i)}
\begin{align}     \label{a:0cap}
(\zeta,-\lambda)\in {N}(\check{x};\cl \mathcal{X})\times\{0\}+
\bigcap_{{(\theta_n)_{n\in\mathbb{N}}\in\mathbb{R}_+^{\mathbb{N}}, \theta_n\uparrow\infty}}
\Limsup_{n\uparrow\infty,\ z\to\check{z}}
\hjk{N}({z};\epi W_{\theta_n}).
\end{align}
\item{(ii)}
Furthermore, 
\begin{align}     \label{c:0cap}
	\zeta\in {N}(\check{x};\cl\mathcal{X})+
\bigcap_{{(\theta_n)_{n\in\mathbb{N}}\in\mathbb{R}_+^{\mathbb{N}}, \theta_n\uparrow\infty
}}\Limsup_{\substack{n\uparrow\infty,\;x\to\check{x},\;\lambda'\downarrow\lambda,\\ W_{\theta}(x)\to\check{y}}}
	\lambda'\hjk{\partial} W_{\theta_n}(x)
\end{align}
if    all the $W_\theta$ are Lipschitz continuous on a given neighbourhood $G$ of $\check{x}$.
\end{list}
\end{subequations}     	
\end{lemma}	
\doc 
	 At the beginning, consider the case when $W_{\inf}$ is continuous.

 Put $C\rav\mathcal{X}\times\mathbb{R}$ and $\Omega_\theta\rav\epi W_\theta$ for every positive $\theta$. 
	Then, the continuity of $W_{\inf}$ gives
 $\Omega_{\sup}\cap \cl C=\epi (\lscinf{\mathcal{X}} W)$. 
 Now, by Lemma~\ref{key}, for 
  every point $\check{z}\rav(\check{x},\check{y})\in \epi (\lscinf{\mathcal{X}} W)$, inclusion
  \eqref{eq:0cap} holds. Hence,
   every normal $(\zeta,-\lambda)$ in ${N}(\check{z};\epi (\lscinf{\mathcal{X}} W))$ satisfies \eqref{a:0cap}. So, 
   \eqref{a:0cap} has been proved.

To prove \eqref{c:0cap}, note that  the cone 
$\hjk{N}(x,g(x);\epi g)$ of  a Lipschitz continuous function $g$ is $\cup_{r\geq  0}r(\hjk{\partial}g(x) \times\{-1\})$.
 Also 
 $\hjk{N}(z;\epi g)$ is contained in $\{0\}$ if the point $z$ does not lie in the graph of $g$.
  Hence,
 we may rewrite the inclusion \eqref{a:0cap} as
 \begin{align*}
 (\zeta,-\lambda)\in {N}(\check{x};\cl\mathcal{X})\times\{0\}+
\bigcap_{\substack{(\theta_n)_{n\in\mathbb{N}}\in\mathbb{R}_+^{\mathbb{N}},\\ \theta_n\uparrow\infty
}}\Limsup_{\substack{n\uparrow\infty,\;x\to\check{x},\;\lambda'>0,\\ W_{\theta}(x)\to\check{y}}}
\lambda'(\hjk{\partial} W_{\theta_n}(x)\times\{-1\}).
 \end{align*}
 Hence a sequence of   $\lambda'_n$ has to converge to $\lambda$ and \eqref{c:0cap} has been proved.

In the general case, consider maps $\widetilde{W}_\theta=\dist(\cdot;\epi W_\theta)$. These maps as well $\widetilde{W}_{\inf}$ are
1-Lipschitz continuos. Therefore, inclusion \eqref{c:0cap} holds for $\lscinf{C} \widetilde{W}$ with $\lambda'=1$. The substitution of \eqref{133} gives \eqref{a:0cap} for $\lscinf{\mathcal{X}} W$. Repeating the previous paragraph literally, we obtain
\eqref{c:0cap}.
\bo 
	\begin{lemma}\label{co}
	Let  maps ${W}_\theta:\mathbb{X}\to\mathbb{R}\cup\{-\infty,+\infty\}$ be lower semicontinuous for all positive $\theta$.
	Given  a subset ${{\mathcal{X}}}\subset\mathbb{X}$, a  point $\check{z}\rav(\check{x},\check{y})\in \epi (\lscsup{\mathcal{X}} W)$, and 
	a normal $(\zeta,-\lambda)$ in ${N}(\check{z};\epi (\lscsup{\mathcal{X}} W))$.  

\begin{subequations}
	Then:
	\begin{list}{}{}
		\item{(i)}
		\begin{align}     \label{a:0co}
			(\zeta,-\lambda)\in \co{N}(\check{x};\cl\mathcal{X})\times\{0\}+\co
			\Limsup_{{z\to\check{z},\  \theta\uparrow\infty}}
			\hjk{N}({z};\epi W_{\theta}).
		\end{align}
		\item{(ii)}
		Furthermore, 
		\begin{align}     \label{c:0co}
			\zeta\in \co{N}(\check{x};\cl\mathcal{X})+\co
		\Limsup_{\substack{\theta\uparrow\infty,\ x\to\check{x},\ \lambda'\to\lambda,\\W_{\theta}(x)\to\check{y}}}
			\lambda'\hjk{\partial} W_{\theta}(x)
		\end{align}
		if    all the $W_\theta$ are Lipschitz continuous on a given neighbourhood $G$ of $\check{x}$.
		\item{(iii)}
		\begin{align}     \label{d:0co}
	\zeta\in {N}(\check{x};\cl\mathcal{X})+\lambda\co
	\Limsup_{\substack{\theta\uparrow\infty,\ x\to\check{x},\\W_{\theta}(x)\to\check{y}}}
	\hjk{\partial} W_{\theta}(x)
\end{align}
if  there exists a neighbourhood $G$ of $\check{x}$ and a positive $R$ such that  all the $W_\theta$ are $R$-Lipschitz continuous on ${G}$.
	\end{list}	
\end{subequations}	
\end{lemma}	
\doc 

      The proof of  \eqref{a:0co} and \eqref{c:0co} 
       follows from \eqref{eq:0co} as well  \eqref{a:0cap} and \eqref{c:0cap} from \eqref{eq:0cap}.

     To prove  \eqref{d:0co}, 
      note that in the case of $R$-Lipschitz continuity of all the $W_\theta|_{G}$, 
      the functions
      $\lsc W_{\theta}|_{G}= W_{\theta}|_{G}$ and     
      $W_{\sup}|_{G}=\big(
      \lscsup{\mathcal{X}} W\big)|_{G}$ have the same property. It follows that
      $\lscsup{\mathcal{X}} W$ coincides with ${\imath}_{\cl\mathcal{X}}+W_{\sup}$ on $G$.
      In addition, since every singular limiting subgradient of 
Lipschitz continuous function $W_{\sup}|_{G}$ is zero,   the subdifferential qualification condition \cite[(2.34)]{mord5} for $W_{\sup}$ and ${\imath}_{\cl\mathcal{X}}$ is applied. Now, by \cite[(2.35) and (2.36)]{mord5} we obtain that 
\[\partial^{\infty}(W_{\sup}+{\imath}_{\cl\mathcal{C}})|_{G}=\partial^{\infty}{\imath}_{\cl\mathcal{C}}|_{G}=
N(\cdot;\cl\mathcal{C})|_{G},\ \partial(W_{\sup}+{\imath}_{\cl\mathcal{C}})|_{G}=\partial W_{\sup}|_{G}.\]
 It means that 
  $$N(\check{x};\epi (\lscsup{\mathcal{X}} W))\!=\!N(\check{x};\epi ({\imath}_{\cl\mathcal{X}}+W_{\sup}))\!=\!
  N(\check{x};\cl{\mathcal{X}})\times\{0\}+N(\check{x},\check{y};\epi W_{\sup}).$$
     Applying \eqref{a:0co} for the function $W_{\sup}|_{G}=(\lscsup{G} W)|_{G}$ and
     using the boundedness of the norms of all Fr\'{e}chet subdifferentials of $W_\theta|_{G}$,
  we obtain 
	\begin{align*}     
		(\zeta,-\lambda)\in N(\check{x};\cl{\mathcal{X}})\times\{0\}+\lambda\co
		\Limsup_{\substack{x\to\check{x},\  \theta\uparrow\infty\\ W_{\theta}({x})\to\check{y}}}
		\hjk{\partial}W_{\theta}({x})\times\{-1\}
	\end{align*}
	for all normal $(\zeta,-\lambda)$ in ${N}(\check{z};\epi
	(\lscsup{\mathcal{X}} W))$.  
	So, \eqref{d:0co} has verified.
	
	The final lemma has been proved.
\bo

	\section{Proofs of Theorems~\ref{9} and \ref{8}}
	\label{ff}
	
	We will prove both theorems simultaneously. The idea of the proof is briefly described as follows.
	At the beginning, following  Halkin's method \cite{Halkin} and ideas of \cite[Section 25.2]{ClarkeNew}, we will reduce the verification of all desired relations on $\mathbb{R}$ with set $U$  to  its verification on finite intervals $[0;n]$ with each finite subset $\Upsilon_n(t)$ of $U$.
	Then, for fixed $n$,  extending $f$ and $f_0$ from $[0;n]$ to $[0;2n]$ by back-track, we will consider a new dynamics function $\bar{f}_{n}$ and integrand $\bar{f}_{0,n}$; this trick ensures the transfer of the right endpoint to $\hjk{y}(0)$.  Next, we will rewrite the corresponding   criteria (the weakly overtaking criterion in Theorem~\ref{9} and the
	overtaking criterion in Theorem~\ref{8}) in terms of a constrained optimisation problem in a finite-dimensional space. Further, we will reduce this problem to some finite-horizon control problem. Proceeding to Halkin's method, we will consider the corresponding Pontryagin Maximum Principle \cite[Theorem 22.26]{ClarkeNew} for this control problem and estimate the happened transversality condition by Lemma~\ref{co} and Lemma~\ref{cap} in the proof of Theorem~\ref{9} and the proof of Theorem~\ref{8}, respectively.
	
	
	\newcorr{ {\bf Step 0.}}

	\corr{
	Consider maps $\mathbb{X}\ni x\mapsto W_\theta(x)\rav \hjk{J}(x;\theta)-\hjk{J}(\hjk{y}(0);\theta)$ for each $\theta>0$.
	Define their partial limits $W_{\sup}$, $W_{\inf}$ from $\mathbb{X}$ to $\mathbb{R}\cup\{-\infty,+\infty\}$ by rule \eqref{infsup}.}
 \newcorr{Set} 
	\begin{align*}
\corr{W_{\textrm{extr}}\rav W_{\sup}\textrm{\ and\ }}\newcorr{(E_{\textrm{extr}})\equiv(E_{\textrm{sup}}) }&\newcorr{\textrm{\ in the proof of Theorem~\ref{9}},}\\
	\corr{W_{\textrm{extr}}\rav W_{\inf}\textrm{\ and\ }}\newcorr{(E_{\textrm{extr}})\equiv(E_{\textrm{inf}}) }&\newcorr{\textrm{\ in the proof of Theorem~\ref{8}.}}
	\end{align*}
   \newcorr{Thus, hypothesis $(E_{\textrm{extr}})$ holds true in both proofs.}

\newcorr{  Notice that $(E'_{\textrm{extr}})$ entails $(E'''_{\textrm{extr}})$ by $\hjk{y}(0)\in \mathcal{C}_{\as}\cap \mathcal{C}_0$ and  $i_{\mathcal{C}_{\as}},i_{\mathcal{C}_{0}}\geq 0$. Furthermore,
     $(E''_{\textrm{extr}})$ with conditions of Theorems~1 and 2, respectively, entails $(E'''_{\textrm{extr}})$. Indeed,
   the local weakly overtaking optimality criterion is assumed in Theorem~1 as well as the local overtaking optimality criterion is assumed in Theorem~2.
   Each criterion leads to  the lower semicontinuity of the corresponding $W_{\textrm{extr}}+i_{\mathcal{C}_{\as}}+i_{\mathcal{C}_0}$ at $\hjk{y}(0)$.
   Since $i_{\mathcal{C}_{0}}$ is zero  around each $x\in\internary \mathcal{C}_0$, $W_{\textrm{extr}}+i_{\mathcal{C}_{\as}}$ is lower semicontinuous at $\hjk{y}(0)$ as well.
   Therefore, 
 assumption $(E''_{\textrm{extr}})$ yields $(E'''_{\textrm{extr}})$. So, we may assume that condition $(E'''_{\textrm{extr}})$ holds true in both proofs.}
 
	{\bf Step 1.}

Let $\Upsilon$ be a finite subset of $\mathcal{U}$.
For all natural $n$ and  nonnegative $t$ 
define the sets  
\begin{align*}
U_n(t)\rav \big\{\upsilon\in U\,\big|\,&
\|f(t,\hjk{y}(t),\upsilon)-f(t,\hjk{y}(t),\hjk{u}(t))\|\ \ \\
&\ +|f_0(t,\hjk{y}(t),\upsilon)|+L_{\upsilon}(t)+L_{\hjk{u}(t)}(t)\leq n
\big\},\\    
\Upsilon_n(t)\rav \big\{u(t)\,\big|\,&u\in\Upsilon,u(t)\in U_n(t)\big\}\cup\{\hjk{u}(t)\}    
\end{align*}
  for all nonnegative $t$.
 In particular,  $\hjk{u}(t)$ lies in $\Upsilon_n(t)$ whenever $t$ and $n$. 
 
Denote by $\mathfrak{Q}_\diamond$ the set of all pair $({\lambda},{\psi})\in [0;1]\times C(\mathbb{R}_+,\mathbb{X}^*)$ satisfying with $(\hjk{y},\hjk{u})$  the equality $\|\psi(0)\|+\lambda=1$, 
	adjoint equation \eqref{sys_psi},  condition  \eqref{400},  and the corresponding condition among  \eqref{WAKKa}--\eqref{WAKKd} (in the proof of Theorem~\ref{9}) or among  \eqref{AKKa}-\eqref{AKKc} (in the proof of Theorem~\ref{8}). One verifies easily that this set is closed. Then, $\mathfrak{Q}_\diamond$ is  compact.

Consider the the following assertion:
\begin{list}{}{}
	\item{$(A)$} 
	for each  finite subset
	$\Upsilon$ of $\mathcal{U}$ and each natural $n$
	there exist a nonzero pair $({\lambda}_{n},{\psi}_{n})\in \{0,1\} \times C([0;n],\mathbb{X}^*)$ 
	satisfying  
		\begin{itemize}
		\item 	the inequality
\begin{equation}\label{937}
	{H}\big(\hjk{y}(t),{\psi}_{n}(t),\upsilon,\lambda_{n},t\big)
	\!\leq\!
	{H}\big(\hjk{y}(t),{\psi}_{n}(t),\hjk{u}(t),\lambda_{n},t\big)
	\end{equation}
	for all $\upsilon\in \Upsilon_{n}(t)$ and almost all $t\in[0;n]$, 
		\item 
as well as 
	 adjoint equation \eqref{sys_psi} with $(\hjk{y},\hjk{u})$, 
	 		\item 
transversality condition  \eqref{400}, 
		\item 
and  the corresponding condition among  \eqref{WAKKa}--\eqref{WAKKd} (in the proof of Theorem~\ref{9}) or among  \eqref{AKKa}-\eqref{AKKc} (in the proof of Theorem~\ref{8}).
	\end{itemize}
\end{list}

In the next two steps
we will show that both  theorems are implied from  the assertion $(A)$.

	{\bf Step 2.}

Assume that assertion $(A)$ is established for all  finite subsets
$\Upsilon$ of $\mathcal{U}$ and all natural $n$.

Fix a such	$\Upsilon$.
Denote by $\mathfrak{Q}_\diamond^{\Upsilon}$ the set of all 
pair $({\lambda},{\psi})\in \mathfrak{Q}_\diamond$ satisfying  inequality \eqref{937} 
 for almost all nonnegative $t$ for each $\upsilon\in\Upsilon$.

 We claim that from $(A)$ it follows that $\mathfrak{Q}_\diamond^{\Upsilon}$ is nonempty and closed.

To prove this, applying the assertion $(A)$ for all natural $n$,  consider the corresponding sequence of $({\lambda}_{n},{\psi}_{n})$.

Scaling each $\psi_{n}$, if necessary, one can also suppose that ${\psi}_{n}(0)=n$ if $\lambda_{n}=0$. 
Then, passing from the sequence of~$n$ to its  subsequence, if necessary, one can assume that\ 
either  all $\lambda_{n_k}$ equal to $1$ and the sequence of $\psi_{n_k}(0)$ converges to a vector $\phi_0\in\mathbb{X}^*$,
or the sequence of $||\psi_{n_k}(0)||$ is unboundedly increasing and the sequence of $\psi_{n_k}(0)/||\psi_{n_k}(0)||$ 
converges to  a nonzero vector  $\phi_0\in\mathbb{X}^*$.
Put $\hjk{\lambda}\rav 1$ and  $r_k\rav 1$ if
 $\lambda_{n_k}\equiv 1$ and put $\hjk{\lambda}\rav 0$ and $r_k\rav 1/||\psi_{n_k}(0)||$ otherwise.
In both cases, the sequence of $(r_k\psi_{n_k}(0),r_k\lambda_{n_k})$ converges to $(\phi_0,\hjk{\lambda})$.

Now, each pair $(r_k\psi_{n_k}(0),r_k\lambda_{n_k})$ satisfies  \eqref{sys_psi}  and inherits all transversality conditions that are verified for $({\psi}_{n_k},{\lambda}_{n_k})$. Further,
by the theorem on continuous dependence of solutions to differential equations on the initial state, the sequence of~$r_k\psi_{n_k}$ converges on  $\mathbb{R}_+$ to the solution~$\hjk{\psi}$ to \eqref{sys_psi} with $\hjk{\psi}(0)=\phi_0$ and $\lambda=\hjk{\lambda}$. Note that
this convergence is uniform on an arbitrary compact interval.
Hence the pair $(\hjk{\psi},\hjk{\lambda})$ also  satisfies \eqref{sys_psi} on the whole $\mathbb{R}_+$. Moreover,   it follows that this pair satisfies inequality
 \eqref{937}  for almost all nonnegative $t$ for each $u\in\Upsilon$.

So,  there exists  a nonzero pair $(\hjk{\lambda},\hjk{\psi})\in \{0,1\} \times C(\mathbb{R}_+,\mathbb{X}^*)$ 
satisfying adjoint equation,  desired transversality conditions, and  inequality
\eqref{937} for almost all nonnegative $t$ for all $\upsilon\in\Upsilon$. Then, the pair $(\lambda_\diamond,\psi_\diamond)\rav\frac{1}{\hjk{\lambda}+\|\hjk{\psi}(0)\|}(\hjk{\lambda},\hjk{\psi})$ lies in $\mathfrak{Q}_\diamond^{\Upsilon}$.

For any converging sequence of $(\lambda^n_\diamond,\psi^n_\diamond)\in \mathfrak{Q}_\diamond^{\Upsilon}$
define $r_n=1$ if $\lambda^n_\diamond=0$ and $r_n=1/\lambda^n_\diamond$ otherwise.
Repeating the previous three paragraphs literally with the sequence of $({\lambda}_{n},{\psi}_{n})\rav r_n(\lambda^n_\diamond,\psi^n_\diamond)$, we obtain that the limit of $(\lambda^n_\diamond,\psi^n_\diamond)$ also  lies in $\mathfrak{Q}_\diamond^{\Upsilon}$.

Thus, $\mathfrak{Q}_\diamond^{\Upsilon}$ is a nonempty and closed subset of $\mathfrak{Q}_\diamond$   for each finite subset $\Upsilon$ of $\mathcal{U}$.

	{\bf Step 3.}

Now we 
show that the both  theorems are implied from $(A)$.

Indeed, for each finite subset $\Upsilon$ of $\mathcal{U}$, the nonempty set
 $\mathfrak{Q}_\diamond^{\Upsilon}$ is  closed in the compact $\mathfrak{Q}_\diamond$.
 By the finite intersection property, there exists a pair $(\lambda_\diamond,\psi_\diamond)$ lies in  $\mathfrak{Q}_\diamond^{\Upsilon}$ for all finite subsets $\Upsilon$.
 	
 We will prove that $(\lambda_\diamond,\psi_\diamond)$ satisfies \eqref{maxH}. 
 Suppose to the contrary  that one can find a subset  of positive measure in which \eqref{maxH} 
 fails.
 Then, for all $n$  large enough, consider the sets
  \[V_n(t)\rav \big\{\upsilon\in U\,\big|\,{H}\big(\hjk{y}(t),{\psi}_\diamond(t),\upsilon,\lambda_\diamond,t\big)
 -
 {H}\big(\hjk{y}(t),{\psi}_\diamond(t),\hjk{u}(t),\lambda_\diamond,t\big)>1/n\big\}\]
 are nonempty on a subset of positive measure.
  Increasing $n$ if necessary, we can guarantee that
  the set
  \[\check{U}_n(t)\rav U_n(t)\cap V_n(t)\]
  is nonempty on a subset $\mathcal{T}$ of positive measure. Since the graph of $\check{U}_n$ is measurable, Aumann's selection theorem \cite{aumann} yields the existence of a measurable function $\check{u}$ having values in $\check{U}_n(t)$ for almost all $t\in\mathcal{T}$. Define $\check{u}(t)=\hjk{u}(t)$ for all $t\notin \mathcal{T}$.
   Then, it follows that \eqref{937} with $\upsilon=\check{u}(t)$ fails
   for $t$ from the set of positive measure, therefore 
   $(\lambda_\diamond,\psi_\diamond)$ does not lies in  $\mathfrak{Q}_\diamond^{\Upsilon}$ if $\Upsilon=\{\check{u}\}$.
 We get the contradiction.
   	
   So,	 $(\lambda_\diamond,\psi_\diamond)$ satisfies \eqref{maxH}. Dividing $(\lambda_\diamond,\psi_\diamond)$ by 
   $\lambda_\diamond$ if $\lambda_\diamond\notin\{0,1\}$, we obtain the desired pair $(\hjk{\lambda},\hjk{\psi})$ with $\hjk{\lambda}\in\{0,1\}$.
   
Thus, it is required  to prove only assertion $(A)$ for each natural $n$ and each finite subset
$\Upsilon$ of $\mathcal{U}$.

	{\bf Step 4}.
	
Fix a finite subset $\Upsilon\subset\mathcal{U}$  and natural $n$ with the corresponding set-valued map $\Upsilon_n:\mathbb{R}\rightrightarrows U$.
For all positive $\alpha\leq n$ denote by $\mathcal{U}_\alpha$ the set of all Lebesgue measurable functions $u:\mathbb{R}_+\to U$ satisfying ${u}|_{[n;\infty)}=\hjk{u}|_{[n;\infty)}$,
 \[\meas \{\tau\in[0;n]\mid u(\tau)\neq\hjk{u}(\tau)\}\leq \alpha,\ \textrm{and}\ u(t)\in \Upsilon_n(t)\ \textrm{for almost all}\ t\in[0;n].\]
 Denote by $\mathbb{X}_\alpha$ the closed $\alpha$-ball centered in $\hjk{y}(0)$.
 At last, take $\beta(n)$ from \eqref{local} in the definition of the corresponding (locally weakly overtaking and locally overtaking) criterion. 

  Introduce $L_{\Upsilon}(t)\rav\sup_{\upsilon\in\Upsilon_n(t)} L_{\upsilon}(t)$ for all $t\in[0;n]$.	Define $M\rav e^{\int_0^n L_{\Upsilon}(\tau)\,d\tau}$. It is finite by $(H2)$ and $(H3)$.

 Consider the neighborhood of  $\Gr\hjk{y}$, the set $\mathbb{G}_\Upsilon=\cap_{\upsilon\in\Upsilon} \mathbb{G}_\upsilon\cap\hjk{\mathbb{G}}$ (see $(H2)$ and $(H3)$). Decreasing $\hjk{\mathbb{G}}$ if necessary, we may get that  $\mathbb{G}_\Upsilon\cap([0;n]\times\mathbb{X})$ is bounded.
 If $(H5)$ holds, decreasing again, we may get that the set $\{(t,x)\in\mathbb{G}_\Upsilon\mid t=0\}$ is contained in $\{0\}\times\mathbb{G}_{\exists}$. Similarly, if $(H6)$ is also fulfilled, this set  may be contained in $\{0\}\times\mathbb{G}_{\partial}$.
 
 Note that, by $(H2)$ and $(H3)$,   for all $u\in\mathcal{U}_n$ the maps $(t,x)\mapsto f(t,x,u(t))$ and $(t,x)\mapsto f_0(t,x,u(t))$ are Lipschitz continuous  on $\mathbb{G}_\Upsilon$ of rank  $L_{\Upsilon}(t)$. Applying the inequality  $\|f(t,x,u)\|\leq \|f(t,\hjk{y}(t),\hjk{u}(t))\|+\|f(t,\hjk{y}(t),u)-f(t,\hjk{y}(t),\hjk{u}(t))\|+L_{\Upsilon}(t)\|x-\hjk{y}(t)\|$, we obtain that the norms of these maps  on $\mathbb{G}_\Upsilon$ are also bounded by the locally summable function $n+\|f(t,\hjk{y}(t),\hjk{u}(t))\|+|f_0(t,\hjk{y}(t),\hjk{u}(t))|+RL_{\Upsilon}(t)$, here $R$ is the diameter of $\mathbb{G}_\Upsilon\cap([0;n]\times\mathbb{X})$. Now, for every $u\in\mathcal{U}_n$   the maps $\rty(x,t,\upsilon;\cdot)$ and $J(x,t,\upsilon;\cdot)$ are well-defined and continuous on an interval if $(t,x)\in{\mathbb{G}}_\upsilon$; furthermore, for every positive $\theta$ the map $x\mapsto J(x,0,\hjk{u};\theta)=\hjk{J}(x;\theta)$
     is continuous around $\hjk{y}(0)$ and, by $(H4)$, is strictly differentiable at this point.
     
     Further,
 decreasing the neighborhood $\mathbb{G}_\Upsilon$ if necessary, on can suppose that the set
 all the graphs of
$\rty(x,t,\hjk{u};\cdot)|_{[0;n]}$, $(t,x)\in \mathbb{G}_\Upsilon$, are containing in $\mathbb{G}_\Upsilon$, i.e., this set is strongly invariant with respect to $(t,x)\mapsto f(t,x,\hjk{u}(t)).$
In addition, there exists a positive $\gamma\leq \min({\beta}(n)/2,1)/2M^2n$ such that, the closed
$2Mn\gamma$-ball centered  in $(t,\hjk{y}(t))$ belongs to $\mathbb{G}_\Upsilon$  for all $t\in[0;n]$.
Now,  all graphs of
$\rty(x,0,\hjk{u};\cdot)|_{[0;n]}$ (with  $x\in \mathbb{X}_\gamma$) are contained in $\mathbb{G}_\Upsilon$.

  At last,  since ${f}$ and ${f}_0$ on ${G}_\Upsilon$ are Lipschitz continuous in $x$ of rank $L_{\Upsilon}(t)$, by definitions of  $\mathbb{G}_{\Upsilon}$, and $M$,  the inequality
	\begin{align}
	    \label{1907_}
	    \||x_1-x_2||\leq M \|\rty(x_1,t,\hjk{u};0)-\rty(x_2,t,\hjk{u};0)\|
	\end{align}
	holds for all $t\in[0;n]$, $(t,x_1),(t,x_2)\in\mathbb{G}_\Upsilon$.
	
   {\bf Step 5}.
	
	Define   the dynamics $\bar{f}_n:[0;2n]\times\mathbb{X}\times U\to\mathbb{X}$ and the integrand 
	$\bar{f}_{0,n}:[0;2n]\times\mathbb{X}\times U\to\mathbb{R}$
	as follows:
	for all $(t,x,u)\in[0;2n]\times{\mathbb{X}}\times U$,
	\begin{align*}
	(\bar{f}_{n},\bar{f}_{0,n})(t,x,u)=\left\{\begin{array}{cc}
	(f,f_0)(t,x,u)\ &\textrm{if}\ t<n;\\
	-(f,f_0)(2n-t,x,\hjk{u}(2n-t))\ &\textrm{if}\ t\geq n.
	\end{array}\right.
	\end{align*}

   Put 
   \[\widetilde{G}\rav\bigcup_{(t,x,s)\in\mathbb{G}_\Upsilon\times\mathbb{R}}\big\{ (t,x,s),(2n-t,x,s)\big\}.
	\]
	Note that since one has $\widetilde{G}\cap([0;n]\times\mathbb{X} \times\mathbb{R})\subset\mathbb{G}_\Upsilon$ and $(\bar{f}_{n},\bar{f}_{0,n})(2n-t,x,u(t))=-(\bar{f}_{n},\bar{f}_{0,n})(t,x,u(t))$, the function $(\bar{f}_{n},\bar{f}_{0,n})$ on $\widetilde{G}$ inherits the properties $f$ and $f_0$ on $\mathbb{G}_\Upsilon$. In particular,   the vector fields $(\bar{f}_{n},\bar{f}_{0,n})(t,\cdot,u)$, $u\in\Upsilon_n(t)$,
	 are locally Lipschitz continuous in $x$ on $\widetilde{G}$ and its norms are bounded by a summable function of time.
    
    Fix a control $u\in\mathcal{U}_\gamma$ and  $x\in\mathbb{X}$ with $(0,x)\in\mathbb{G}_\Upsilon$.
    Denote by $(\bar{\rty}_n,\bar{\rtw}_n)(x,u;\cdot)$ a  solution  to
    \begin{equation}\label{1789}
	\frac{d\bar{y}(\tau)}{d\tau}=\bar{f}_n(\tau,\bar{y}(\tau),u(\tau)),\ \frac{d\bar{w}(\tau)}{d\tau}=\bar{f}_{0,n}(\tau,\bar{y}(\tau),u(\tau))    
	\end{equation}
	with the initial condition $(\bar{y},\bar{w})(0)=(x,0)$
	on the  maximum existing interval  $I_{u,x}\subset[0;2n]$.
	Notice that this solution is unique if 
	\begin{equation}\label{1907}
	    \Gr (\bar{\rty}_n,\bar{\rtw}_n)(x,u;\cdot)|_{I_{u,x}}\subset \widetilde{G}.
	\end{equation}
	Now,  we will prove that $I_{u,x}=[0;2n]$, \eqref{local}, and \eqref{1907} hold  if either $u\equiv\hjk{u}$ or $x\in\mathbb{X}_\gamma$.

{\bf Step 6}.

	 In the case $u=\hjk{u}$, $x\in\mathbb{G}_\Upsilon$,
	by the construction of $\mathbb{G}_\Upsilon$ and the definition $\bar{f}$, the solution
	$(\bar{\rty}_n,\bar{\rtw}_n)(x,\hjk{u};\cdot)$ exists and this solution is unique on $[0;n]$.
	Since $(\bar{f}_{n},\bar{f}_{0,n})$ and $\widetilde{G}$ are symmetric in $t$, each solution
	$(\bar{y}_n,\bar{w}_n)|_{[0;n]}$
	to \eqref{1789}
	 has the unique continuation on $[0;2n]$ by the rule: 
	\[(\bar{y}_n,\bar{w}_n)(2n-t)=(\bar{y}_n,\bar{w}_n)(t)\quad \forall t\in[0;n].\]
In particular, the motion  \[\hjk{\bar{y}}_n(t)\rav\bar{\rty}\big(\hjk{y}(0),\hjk{u};t\big),\ 
	\hjk{\bar{w}}_n(t)\rav\bar{\rtw}\big(\hjk{y}(0),\hjk{u};t\big)\qquad \forall t\in[0;2n]\]
	satisfies $\hjk{\bar{y}}_n(2n)=\hjk{\bar{y}}_n(0)=\hjk{y}(0)$ and $\hjk{\bar{w}}_n(2n)=\hjk{\bar{w}}_n(0)=0.$

	Consider a $u\in\mathcal{U}_\gamma$,  $x\in\mathbb{X}_\gamma$ with its solution $(\bar{\rty}_n,\bar{\rtw}_n)(x,u;\cdot)$ on the maximal interval $I_{u,x}$. This solution exists and is unique at least until the exit  $\bar{\rty}_n(x,u;\cdot)$ from the bounded set $\widetilde{G}_\Upsilon\cap([0;n]\times\mathbb{R}^n)$.
 Note also that, since 
 $\mathbb{G}_\Upsilon$ is strongly invariant with respect to the generated by $\hjk{u}$ dynamics, for every $t\in[0;n]$ we get $(t,(\bar{\rty}_n,\bar{\rtw}_n)(x,u;t))\in\widetilde{G}$ in the case of $(0,\rty(\bar{\rty}_n(x,u;t),t,\hjk{u};0),0)\in\mathbb{G}_\Upsilon$. The following instance
$$\tau_{u,x}\rav\min\{t\in[0;n]\mid \|\rty(\bar{\rty}_n(x,u;t),t,\hjk{u};0)-\hjk{y}(0)\|=2nM\gamma\}\cup\{n\}$$
lies in $I_{u,x}$ because the closed
$2nM\gamma$-ball centered  in $\hjk{y}(t)$ belongs to $\mathbb{G}_\Upsilon$.
We claim that   $\tau_{u,x}=n$, $[0;n]\subset I_{u,x}$ hold true, and $y(\cdot)\rav\bar{\rty}_n(x,u;\cdot)|_{I_{u,x}}$ satisfies
\eqref{local}; in addition, the corresponding pair $(y(\cdot)|_{[0;n]},u(\cdot))$ is a control process on $[0;n]$.

Indeed,  note that   $(t,y(t))\in \mathbb{G}_\Upsilon$ and $\|\rty(y(t),t,\hjk{u};0)-\hjk{y}(0)\|\leq 2nM\gamma$ hold for all $t\in[0;\tau_{u,x}]$ by the definition of $\tau_{u,x}$. Now, by \eqref{1907_} and $x\in\mathbb{X}_\gamma$, the inequality 
$\|\rty(y(t),t,\hjk{u};t)-\hjk{y}(t)\|\leq M\|\rty(y(t),t,\hjk{u};0)-\hjk{y}(0)\|\leq 2nM^2\gamma<1$ also holds. From  the definition of $U_n(t)$, it follows
\begin{align*}
\Big\|
\frac{d}{dt}(y(\tau)-\rty(y(t),t,\hjk{u};\tau))\big|_{\tau=t}\Big\|=\|{f}(t,y(t),u(t))-{f}(t,{y}(t),\hjk{u}(t))\|\\
\leq (L_{u(t)}(t)+L_{\hjk{u}(t)}(t))\|y(t)-\hjk{y}(t)\|+
\|{f}(t,y(t),u(t))-{f}(t,\hjk{y}(t),u(t))\|\\
<(L_{u(t)}(t)+L_{\hjk{u}(t)}(t))2nM^2\gamma+
\|{f}(t,y(t),u(t))-{f}(t,\hjk{y}(t),u(t))\|
<n
\end{align*}
and $\frac{d}{dt}(\rty(y(t),t,\hjk{u};0)-\hjk{y}(0))\leq Mn$. Now, for almost every $t$, $\frac{d}{dt}\rty(y(t),t,{u};0)=0$ in the case  $u=\hjk{u}(t)$, therefore
the inclusions $u\in\mathcal{U}_\gamma$ and $x\in\mathbb{X}_\gamma$ entail
\begin{align*}
\|y(0)-\rty(y(t),t,\hjk{u};0)\|<
\int_{\{\tau\in[0;t]\,|\,\hjk{u}(\tau)\neq u(\tau)\in\Upsilon_n(\tau)\}} Mn\,dt\leq
Mn\gamma
\end{align*}
and $\|\rty(y(\tau_{u,x}),\tau_{u,x},\hjk{u};0)-\hjk{y}(0)\|\leq \gamma+Mn\gamma< 2nM\gamma$. By the definition it follows that $\tau_{u,x}=n$. Now, the graph of $y|_{I_{u,x}\cap[0;n]}$ is containing in   $\mathbb{G}_\Upsilon$; therefore, $[0;n]\subset I_{u,x}$. Furthermore,  we also obtain $\|\hjk{y}(t)-\rty(x,0,{u};t)\|\leq M\|\hjk{y}(0)-\rty(x,0,{u};0)\|\leq 2nM^2\gamma<{\beta}(n)/2$ for all $t\in[0;n]$. Hence, inequality \eqref{local} has been proved too. Thus, the pair $(y(\cdot)|_{[0;n]},u(\cdot))$ is a control process of the original problem
\eqref{sys0_}--\eqref{sysK_}.
   
	So, $[0;n]\subset I_{u,x}$ and  the point $(n,\bar{\rty}_n(x,u;n))$ lies in $\mathbb{G}_\Upsilon$.
Let $\check{y}$ be the  generated by optimal control $\hjk{u}$
	motion  of  \eqref{sys_x} satisfying $\bar{\rty}_n(x,u;n)=\check{y}(n)$, i.e., $\check{y}(\cdot)=\rty(\bar{\rty}_n(x,u;n),n,\hjk{u};\cdot)$.
	However, $(\bar{\rty}_n,\bar{\rtw}_n)(\check{y}(0),\hjk{u};\cdot)|_{[n;2n]}$ solves \eqref{1789} with 
	$(\bar{\rty}_n,\bar{\rtw}_n)(\check{y}(0),\hjk{u};n)=(\bar{\rty}_n,\bar{\rtw}_n)(x,u;n)$, therefore we obtain \eqref{1907} and $[0;2n]=I_{u,x}$, taking
	\begin{align}
	\label{11301}
	(\bar{\rty}_n,\bar{\rtw}_n)\big(x,u;2n\big)=
	\big(\check{y}(0),J(x,0,{u};n)-J\big(\check{y}(0),0,\hjk{u};n\big)\big),
	\end{align}
	here $\check{y}(\cdot)=\rty(\bar{\rty}_n(x,u;n),n,\hjk{u};\cdot)$.
	
	{\bf Step 7}
	
	Define  the function
	$\bar{H}_n:{\mathbb{X}}\times{\mathbb{R}}\times{\mathbb{R}}\times{\mathbb{R}}\times({\mathbb{X}}\times{\mathbb{R}}\times{\mathbb{R}}\times{\mathbb{R}})^*\times {U}\times[0;2n]\to{\mathbb{R}}$
	by the rule
	\[\bar{H}_n(r,w,x,s,\mu_r,{\mu}_w,\mu_y, \mu_s,u,t)\rav \mu_r1_{u\neq\hjk{u}(t)}+\mu_w\bar{f}_{0,n}\big(t,x,u\big)+\mu_y \bar{f}_n\big(t,x,u\big)\]
	for all
	$(r,w,x,s,\mu_r,{\mu}_w,\mu_y, \mu_s,u,t)\!\in\!{\mathbb{X}}\times{\mathbb{R}}\times{\mathbb{R}}\times{\mathbb{R}}\times({\mathbb{X}}\times{\mathbb{R}}\times{\mathbb{R}}\times{\mathbb{R}})^*\times {U}\times[0,2n].$
	Note that, due to $(H4)$,  this map is strictly  differentiable in $(r,w,x,s)$  if $x=\hjk{y}(t)$,   $u=\hjk{u}(t)$. In particular, one has
	\begin{subequations}
	\begin{align}   
	\frac{\partial \bar{H}_n}{\partial (r,w,s)}(r,w,x,s,{\mu}_r,{\mu}_w,\mu_y,{\mu}_s,u,t)&=(0,0,0)
	\label{2061}
	\end{align}
	for all $ t\in[0;2n]$ and 
	\begin{align}
	\frac{\partial \bar{H}_n}{\partial x}(r,w,x,s,0,{\mu}_w,\mu_y,0,\hjk{u}(t),t)&=-\frac{\partial \bar{H}_n}{\partial x}(r,w,x,s,0,{\mu}_w,\mu_y,0,u,2n-t),
	\label{2063}\\
	\frac{\partial H}{\partial x}(x,\mu_y,\hjk{u}(t),\lambda,t)&=\frac{\partial \bar{H}_n}{\partial x}(r,w,x,s,0,-\lambda,\mu_y,0,\hjk{u}(t),t),
	\label{2062}\\
	\label{2060}   
	H(x,\mu_y,u,\lambda,t)&= \bar{H}_n(r,w,x,s,0,-\lambda,\mu_y,0,u,t)
	\end{align}   
	\end{subequations}	
	for all $t\in[0;n)$ and $\lambda\in\mathbb{R}_+$
	whenever $(r,w,x,s,\mu_r,{\mu}_w,\mu_y,\mu_s,u).$

	 At last, set  $\hjk{\bar{r}}\rav0\in\mathbb{R}$ and $\hjk{\bar{s}}\rav0\in\mathbb{R}.$

	{\bf Step 8}.
	
	Note that 
	\begin{equation}\label{3150}
	W_{\theta}(\hjk{y}(0))=W_{\sup}(\hjk{y}(0))=W_{\inf}(\hjk{y}(0))=0\qquad \forall \theta>0. 
	\end{equation}
	Define  also the functions $\lscsup{\mathcal{\mathcal{C}_{\as}}} W$ and $\lscinf{\mathcal{\mathcal{C}_{\as}}} W$ by rule \eqref{infsupOmega}.
	Set 
		\begin{align*}
\lscextr{\mathcal{C}_{\as}} W\rav\newcorr{\lscsup{\mathcal{C}_{\as}} W}
	\textrm{\ in the proof of Theorem~\ref{9}},\\
	\lscextr{\mathcal{C}_{\as}} W\rav\newcorr{\lscinf{\mathcal{C}_{\as}} W}\textrm{\ in the proof of Theorem~\ref{8}.}
	\end{align*}
\newcorr{Notice that $\lscextr{\mathcal{C}_{\as}} W$ is lower semicontinuous by the definition;
   furthermore,  $(E'''_{\textrm{extr}})$ yields $(\lscextr{\mathcal{C}_{\as}} W)(\hjk{y}(0))=W_{\textrm{extr}}(\hjk{y}(0))=0$.}

	Consider a  $x\in\mathcal{C}_0\cap \mathbb{X}_\gamma$ and  $u\in \mathcal{U}_\gamma\subset\mathcal{U}_n$. By $y(\cdot)$ 
	denote the generated by control ${u}$ motion of \eqref{sys_x} with $x=y(0)$.
	Now, $(y|_{[0;n]},u)$ is a control process and satisfies
	 \eqref{local} by  $\gamma<{\beta}(T)/2$. 
	 Consider also  the generated by optimal control $\hjk{u}$
	motion $\check{y}$   of \eqref{sys_x} satisfying $y(n)=\check{y}(n)$. Now, $y(t)=\check{y}(t)$ for all $t\geq n$; in particular, $\check{y}(0)\in\mathcal{C}_{\as}$ iff $\Limsup_{\theta\uparrow\infty}\{\Lambda(\theta,y(\theta)\}\subset\mathcal{C}_{\infty}$ holds and $J(x,0,u;t)$ is finite for all positive $t$. Since,
	by \eqref{11301}, one has
	$\check{y}(0)=\bar{\rty}_n(y(0),{u};2n)$, we show that
	$(y,u)$ is an admissible control process iff
	$\bar{\rty}_n(y(0),{u};2n)$ 
	lies in $\mathcal{C}_{\as}$.

	Assume that $\bar{\rty}_n(y(0),{u};2n)\in\mathcal{C}_{\as}$.
	In this case, by definition of $J$, for the admissible control process $(y,u)$, we obtain 
	\begin{align*}
	J(y(0),0,u;\theta)\!-\!\hjk{J}(\hjk{y}(0);\theta)
	&\;=\;J(y(0),0,u;n)+J(\check{y}(n),n,\hjk{u};\theta)\!-\!\hjk{J}(\hjk{y}(0);\theta)\\
	&\;=\;J(x,0,u;n)\!-\!\hjk{J}(\check{y}(0);n)+\hjk{J}(\check{y}(0);\theta)\!-\!\hjk{J}(\hjk{y}(0);\theta)\\
	&\ravref{11301}\bar{\rtw}_n(x,{u};2n)+J\big(\bar{\rty}_n(x,{u};2n),0,\hjk{u};\theta\big)\!-\!\hjk{J}\big(\hjk{y}(0);\theta\big)
	\end{align*}
	for all $\theta>n$.  Passing to the corresponding limit as $\theta\uparrow\infty$, we have
	\begin{align*}
		l(\hjk{y}(0))
		\leq 
		l(y(0))\!+\!\bar{\rtw}_n(x,{u};2n)\!+\!W_{\sup}(\bar{\rty}_n(x,{u};2n))\textrm{\;if\;}(\hjk{y},\hjk{u})&\textrm{\;is\;weakly}\textrm{\;\;\;\;\;overtaking\;optimal,}\\
		l(\hjk{y}(0))
		\leq
		 l(y(0))\!+\!\bar{\rtw}_n(x,{u};2n)\!+\!W_{\inf}(\bar{\rty}_n(x,{u};2n))\textrm{\;if\;}(\hjk{y},\hjk{u})&\textrm{\;is\;overtaking\;optimal.}
	\end{align*}
	In the proof of Theorem~\ref{9}, as well as the proof of Theorem~\ref{8}, we obtain
	\begin{align*}
	l(\hjk{y}(0))\leq l(y(0))+\bar{\rtw}_n(y(0),{u};2n)+W_{\textrm{extr}}(\bar{\rty}_n(y(0),{u};2n))
	\end{align*}
	for all control processes $(y,u)$ satisfying  $u\in \mathcal{U}_\gamma$, $\bar{\rty}_n(y(0),{u};2n)\in\mathcal{C}_{\as}$, and  $y(0)\in\mathcal{C}_0\cap \mathbb{X}_\gamma$.
	Further, the infimum of the right side of this inequality is attained  at $(y,u)=\big(\hjk{y},\hjk{u}\big)$ and is equalled to $l(\hjk{y}(0))$.
		So, the point $\big(\hjk{y}(0),\hjk{w}(2n),\hjk{\bar{y}}(2n))=(\hjk{y}(0),0,\hjk{y}(0)\big)$ has to
	\[\textrm{ minimize\ } l(x_0)+w+W_{\textrm{extr}}(x_1)\ \textrm{subject to}\ (x_0,w,x_1)\in\mathcal{Q}_{n}, x_1\in\mathcal{C}_{\as}.\]
	Here,
	\[\mathcal{Q}_{n}\rav\big\{\big(x_0,\bar{\rtw}_n(x_0,u;2n),\bar{\rty}_n(x_0,u;2n)\big)\in{\mathbb{X}_\gamma\times\mathbb{R}\times\mathbb{X}\mid
	u\in\mathcal{U}_\gamma,\ x_0\in\mathcal{C}_0}\cap\mathbb{X}_\gamma\big\}.\]
	
	
	\newcorr{Besides $(\lscextr{\mathcal{C}_{\as}} W)(\hjk{y}(0))=W_{\textrm{extr}}(\hjk{y}(0))=0$ holds true, as noted above, and $\lscextr{\mathcal{C}_{\as}} W$ is a lower semicontinuous envelope of $W_{\textrm{extr}}+i_{{\mathcal{C}_{\as}}}$. This guarantees that the point
	$$\big(\hjk{y}(0),\hjk{w}(2n),\hjk{\bar{y}}(2n))=(\hjk{y}(0),0,\hjk{y}(0)\big)\in\mathcal{Q}_{n}$$ is a $\gamma'$-local	{\cite[Section~22.6]{ClarkeNew}} minimizer for}
	\[\newcorr{\textrm{minimize\ } l(x_0)+w+(\lscextr{\mathcal{C}_{\as}} W)(x_1)\ 
	\textrm{\ subject to\ }	(x_0,w,x_1)\in\mathcal{Q}_{n}
	}\]\newcorr{for a sufficiently small positive $\gamma'.$ Decreasing $\gamma$ if necessary, we can assume that $\gamma\leq \gamma'.$}

		By the definition of $\mathcal{Q}_{n}$, for  each control process 
	$(y,{u})$ with $u\in\mathcal{U}_\gamma$ and  $||y-\hjk{y}||_{C([0,n];\mathbb{X})}<\gamma$,  the process $\big(({\bar{\rty}}_n,{\bar{\rtw}}_n)(y(0),{u};\cdot);{u}(\cdot)\big)$ 
	satisfies $(y(0),({w},y)(2n))\in\mathcal{Q}_{n}$. Since  $s\geq (\lscextr{\mathcal{C}_{\as}} W)(x)$  for all $(x,s)\in\epi (\lscextr{\mathcal{C}_{\as}} W)$, 
	the process $\big(\hjk{\bar{y}}_n,\hjk{\bar{w}}_n,\hjk{{s}}=0,\hjk{u}\big)$ is a $\gamma$-local optimal solution 
	to the following problem
	\begin{align*}
	\textrm{minimize\ }&  l(\bar{y}(0))+\bar{w}(2n)+s(2n)\\
	\textrm{subject to\ }&\frac{d\bar{y}(t)}{dt}=\bar{f}_n(t,\bar{y}(t),u(t)),\ 
	\frac{d\bar{w}(t)}{dt}=\bar{f}_{0,n}(t,\bar{y}(t),u(t)),\ \frac{ds(t)}{dt}=0,
	\\
	&\bar{w}(0)=0,\ \bar{y}(0)\in\cl\mathcal{C}_0,\ (\bar{y},s)(2n)\in\epi (\lscextr{\mathcal{C}_{\as}} W),\ u\in\mathcal{U}_n, \\
	&\meas \{\tau\in[0;n]\mid u(\tau)\neq\hjk{u}(\tau)\}\leq \gamma.
	\end{align*}
	Notice that the map $t\mapsto\meas\{\tau\in[0;t]\mid u(\tau)\neq\hjk{u}(\tau)\}$
	is the solution to the equation
	$\frac{dr(t)}{dt}=1_{u(t)\neq\hjk{u}(t)}$ with the initial condition $r(0)=0$. Now,  
	the process $\big(\hjk{\bar{y}}_n,\hjk{\bar{w}}_n,\hjk{{s}}=0,\hjk{{r}}=0,\hjk{u}\big)$ is a $\gamma$-local
		minimizer 
	for 
	\begin{align*}
	\textrm{minimize\ }&  l(\bar{y}(0))+\bar{w}(2n)+s(2n)\\
	\textrm{subject to\ }&\frac{d\bar{y}(t)}{dt}=\bar{f}_n(t,\bar{y}(t),u(t)),\ 
	\frac{d\bar{w}(t)}{dt}=\bar{f}_{0,n}(t,\bar{y}(t),u(t)),\\
	&\frac{dr(t)}{dt}=1_{u(t)\neq\hjk{u}(t)},\ \frac{ds(t)}{dt}=0,\  
	\\
	&(r,\bar{w})(0)=0,\ \bar{y}(0)\in\mathcal{C}_0,\  (\bar{y},s)(2n)\in\epi (\lscextr{\mathcal{C}_{\as}} W),\ u\in\mathcal{U}_n.
	\end{align*}
	\newcorr{At last, since  $\lscextr{\mathcal{C}_{\as}} W$ is lower semicontinuous at $\hjk{y}(0)$, in this problem the set $\mathcal{C}_0$ can be replaced by $\cl\mathcal{C}_0$; so,  $\big(\hjk{\bar{y}}_n,\hjk{\bar{w}}_n,\hjk{{s}}=0,\hjk{{r}}=0,\hjk{u}\big)$ is a $\gamma$-local	minimizer for}
  \begin{align*}
	\newcorr{\textrm{minimize\ }}&  \newcorr{l(\bar{y}(0))+\bar{w}(2n)+s(2n)}\\
	\newcorr{\textrm{subject to\ }}&\newcorr{\frac{d\bar{y}(t)}{dt}=\bar{f}_n(t,\bar{y}(t),u(t)),\ 
	\frac{d\bar{w}(t)}{dt}=\bar{f}_{0,n}(t,\bar{y}(t),u(t)),}\\
	&\newcorr{\frac{dr(t)}{dt}=1_{u(t)\neq\hjk{u}(t)},\ \frac{ds(t)}{dt}=0,\ }  
	\\
	&\newcorr{(r,\bar{w})(0)=0,\ \bar{y}(0)\in\cl\mathcal{C}_0,\  (\bar{y},s)(2n)\in\epi (\lscextr{\mathcal{C}_{\as}} W),\ u\in\mathcal{U}_n.}
	\end{align*}
Note that $\bar{H}_n$ is the
Hamilton--Pontryagin function to this problem.
	
	{\bf Step 9}.

	Now we may apply the Pontryagin Maximum Principle \cite[Theorem~22.26]{ClarkeNew}, because its assumptions \cite[Hypothesis 22.25]{ClarkeNew} were  implied from $(H2)$ and $(H3)$ in Step 5. Furthermore, due to $(H4)$ we may look for co-state arc among solutions to the adjoint equation instead of the adjoint inclusion.

	Due to \cite[Theorem~22.26]{ClarkeNew}, there exist   a number $\lambda_n\in\{0,1\}$ and a co-state arc $(\mu_r,{\mu}_y,\mu_w, \mu_s)\in C([0;2n],(\mathbb{R}\times\mathbb{X}\times\mathbb{R}\times\mathbb{R})^*)$  with  
	\[(\mu_r(t),{\mu}_y(t),\mu_w(t),\mu_s(t),\lambda_n)\neq 0 \qquad \forall t\in[0;2n]\]
	such that the following transversality conditions hold
	\begin{align*}   
	&(\mu_r,{\mu}_w,\mu_y, \mu_s)(0)\!\in\!\{(0,0)\}\times\lambda_n \partial l(\hjk{\bar{y}}(0))\times\{0\}+\mathbb{R}\times\mathbb{R}\times N(\hjk{\bar{y}}(0);\cl\mathcal{C}_0)
	\times\{0\},\\
	-&(\mu_r,{\mu}_w,\mu_y, \mu_s)(2n)\!\in\!\lambda_n(0,1,0,1)+\{(0,0)\}\times N\big((\hjk{\bar{y}},\hjk{s})(2n);\epi(\lscextr{\mathcal{C}_{\as}} W)),
	\end{align*}
	and such that the co-state arc $(\mu_r,{\mu}_w,\mu_y,\mu_s)$ satisfies the adjoint equation: 
	\begin{multline}\label{inc}
	-\frac{d(\mu_r,\mu_w,\mu_y,\mu_s)}{dt}(t)=\frac{\partial\bar{H}_n}{\partial(r,w,x,s)}\big(\hjk{{r}}(t),\hjk{\bar{w}}_n(t),\hjk{\bar{y}}_n(t),\hjk{{s}}(t),\mu_r(t),\mu_w(t),\mu_y(t),\mu_s(t),\hjk{u}(t),t\big),
	\end{multline}
	as well as  the Pontryagin maximum condition on $[0;2n]$:
	\begin{equation}\label{mH}
	\hjk{u}(t)\in \arg\max_{\upsilon\in \Upsilon_n(t)}\bar{H}_n\big(\hjk{{r}}(t),\hjk{\bar{w}}_n(t),\hjk{\bar{y}}_n(t),\hjk{{s}}(t),\mu_r(t),\mu_w(t),\mu_y(t),\mu_s(t),\upsilon,t\big).
	\end{equation}
	
	Note that, by \eqref{inc} and \eqref{2061},
	functions $\mu_r$,
	$\mu_w$, and $\mu_s$ are constant, therefore, one has $\mu_s\equiv 0$ and $\mu_r\equiv 0$ by
	 the first  and second transversality conditions respectively. Further, by
	the symmetry of $(\bar{f}_{n},\bar{f}_{0,n})$ and $\bar{H}_n$
	in $t$ (see \eqref{2063}),
	the co-state arc $\mu_y|_{[0;n]}$ as a generated by $(\hjk{\bar{y}},\hjk{u})$ solution to \eqref{inc} satisfies $\psi_n(0)=\mu_y(0)=\mu_y(2n)$.

	Then, according to   $\psi_n(0)=\mu_y(0)=\mu_y(2n)$, $\hjk{\bar{y}}(0)=\hjk{\bar{y}}(2n)=\hjk{y}(0)$, and  $\hjk{r}=\hjk{s}=0$, we may rewrite the transversality conditions as follows:
	\begin{align*}   
	\mu_y(0)&\in\lambda_n \partial l(\hjk{y}(0))+ N(\hjk{y}(0);\cl\mathcal{C}_0),\\
	-({\mu}_w,\mu_y,0)(0)
	&\in
	\lambda_n(1,0,1)+\{0\}\times N\big(\hjk{y}(0),0;\epi(\lscextr{\mathcal{C}_{\as}} W))).
	\end{align*}
	Set $\psi_n=\mu_y|_{[0;n]}$, then
	the first transversality condition entails  \eqref{400}, 
	the second transversality condition leads to $\lambda_n=-\mu_w$ and
	\begin{equation}   
	-(\psi_n(0),\lambda_n)\in N\big(\hjk{y}(0),0;\epi(\lscextr{\mathcal{C}_{\as}} W)),\qquad\label{2982}
	\end{equation}
	at the same time \eqref{inc} yields \eqref{sys_psi}  on $[0;n]$ by \eqref{2062}.
	Further, for every $t\geq 0$, since $({\mu}_r(t),{\mu}_w(t),\mu_y(t),\mu_s(t),\lambda_n)=(0,-\lambda_n,\psi_n(t),0,\lambda_n)$ is nonzero, 
	the pair $(\psi_n(t),\lambda_n)$ is also nonzero.
	
	Note that $\hjk{\bar{y}}_n|_{[0;n]}= \hjk{y}|_{[0;n]}$ holds. Then, by \eqref{2060},
	\begin{align*}   
	\bar{H}_n\big(\hjk{{r}}(t),\hjk{\bar{w}}_n(t),\hjk{\bar{y}}_n(t),\hjk{{s}}(t),\mu_r(t),\mu_w(t),\mu_y(t),\mu_s(t),\upsilon,t\big)
	={H}\big(\hjk{y}(t),\psi_n(t),\upsilon,\lambda_n,t\big)
	\end{align*}
	holds
	for almost all $t\in[0;n]$ and all $\upsilon\in \Upsilon_n(t)$. Therefore, 
	the Pontryagin maximum condition \eqref{mH} on $[0;n]$ implies \eqref{937} with $\psi_n$ for all $\upsilon\in\Upsilon_n(t)$ and almost all $t\in[0;n]$.

	Thus, conditions \eqref{400} and \eqref{937} have also been proved
	and we need to  consider only one of transversality conditions among
	 \eqref{WAKKa}--\eqref{WAKKd} (in the proof of Theorem~\ref{9}) or \eqref{AKKa}-\eqref{AKKc} (in the proof of Theorem~\ref{8}).

	To prove Theorem~\ref{8}, note that under the corresponding assumption the required inclusion among \eqref{AKKa}--\eqref{AKKc} for \eqref{2982} holds true by one of the relations \eqref{a:0cap}--\eqref{c:0cap} in Lemma~\ref{cap}.
	Theorem~\ref{8} has been proved. 
	
	To prove Theorem~\ref{9}, note that, 	since $W_{\theta}(\hjk{y}(0))=W_{\sup}(\hjk{y}(0))=0$ for all positive $\theta$ by \eqref{3150},  one of the relations \eqref{a:0co}--\eqref{d:0co} in  Lemma~\ref{co} with the corresponding additional assumption
	entails the required  inclusion among \eqref{WAKKa}--\eqref{WAKKd} for \eqref{2982}.
	Theorem~\ref{9} has been proved.

		\begin{remark}\label{nonconvex}
		Note that in the proofs of Theoremae~\ref{9} and \ref{8}, under conditions $(H0)$--$(H4)$ and $(E_{\textrm{extr}})$, we deduce \eqref{2982}: 
		\begin{equation}   \label{2275}
	-(\hjk{\psi}(0),\hjk{\lambda})\in N\big(\hjk{y}(0),0;\epi\lsc\limsup_{\theta\uparrow\infty}\big(\hjk{J}(\cdot;\theta)-\hjk{J}(\hjk{y}(0);\theta)+\imath_{\mathcal{C}_{\as}}\big)\big)
	\end{equation}
	in the case of weakly overtaking optimal process $(\hjk{y},\hjk{u})$ as well as 
		\begin{equation}   \label{2277}
	-(\hjk{\psi}(0),\hjk{\lambda})\in N\big(\hjk{y}(0),0;\epi\lsc\liminf_{\theta\uparrow\infty}\big(\hjk{J}(\cdot;\theta)-\hjk{J}(\hjk{y}(0);\theta)+\imath_{\mathcal{C}_{\as}}\big)\big)
	\end{equation}
	in the case overtaking optimal process $(\hjk{y},\hjk{u})$. 
		 Conditions \eqref{2275} and \eqref{2277} are  very awkward for applications, but can be stronger than \eqref{WAKKa} and \eqref{AKKa} respectively.
	\end{remark}

	\section*{Some questions instead of concluding remarks}
	The key feature of this paper is the boundary conditions~\eqref{WAKKa} and \eqref{AKKa}, necessary for the weakly overtaking criterion and the overtaking criterion without any {\it a priori} assumptions on asymptotics. 
	In addition, in these necessary conditions the optimal control is not assumed to be bounded, the behavior of the dynamics and the integrand with respect to $t$ and $u$ is measurable, not necessarily continuous. The  condition obtained in Corollary~\ref{55}  cuts out the unique co-state arc under assumption~\eqref{4004};  however, one would like to have asymptotic assumptions that guarantee such uniqueness but are easier to test than~\eqref{4004}. Moreover, although  formula~\eqref{WAKKa}, as an analog of the Clarke subdifferential, matches in its form the necessary conditions of optimality 
	that are customary in the  variational analysis, it is only convenient to apply it to simple models (see~Example~\ref{ggg}). 
	So, this condition needs to be improved to the level of Proposition~\ref{done}. Besides, any simple assumptions making these transversality	conditions 	
	sufficient would be useful.

	\section*{Acknowledgements}
	I would like to express my gratitude
	to Boris Mordukhovich and Alexander Kruger for a valuable discussion in
	during the writing this article.


\begin{thebibliography}{10}
		
		
		\bibitem{ArVi}
		Arutyunov, A.V.,  Vinter, R.B.: A simple `finite approximations proof` of the
		{Pontryagin} {Maximum} {Principle} under reduced differentiability
		hypotheses. { Set-Valued Anal.} {\bf 12},  5--24 (2004).
		\url{https://doi.org/10.1023/B:SVAN.0000023406.16145.a8}
		
		\bibitem{kr_as2004}
		 Aseev, S.M.,   Kryazhimskii, A.V.: The {Pontryagin} {Maximum} {Principle} and
		transversality conditions for a class of optimal control problems with
		infinite time horizons. { SIAM J. Control Optim.} {\bf 43},  1094--1119 (2004).
		\url{https://doi.org/10.1137/S0363012903427518}
		
		\bibitem{kr_as}
		 Aseev, S.M.,  Kryazhimskii, A.V.: The {Pontryagin} {Maximum} {Principle} and problems of
		optimal economic growth. { Proc. Steklov Inst. Math.} {\bf 257},  1--255 (2007).
		\url{https://doi.org/10.1134/S0081543807020010}
		
		
		\bibitem{av_new}
		 Aseev, S., Veliov,  V.: Needle variations in infinite-horizon optimal control. In: 
		 Wolansky,~G., Zaslavski,~A., (eds.) Variational and optimal control problems on unbounded domains.
		pp. 1--17.
		American
		Mathematical Soc, Providence, (2014)
		
		
		\bibitem{aseev2019new}
		 Aseev, S.M.,   Veliov, V.M.: Another view of the maximum principle for infinite-horizon
		optimal control problems in economics. {Russ. Math. Surv.}
		{\bf 74},  963--1011 (2019).
		\url{https://doi.org/10.1070/rm9915}
		
		
		\bibitem{aucl}
		Aubin, J.  Clarke, F.: Shadow prices and duality for a class of optimal control
		problems. {SIAM J. Control Optim.} {\bf 17},  567--586 (1979).
		\url{https://doi.org/10.1137/0317040}
		
		\bibitem{aumann}
		Aumann R.J.: Measurable utility and measurable choice theorem. Proc. Colloque Int.
CNRS ``La D\'ecision'' (Aix-en-Provence, 1967), pp. 15-26, CNRS, Paris, 1969.
		
		\bibitem{bald}
		 Balder, E.J.: An existence result for optimal economic growth problems. {J. Math.
			Anal.} {\bf 95},  195--213 (1983).
		\url{https://doi.org/10.1016/0022-247X(83)90143-9}
		
			\bibitem{bch}
		Beltratti, A., Chichilnisky, G.,  Heal, G.: Sustainable growth and the green golden rule. In:
		 Goldin, I., Winters, L.A.(eds)
		NBER Series, Vol. 4430,
		pp. 147--166.
		Cambridge University Press (1993). \url{https://doi.org/10.3386/w4430}
		
		
		\bibitem{belyakov2018}
		Belyakov, A.O.: Necessary conditions for infinite horizon optimal control problems
		revisited. Preprint  arXiv:1512.01206v2 (2015)
		
		
		
		\bibitem{belyakov2020}
		 Belyakov, A.O.: On sufficient optimality conditions for infinite horizon optimal
		control problems. {Proc. Steklov Inst. Math.} {\bf 308},  56--66 (2020).
		\url{https://doi.org/10.1134/S0081543820010058}
		
		\bibitem{besov}
		 Besov, K.O.: On {B}alder's existence theorem for infinite-horizon optimal control
		problems. {Math. Notes}  {\bf 103},  167--174 (2018).
		\url{https://doi.org/10.1134/S0001434618010182}
		
		\bibitem{slovak}
		  Bogusz, D.: On the existence of a classical optimal solution and of an almost
		strongly optimal solution for an infinite-horizon control problem. {J. Optim.
			Theory Appl.} {\bf 156},  650--682 (2013).
		\url{https://doi.org/10.1007/s10957-012-0126-2}
		
		
		\bibitem{BorZhu} Borwein, J.M.,  Zhu, Q.J.: Techniques of variational analysis.	Springer, Berlin (2004).	
		
		\bibitem{brodskii}
		 Brodskii, Y.I.: Necessary conditions for a weak extremum in optimal control
		problems on an infinite time interval. {Math. Sb.}
		{\bf 34},  327--343 (1978).
		\url{https://doi.org/10.1070/SM1978v034n03ABEH001208}
		
		
		
		
		\bibitem{cannarsa2018value}
		 Cannarsa, P., Frankowska,  H.: Value function, relaxation, and transversality
		conditions in infinite horizon optimal control. {J. Math Anal.} {\bf 457},
		  1188--1217 (2018). 
		  \url{https://doi.org/10.1016/j.jmaa.2017.02.009}
		
		
		\bibitem{car1}
		 Carlson, D.A.: Uniformly overtaking and weakly overtaking optimal solutions in
		infinite--horizon optimal control: when optimal solutions are agreeable. {J.
			Optim. Theory Appl.} {\bf 64},  55--69 (1990).
		\url{https://doi.org/10.1007/BF00940022}
		
		
		 
		\bibitem{ClarkeNew}
		 Clarke, F.: Functional analysis, calculus of variations and optimal control.  Springer (2013).
		
		
		
		\bibitem{evol}
		 Degiovanni, M.,  Marino, A., Tosques, M.: Evolution equations with lack of convexity.
		{Nonlinear Analysis: Theory, Methods \& Applications}
		{\bf 9},  1401--1443 (1985).
        \url{https://doi.org/10.1016/0362-546X(85)90098-7}		
		
		\bibitem{dm}
		 Dmitruk A.V., Kuz'kina, N.V.: An existence theorem in an optimal control problem on
		an infinite time interval. {Math. Notes}  {\bf 78},  466--480 (2005).
        \url{https://doi.org/10.1007/s11006-005-0147-3}		
		
		\bibitem{slope}
		 Geoffroy, M., Lassonde, M.: Stability of slopes and subdifferentials. {Set-Valued
			Anal.} {\bf 11}, (2003) 257--271. \url{https://doi.org/10.1023/A:1024406403469}
		
		
		\bibitem{feichtinger}
		 Grass, D.,  Caulkins, J.P.,  Feichtinger, G., Tragler, G.: Optimal control of nonlinear
		processes.  Berlin, Springer  (2008).
		
		 \bibitem{GBG}
		 Gromov, D., Bondarev, A., Gromova, E.:
	On periodic solution to control problem with time-driven switching.
	Optim. Lett. (2021). \url{https://doi.org/10.1007/s11590-021-01749-6}
	
	\bibitem{GBG2022}
	Gromov, D.,  Shigoka, T., Bondarev, A.:
	Optimality and sustainability of hybrid limit cycles in the pollution control problem with regime shifts.
	Preprint  arXiv:2207.12486 (2022).
	
		\bibitem{Halkin}
		 Halkin, H.: Necessary conditions for optimal control problems with infinite
		horizons. {Econometrica} {\bf 42}, (1974) 267--272.
		\url{https://doi.org/10.2307/1911976}
		
		\bibitem{hartman}
		Hartman, P.: Ordinary differential equations.  John Wiley \& Sons, New York (1964)
		
		\bibitem{kami}
		 Kamihigashi, T.: Necessity of transversality conditions for infinite horizon
		problems. {Econometrica} {\bf 69},  995--1012 (2001).
        \url{https://doi.org/10.1111/1468-0262.00227}		
		
		\bibitem{JDCS}
		 Khlopin, D.: Necessity of vanishing shadow price in infinite horizon control
		problems. {J. Dyn. Con. Sys.} {\bf 19},  519--552 (2013).
		\url{https://doi.org/10.1007/s10883-013-9192-5}
		
		\bibitem{optim}
		Khlopin, D.: Necessity of limiting co-state arc in {Bolza-type} infinite horizon
		problem. {Optimization} {\bf 64},   2417--2440 (2015).
		\url{https://doi.org/10.1080/02331934.2014.971413}
		
		\bibitem{khlopin2015lipschitz}
		 Khlopin, D.: On lipschitz continuity of value functions for infinite horizon
		problem. {Pure Appl. Funct. Anal.} {\bf 2},  535--552 (2017)
		
		
		\bibitem{KhlopinIMM2018}
		Khlopin, D.: On necessary limit gradients in control problems with infinite
		horizon. {Trudy Instituta Matematiki i Mekhaniki UrO RAN}
		{\bf 24},  247--256 (2018) (In Russian).
		\url{ https://doi.org/10.21538/0134-4889-2018-24-1-247-256}
		
		
		\bibitem{KhlopinIFAC}
		Khlopin, D.: A maximum principle for one infinite horizon impulsive control
		problem. {IFAC-PapersOnLine} {\bf 51},  213--218 (2018).
		\url{https://doi.org/10.1016/j.ifacol.2018.11.383}
		
		
		
		
		\bibitem{trieman}
		 Ledyaev, Y.S.,  Treiman, J.S.: Sub-and supergradients of envelopes, semicontinuous
		closures, and limits of sequences of functions. { \it Russ. Math. Surv.}
		{\bf 67},  345--373 (2012).
		\url{https://doi.org/10.1070/rm2012v067n02abeh004789}
		
		
		\bibitem{Mordukh1}
		Mordukhovich, B.S.: Variational analysis and generalized differentiation {I}:
		Basic theory. Springer (2006)
		
		\bibitem{mord5}
		Mordukhovich, B.S.: Variational analysis and applications. Springer (2018)
		
		\bibitem{mordukhovich2013subdifferentials}
		Mordukhovich, B.S., Nghia, T.: Subdifferentials of nonconvex supremum functions and
		their applications to semi-infinite and infinite programs with lipschitzian
		data. {SIAM J. Control Optim.} {\bf 23},  406--431 (2013).
		\url{https://doi.org/10.1137/110857738}
		
		
		\bibitem{Pereira}
		Pereira, F.  Silva, G.: Necessary conditions of optimality for state constrained
		infinite horizon differential inclusions. In 50th IEEE Conference on
		Decision and Control and European Control Conference (CDC-ECC). IEEE 
		6717--6722 (2011)
		
		
		\bibitem{perez2019subdifferential}
		P{\'e}rez-Aros, P.:  Subdifferential formulae for the supremum of an arbitrary
		family of functions. {SIAM J. Control Optim.} {\bf 29},  1714--1743 (2019)
		\url{https://doi.org/10.1137/17M1163141}
		
		\bibitem{ppp}
		 Pontryagin, L.S.,  Boltyanskij, V.G.,  Gamkrelidze, R.V.,  Mishchenko, E.F.: The mathematical theory
		of optimal processes.  Fizmatgiz, Moscow  (1961) (In Russian)
		
		
		\bibitem{RW}
		 Rockafellar, R.T., Wets,  R.J.B.: Variational analysis.
		  Springer (2009)
		
		\bibitem{sagara}
		 Sagara, N.: Value functions and transversality conditions for infinite-horizon
		optimal control problems. { Set-Valued Var. Anal.} {\bf 18},  1--28 (2010).
		\url{https://doi.org/10.1007/s11228-009-0132-1}
		
		
		\bibitem{norv}
		 Seierstad, A.: Necessary conditions for nonsmooth, infinite-horizon optimal
		control problems. {J. Optim. Theory Appl.} {\bf 103},  201--230 (1999).
		\url{https://doi.org/10.1023/A:1021733719020}
		
		\bibitem{ssbook}
		 Seierstad, A.,   {Syds{\ae}ter}, K.: Optimal control theory with economic
		applications.  North-Holland, Amsterdam (1987)
		
		\bibitem{shell}
		Shell, K.: Applications of {Pontryagin's} {Maximum} {Principle} to economics.  Vol.~11 of Lect. Notes Oper. Res. Math.
		Econ.(
		Mathematical systems theory and economics, 1),  pp. 241--292, Springer,  Berlin (1969)
		
		
		\bibitem{Shvartsman}
		 Shvartsman, I.A.: New approximation method in the proof of the maximum principle
		for nonsmooth optimal control problems with state constraints. {J. Math. Anal.}
		{\bf 326},  974--1000 (2007).
		\url{https://doi.org/10.1016/j.jmaa.2006.03.056}
		
		
		
		\bibitem{smirn}
		 Smirnov, G.V.: Transversality condition for infinite--horizon problems. {J. Optim.
			Theory Appl.} {\bf 88},  671--688 (1996).
		\url{https://doi.org/10.1007/BF02192204}
		
		
		\bibitem{sorger}
		 Sorger, G.: Competitive dynamic advertising: A modification of the case game. 
		{J. Econ. Dyn. Con.} {\bf 13},  55--80 (1989).
		\url{https://doi.org/10.1016/0165-1889(89)90011-0}
		
		\bibitem{stern}
		Stern, L.E.: Criteria of optimality in the infinite-time optimal control problem.
		{{J. Optim. Theory Appl.}}  {\bf 44(3)},  497--508 (1984).
		\url{https://doi.org/10.1007/BF00935464}
		
		\bibitem{tan_rugh}
		 Tan, H.,  Rugh, W.J.: Nonlinear overtaking optimal control: sufficiency, stability,
		and approximation. {IEEE Transactions on Automatic Control}
		{\bf 43},  1703--1718 (1988).
		 \url{https://doi.org/10.1109/9.736068}
		
		
		\bibitem{Tauchnitz}
		 Tauchnitz, N.: The {Pontryagin} {Maximum} {Principle} for nonlinear optimal
		control problems with infinite horizon. {J. Optim. Theory Appl.}
		{\bf 167},  27--48 (2015).
		\url{https://doi.org/10.1007/s10957-015-0723-y}
		
		\bibitem{Tauchnitz2020}
		Tauchnitz, N.:  Pontryagin's Maximum Principle for Infinite Horizon Optimal Control Problems with Bounded Processes and with State Constraints. Preprint arXiv:2007.09692 (2020)
		
		
		
		\bibitem{vinter}
		Vinter, R.: Optimal control.  Birkh\"{a}user, Boston (2000)
	\end{thebibliography}
\end{document}